\documentclass[12pt,reqno]{amsart}
\usepackage[margin=1in]{geometry}
\usepackage{amsmath,amssymb,amsthm,graphicx,amsxtra, setspace}
\usepackage[utf8]{inputenc}
\usepackage{mathrsfs}
\usepackage{hyperref}
\usepackage{upgreek}
\usepackage{mathtools}
\usepackage{xcolor}
\allowdisplaybreaks

\usepackage[pagewise]{lineno}

\usepackage{graphicx,eurosym}
\usepackage{hyperref}
\usepackage{mathtools}

\usepackage[cyr]{aeguill}

\colorlet{darkblue}{blue!50!black}

\hypersetup{
	colorlinks,%
	citecolor=blue,%
	filecolor=red,%
	linkcolor=darkblue,%
	urlcolor=blue,%
	pdfnewwindow=true,%
	pdfstartview={FitH}
}

\colorlet{darkblue}{red!100!black}

\newtheorem{theorem}{Theorem}[section]
\newtheorem{lemma}[theorem]{Lemma}
\newtheorem{proposition}[theorem]{Proposition}

\newtheorem{corollary}[theorem]{Corollary}
\newtheorem{definition}[theorem]{Definition}

\newtheorem{remark}[theorem]{Remark}
\newtheorem{hypothesis}[theorem]{Hypothesis}

\let\originalleft\left
\let\originalright\right
\renewcommand{\left}{\mathopen{}\mathclose\bgroup\originalleft}
\renewcommand{\right}{\aftergroup\egroup\originalright}

\theoremstyle{definition}

\def\1{\mathcal{O}}
\def\t{r\wedge\tau_N^m}
\def\T{T\wedge\tau_m^N}
\def\tt{t\wedge\tau_m^N}
\def\td{(t+\delta)\wedge\tau_m^N}
\def\wi{\widehat}
\def\vi{\widetilde}

\def\d{\mathrm{d}}

\def\I{\mathrm{I}}
\def\B{\mathrm{B}}
\def\D{\mathrm{D}}
\def\A{\mathrm{A}}

\def\W{\mathrm{W}}
\def\R{\mathbb{R}}
\def\E{\mathbb{E}}
\def\Q{\mathrm{Q}}
\def\H{\mathbb{H}}
\def\V{\mathbb{V}}
\def\e{\epsilon}
\def\2{\mathcal{E}}
\def\L{\mathrm{L}}
\def\u{\boldsymbol{u}}
\def\v{\boldsymbol{v}}
\def\w{\boldsymbol{w}}

\def\Y{\mathbf{Y}} 

\def\P{\mathbb{P}}
\def\N{\mathbb{N}}
\def\x{\boldsymbol{x}}

\def\bfX{\mathbf{X}}
\def\Z{\mathrm{Z}}
\def\PP{\mathrm{P}}
\def\U{\mathbb{U}}

\def\vphi{\varphi}

\newcommand{\Addresses}{{
		\footnote{
			\noindent \textsuperscript{1,2}Department of Mathematics, Indian Institute of Technology Roorkee-IIT Roorkee,
			Haridwar Highway, Roorkee, Uttarakhand 247667, INDIA.\par\nopagebreak
			\noindent  \textit{e-mail:} \texttt{Manil T. Mohan: maniltmohan@ma.iitr.ac.in, maniltmohan@gmail.com.}
			
			\textit{e-mail:} \texttt{Ankit Kumar: akumar14@mt.iitr.ac.in.}
			
			\noindent \textsuperscript{*}Corresponding author.
			
			\textit{Key words:} Stochastic partial differential equations, locally monotne, pseudo-monotone, variationnal solutions, L\'evy noise. 
			
			Mathematics Subject Classification (2020): Primary 60H15; Secondary 35R60, 35Q35, 37L55.

}}}
\begin{document}	
	
		\title[Well-posedness of a class of SPDE with L\'evy noise]{Well-posedness of a class of stochastic partial differential equations with fully monotone coefficients perturbed by L\'evy noise
			\Addresses}

	\author[A. Kumar and M. T. Mohan]
	{Ankit Kumar\textsuperscript{1} and Manil T. Mohan\textsuperscript{2*}}

	\maketitle

\begin{abstract}
	In this article, we consider the following class of stochastic partial differential equations (SPDE): 
		\begin{equation*}
		\left\{
		\begin{aligned}
			\d \mathbf{X}(t)&=\mathrm{A}(t,\mathbf{X}(t))\d t+\mathrm{B}(t,\mathbf{X}(t))\mathrm{d}\mathrm{W}(t)+\int_{\mathrm{Z}}\gamma(t,\mathbf{X}(t-),z)\widetilde{\pi}(\mathrm{d} t,\mathrm{d} z),\; t\in[0,T],\\
			\mathbf{X}(0)&=\boldsymbol{x} \in \mathbb{H},
		\end{aligned}
		\right.\end{equation*}
	with fully locally monotone coefficients in a Gelfand triplet $\mathbb{V}\subset \mathbb{H}\subset\mathbb{V}^*$, where the mappings
	\begin{align*}
		\mathrm{A}:[0,T]\times \mathbb{V}\to\mathbb{V}^*,\quad \mathrm{B}:[0,T]\times \mathbb{V}\to\mathrm{L}_2(\mathbb{U},\mathbb{H}), \quad \gamma:[0,T]\times\mathbb{V}\times\mathrm{Z}\to\mathbb{H},
	\end{align*}are measurable, $\mathrm{L}_2(\mathbb{U},\mathbb{H})$ is the space of all Hilbert-Schmidt operators from $\mathbb{U}\to\mathbb{H}$, $\mathrm{W}$ is a $\mathbb{U}$-cylindrical Wiener process and $\widetilde{\pi}$ is a compensated time homogeneous Poisson random measure. Such kind of SPDE cover a large class of quasilinear SPDE and  a good number of fluid dynamic models.  Under certain generic assumptions of $\mathrm{A},\mathrm{B}$ and $\gamma$, using the classical Faedo-Galekin technique, a compactness method and  a version of Skorokhod's representation theorem,  we prove the existence of a \emph{probabilistic weak solution} as well as \emph{pathwise uniqueness of solution}. We use the  classical Yamada-Watanabe theorem to obtain the existence of a  \emph{unique probabilistic strong solution}. Furthermore, we establish a result on the continuous dependence of the solutions on the initial data. Finally, we  allow both diffusion coefficient $\mathrm{B}(t,\cdot)$ and jump noise coefficient $\gamma(t,\cdot,z)$  to depend on both $\mathbb{H}$-norm  and $\mathbb{V}$-norm, which implies that both the coefficients could also depend on the gradient of solution. Under some assumptions on the growth coefficient corresponding to the $\mathbb{V}$-norm, we establish the global solvability results also. 
	
\end{abstract}
	\section{Introduction}\label{Introduction}\setcounter{equation}{0}
	Let  $\H$ and $\V$ be a separable Hilbert space and a reflexive Banach space,  respectively, such that the embedding $\V\subset\H$ is compact. Let $\V^*$ be the dual space of $\V$ and $\H^*$($\cong\H$) be the dual space of $\H$. The norms of $\H,\V$ and $\V^*$ are denoted by $\|\cdot\|_\H,\|\cdot\|_\V $ and $\|\cdot\|_{\V^*},$ respectively, and we have  a  Gelfand triplet $	\V\subset \H\subset \V^*.$ Let us represent $(\cdot,\cdot)$ for the inner product in $\H$ and  $\langle \cdot,\cdot\rangle,$ the duality pairing between $\V^*$ and $\V$. Also, if $\u\in\H$ and $\v\in\V,$ then we have $\langle\u,\v\rangle=	(\u,\v). $ Let $\U$ be an  another separable Hilbert space, and $\L_2(\U,\H)$ be the space of all Hilbert-Schmidt operators from $\U$ to $\H$ with the norm $\|\cdot\|_{\L^2}$ and inner product $(\cdot,\cdot)_{\L^2}$.  
	
Let $(\Omega,\mathscr{F},\P)$ be a complete probability space equipped with an increasing family of sub-sigma fields $\{\mathscr{F}_t\}_{t\geq0}$ of $\mathscr{F}$ satisfying:
\begin{itemize}
	\item[(i)] $\mathscr{F}_0$ contains all elements $A \in\mathscr{F}$ with $\P(A)=0$.
	\item[(ii)] $\mathscr{F}_t=\mathscr{F}_{t+}=\cap_{s>t}\mathscr{F}_s$, for $0\leq t\leq T$.
\end{itemize}  	Let $\W(\cdot)$ be a cylindrical Wiener process  (\cite{DaZ}) on  $\U$ defined on the filtered probability space $(\Omega,\mathscr{F},\{\mathscr{F}_t\}_{t\geq 0},\P)$.  Let  $(\Z,\mathcal{I})$ be a measurable space and $\vi{\pi}$ is a compensated time homogeneous Poisson random measure.  

Recently, the authors in \cite{MRSSTZ} provided the well-posedness results for a class of 	stochastic partial differential equations (SPDE) with fully locally monotone coefficients  driven by multiplicative Gaussian noise. Motivated from their work, we consider the following class of SPDE with fully monotone coefficients in the Gelfand triplet $\V\subset \H \subset \V^*$ perturbed by L\'evy noise:
	\begin{equation}\label{1.1}
		\left\{
		\begin{aligned}
			\d \bfX(t)&=\A(t,\bfX(t))\d t+\B(t,\bfX(t))\d\W(t)+\int_{\Z}\gamma(t,\bfX(t-),z)\vi{\pi}(\d t,\d z),\; t\in[0,T],\\
			\bfX(0)&=\x \in \H.
		\end{aligned}
		\right.\end{equation}
	The maps appearing in the system \eqref{1.1} are measurable and defined as:
	\begin{align*}
		\A:[0,T]\times \V\to\V^*,\quad \B:[0,T]\times \V\to\L_2(\U,\H), \quad \gamma:[0,T]\times\V\times\mathrm{Z}\to\H.
	\end{align*}
		In this work, we are  concerned about the existence and uniqueness results of the variational solutions of the system \eqref{1.1} in the Gelfand triplet. 	The method we are using here is a variational approach and the  monotone operators theory plays an important role.  The theory of monotone operators was initiated by the author in \cite{GJM1} and then studied by several authors for instance one can see \cite{VB1,VB2,JLL,EZ} etc., and references therein.
	
	The author in \cite{EP1,EP2} was the first one who extended the monotone theory from deterministic  to stochastic setting. Later, several authors worked on this theory and used it for the early applications of variational approaches to SPDE, cf. \cite{IG,JRMR,MRFYW,MRXZ} etc., and references therein. In the literature, a large number of works are available which deal with the well-posedness of SPDE under variational framework, we are citing some of them here. Well-posedness of an SPDE driven by multiplicative Gaussian noise  with generalized coercivity conditions or Lyapunov conditions has been discussed in \cite{WLMR2,NDS}, etc., and references therein.  The existence and uniqueness of analytically    strong solutions of a class of SPDE of gradient type is obtained in \cite{BG}. Several authors have contributed in this direction and studied other properties of solutions such as the regularity \cite{VBMR}, asymptotic behavior (especially the large deviation principle)  \cite{PLC,ICAM1}, etc., and references therein. 
	
	A good number of papers is available in the literature  which deals with SPDE driven by L\'evy noise and we are citing some of them.  Well-posedness of SPDE  perturbed by L\'evy noise  having generalized coercivity conditions, and locally monotone coefficients   is established in \cite{ZBWLJZ},  monotone coefficients is obtained in \cite{TKMR}, and  of   the Ladyzenskaya-Smagorinsky type is established in \cite{PNKTRT}. Existence and uniqueness of the probabilistic strong (analytic weak) solutions for 2D stochastic Navier-Stokes equations driven by L\'evy noise on the torus and bounded or unbounded domains is established in  \cite{ZDYX1,ZBEH}, respectively. The existence of martingale solutions for 3D Navier-Stokes equations subjected to L\'evy noise is obtained in \cite{EM,EM1,KSSSS} etc.   The local solvability of stochastic Navier-Stokes equations in $\R^m,\;m\geq 2$, perturbed by L\'evy noise in $\L^p$-spaces for $p\in[m,\infty)$ with initial data in $\L^m(\R^m)$ has been discussed in \cite{MTMSSS1}. The  existence of strong and  martingale solutions  of different types of SPDE driven by L\'evy noise have been discussed in \cite{ZBDG,JCPNRT,ZDJZ,MTM2,MTM4,MTMSSS2,EM,EM1,ZTHWYW,JXJZ}, etc. and references therein.
	
	The classical framework of the monotone operator theory for SPDE with coefficients satisfying the local monotoncity conditions was extended in \cite{WLMR1} in the following way:  For any $\u,\v\in\V$, assume 
	\begin{align*}
		\langle \A(t,\u)-\A(t,\v),\u-\v\rangle \leq \big[C+\rho(\u)+\eta(\v)\big]\|\u-\v\|_\H^2,
	\end{align*}where $\rho(\u)$ or $\eta(\v)$ are locally bounded functions on the space $\V$. They have imposed a  necessary condition  that only one of them  can be non-zero, that is, either $\rho(\u)\equiv0$ or $\eta(\v)\equiv0$. Even if it is a restrictive assumption,   this framework covered several interesting examples such as 2D stochastic Navier-Stokes equations, 2D magneto-hydrodynamic equations, and several other hydro-dynamic models. Later, the authors in \cite{WL4} discussed a more general type of locally monotone conditions, where both $\rho(\u)$ and $\eta(\v)$ are non-zero. The author in \cite{WL4} introduced pseudo-monotone operators to the  variational approach but  it was limited to SPDE driven by additive noise only.  Several improvements in this direction are still going on and authors in \cite{WLRMJLDS1,WLRMJLDS2} discussed time fractional SPDE driven by additive noise with fully monotone coefficients. The case of additive noise can be tackled in a fairly easy way as one can transform the SPDE into a PDE with random coefficients, and then use the results available for deterministic equation with pseudo-monotone coefficients. The biggest drawback of this approach is that it cannot be extend in case of general  multiplicative noise. Thus, the well-posedness of SPDE  with fully monotone coefficients perturbed by multiplicative noise was an open problem for almost one decade, which is also pointed out  by several authors in their works \cite{ZBWLJZ,WL4,WLMR2,MRSSTZ}.  In a recent article \cite{MRSSTZ}, the authors established the well-posedness of SPDE perturbed by multiplicative Gaussian noise with fully monotone coefficients. Our aim  is to extend this work to L\'evy noise so that one can handle a large number of SPDE perturbed by L\'evy noise.

We analyze the well-posedness  of the SPDE \eqref{1.1} with fully local monotone coefficients driven by L\'evy noise in this paper. The major aims of this work are as follows:
\begin{itemize}
	\item The first objective is to establish the well-posedness of the SPDE  \eqref{1.1}  with the set of assumptions on the coefficients given in Hypothesis \ref{hypo1} below. We borrow the  ideas from \cite{WLMR2,PNKTRT,MRSSTZ,JZZBWL}, etc., to achieve this goal. In the first part, we assume that both coefficients $\B(t,\cdot)$ and $\gamma(t,\cdot,\cdot)$ are continuous on the  Hilbert space $\H$ and establish  the uniform estimates for the approximating solution (Faedo-Galerkin approximation). Later, we prove the tightness of the approximating solution in the space $\mathcal{Y}:=\D([0,T];\V^*)\cap\L^\beta(0,T;\H)$. Then, we recall some results which help us to obtain the existence of a \emph{probabilistically weak solution} as well as  \emph{pathwise uniqueness} for solutions. Therefore, by the classical Yamada-Watanabe theorem (see Theorem 8, \cite{HZ}) we obtain  the existence of a \emph{unique probabilistic strong solution}. Moreover, we establish the continuous dependency of the solutions on the initial data.
	\item In the second part of this article, we modify our assumptions on the coefficients (see Hypothesis \ref{hypo2} below), that is, we allow both the coefficients $\B(t,\cdot)$ and $\gamma(t,\cdot,\cdot)$ can depend on $\V$-norm, which means that $\B$ and $\gamma$ depend on the gradient  of the solution ($\nabla\bfX$).  We derive  the uniform estimates for the approximating solution (Faedo-Galekin approximation)  and the tightness property in the space $\L^\beta(0,T;\H)$ follows from the first part. Later, we prove the existence of a \emph{probabilistically weak solution}, under certain assumptions on the growth constant corresponding to the $\V$-norm. Finally,  pathwise uniqueness for solutions remains valid  as in the first part, and then Yamada-Watanabe theorem  ensures the existence of a \emph{unique probabilistic strong solution} and the continuous dependency of solutions on the initial data follows on the similar lines as in the first part.
\end{itemize}

Let us now provide a brief description of the current work. As discussed earlier, we prove the well-posedness under two different assumptions: Hypotheses \ref{hypo1} and \ref{hypo2}. In the first  two sections \ref{sec2}  and \ref{sec3},  we choose both the diffusion coefficient $\B$ and jump noise coefficient $\gamma$ are continuous on the Hilbert space $\H$.  We start section \ref{sec2}, by providing the basics of time homogeneous Poisson random measure, and then we state  our main Theorems \ref{thrm1} and \ref{thrm2}. In section \ref{sec4}, we choose both coefficients $\B$ and $\gamma$ to be dependent  on the gradient  of the solution $\bfX$ to the system \eqref{1.1}.

We start section \ref{sec3} with the  uniform energy estimates (Lemma \ref{lem1}) for the approximated solution $\{\Y_m\}$ of the system   \eqref{1.1} using Galerkin approximation. Then, we move to the proof of tightness (Lemma \ref{lem2} and Proposition \ref{proptightness}) of the laws of $\{\Y_m\}$  in the space $\mathcal{Y}$. In next step, we apply Prokhorov's theorem (Lemma \ref{lemA.4}) and a version of Skorokhod's representation theorem (Theorem \ref{thrmA.5})  to construct an another probability space, which ensures the almost sure convergence of the sequence  $\{\Y_m\}$ to some element $\Y$ in the space $\mathcal{Y}$   along some subsequence. Later, we prove the strong convergence  of $\B(\cdot,\Y_m(\cdot))$, $\gamma(\cdot,\Y_m(\cdot),\cdot)$ to $\B(\cdot,\Y_m(\cdot))$,   $\gamma(\cdot,\Y(\cdot),\cdot),$ respectively (Lemmas \ref{lem4} and \ref{lemma}). Moreover, we use the pseudo-monotonicity of the operator $\A(\cdot,\Y(\cdot))$ to obtain  the existence of a \emph{probabilistically weak solution}  $\Y$ (Theorem \ref{thrm1}). By the classical Yamada-Watanabe theorem (see Theorem 8, \cite{HZ}), we ensure the existence of a \emph{unique probabilistic strong solution} using the existence of a \emph{probabilistically weak solution} and pathwise uniqueness for the solutions (Theorem \ref{thrm4}).  We wind up section \ref{sec3}  by giving a proof Theorem \ref{thrm2}, that is, the continuous dependency of the solutions on the initial data.

In section \ref{sec4}, we revise our assumptions (Hypothesis \ref{hypo2}) and the steps in first part, like a modification in the uniform estimates (Lemma \ref{lem7}). We obtain the tightness of the family $\{\mathscr{L}(\Y_m):m\in\N\}$ in the space $\L^\beta(0,T;\H)$ but not in the space $\D([0,T];\V^*)$. To identify the limit $\Y$ of the sequence $\{\Y_m\}$ of the SPDE \eqref{1.1}, we again apply Prokhorov's theorem (Lemma \ref{lemA.4}) and a version of Skorokhod's representation theorem (Theorem \ref{thrmA.5}) such that $\{\Y_m\}$ converges almost surely to the limit $\Y$ in $\L^\beta(0,T;\H)$ on the newly constructed  probability space. Under certain assumptions on the growth constant corresponding to the $\V$-norm (see \eqref{412} below),  we prove  the existence of a \emph{probabilistically weak solution} (Lemma \ref{lem8}) using   monotonicity technique, and then we provide  the proof of  Theorem \ref{thrm5}. The rest part, that is, pathwise uniqueness of the solutions and continuous dependency of the solutions on the initial data follow from Theorems \ref{thrm4} and \ref{thrm2}, respectively. 

Some useful results like tightness of laws (Lemma \ref{lemA.3} and \ref{lemRoc}), Prokhorov's theorem (Lemma \ref{lemA.4}) as well as a   version of Skorokhod's representation theorem (Theorem \ref{thrmA.5}) are recalled in Appendix \ref{sec5}.

\section{Well-posedness Results}\label{sec2}\setcounter{equation}{0}
In this section, we state  the well-posedness results of  the system \eqref{1.1}. Before going to that let us provide a brief introduction to time homogeneous Poisson random measure and provide the necessary assumptions on the operators $\A$, $\B$ and $\gamma$. 



%

\subsection{Time homogeneous Poisson random measure} For the basics of time homogeneous Poisson random measure, we follow the works \cite{DAP,ZBEH,ZBEHPAR1,NISW,EM,SPJZ,JZZBWL} etc. For any topological space $\mathbb{Y},$ let $\mathcal{B}(\mathbb{Y})$ denotes it Borel-$\sigma$ algebra. We denote the set of natural numbers by $\N$, $\bar{\N}:=\N\cup \{\infty\}, \R_+:=[0,\infty)$. Let  $(\mathcal{S},\mathscr{I})$ be a measurable space  and $\mathtt{M}_{\bar{\N}}(\mathcal{S})$ denote the set of all $\bar{\N}$-valued measures on the space $(\mathcal{S},\mathscr{I})$. On the set  $\mathtt{M}_{\bar{\N}}(\mathcal{S})$, we consider the sigma-field $\mathcal{M}_{\bar\N}(\mathcal{S})$ as a smallest sigma-field such that for all $ B\in \mathcal{I}$, the map
\begin{align*}
	i_B:\mathtt{M}_{\bar{\N}}(\mathcal{S}) \ni \nu \mapsto \nu(B)  \in \bar{\N}
\end{align*}is measurable.
\begin{definition}
	Let us consider a measurable space $(\Z,\mathcal{I})$. A {\bf time homogeneous Poisson random measure} $\pi$ on the space $(\Z,\mathcal{I})$ over $(\Omega,\mathscr{F},\{\mathscr{F}_t\}_{t\geq 0},\P)$ is a measurable function defined by 
	\begin{align*}
		\pi:(\Omega,\mathscr{F})\to  (\mathtt{M}_{\bar\N}(\R_+\times \Z),\mathcal{M}_{\bar\N}(\R_+\times \Z))
	\end{align*}such that 
\begin{enumerate}
	\item for all $B\in\mathcal{B}(\R_+)\otimes \mathcal{I},\pi(B):\Omega\to\bar{\N}$ is a Poisson random measure with parameter $\E[\pi(B)]$;
	\item the measure $\pi$ is independently scattered, that is, if the sets $B_i\in   \mathcal{B}(\R_+)\otimes \mathcal{I}, i=1,\ldots,m$ are disjoint, then the random variables $\pi(B_i),i=1,\ldots,m$, are independent;
	\item for all $K\in \mathcal{I}$, the $\bar{\N}$-valued process $\{N(t,K)\}_{t\geq0}$ defined by 
	\begin{align*}
		N(t,K):=\pi((0,t]\times K),\;t\geq0
	\end{align*} is $\mathscr{F}_t$-adapted and its increments are independent from the history, that is, if $0\leq s<t$, then $N(t,K)-N(s,K)=\pi((s,t]\times K)$ is independent of $\mathscr{F}_s$.
\end{enumerate}
\end{definition}
Let us assume $\pi$ be a  time homogeneous Poisson random measure, then the formula 
\begin{align*}
	\lambda(S):=\E[\pi((0,1]\times S)],\; S\in\mathcal{I}
\end{align*}defines a measure on the space $(\Z,\mathcal{I})$ called an intensity measure of $\lambda$. Moreover, for all $T<\infty$ and all $S\in\mathcal{I}$ such that $\E[\pi((0,1]\times S)]<\infty$, the $\R$-valued process $\{\vi{N}(t,S)\}_{t\in(0,T]}$ defined by 
\begin{align*}
	\vi{N}(t,S):=\pi((0,1]\times S)-t\lambda(S),\;t\in(0,T]
\end{align*}is an integrable martingale on $(\Omega,\mathscr{F},\{\mathscr{F}_t\}_{t\geq 0},\P)$ . The compensator measure of $\pi$ is defined as a random measure $\d\otimes\lambda$ on $\mathcal{B}(\R_+)\times\mathcal{I}$, where $\d$ denotes the Lebesgue measure. The difference between a time homogeneous Poisson random measure $\pi$ and its compensator, written as 
\begin{align*}
	\vi{\pi}:=\pi-\d \otimes\lambda,
\end{align*}is known as a \emph{compensated time homogeneous Poisson random measure}.

Next, we discuss some basic properties of the stochastic integral with respect to $\vi{\pi}$ (cf.  \cite{DAP,ZBEHPAR1,NISW,SPJZ} etc.). Let $\H$ represent a separable Hilbert space and the predictable $\sigma$-filed on $[0,T]\times\Omega$ is denoted by $\mathcal{P}$. Let us denote the space of $\H$-valued, $\mathcal{P}\otimes\mathcal{I}$-measurable processes  with 
\begin{align*}
	\E\bigg[\int_0^T\int_\Z\|\zeta(t,\cdot,z)\|_\H^2\lambda(\d z)\d t\bigg]<\infty,
\end{align*}
by $\mathfrak{L}^2_{\lambda,T}(\mathcal{P}\otimes\mathcal{I}, \d\otimes\P\otimes\lambda;\H)$.
If $\zeta \in\mathfrak{L}^2_{\lambda,T}(\mathcal{P}\otimes\mathcal{I}, \d\otimes\P\otimes\lambda;\H)$, then the process $\int_0^t\int_\Z\zeta(s,\cdot,z)\vi{\pi}(\d s,\d z)$, is a c\'adl\'ag square integrable martingale and the following It\^o's isometry holds:
\begin{align*}
	\E\bigg[\bigg\|\int_0^T\int_\Z\zeta(t,\cdot,z)\vi{\pi}(\d t,\d z)\bigg\|_\H^2\bigg]=\E\bigg[\int_0^T\int_\Z\|\zeta(t,\cdot,z)\|_\H^2\lambda(\d z)\d t\bigg].
\end{align*}Since the integral $\mathrm{M}(t):=\int_0^T\int_\Z\zeta(t,\cdot,z)\vi{\pi}(\d t,\d z)$ is an $\H$-valued square integrable martingale, there exist  increasing c\'adl\'ag processes so-called quadratic variation process $[\mathrm{M}]_t$ and Meyer process $\langle \mathrm{M}\rangle_t$ such that $ [\mathrm{M}]_t-\langle \mathrm{M}\rangle_t$ is a local martingale (for more details see section 1.6, \cite{HK} and section 2.2, \cite{MM}). For the process, $\mathrm{M}(\cdot)$, one can verify that $[\mathrm{M}]_t=\int_0^t\int_\Z\|\zeta(s,\cdot,z)\|_\H^2\pi(\d s,\d z)$ and $\langle\mathrm{M}\rangle_t=\int_0^t\int_\Z\|\zeta(s,\cdot,z)\|_\H^2\lambda(\d  z)\d s$ (see Example 2.8, \cite{UMMTMSSS}). Indeed, $\E\{\|\mathrm{M}(t)\|_\H^2\}=\E\{[\mathrm{M}]_t\}=\E\{\langle\mathrm{M}\rangle_t\}$, so that we obtain 
\begin{align}\label{21}
	\E\bigg[\int_0^t\int_\Z\|\zeta(s,\cdot,z)\|_\H^2\pi(\d s,\d z)\bigg]=\E\bigg[\int_0^t\int_Z\|\zeta(s,\cdot,z)\|_\H^2\lambda(\d z)\d s\bigg], \; \text{ for all } t\in[0,T].
\end{align}

\subsubsection{The space $\mathrm{D}([0,T];\mathbb{S})$} Let $(\mathbb{S},d_1)$ denote a complete separable metric space (Polish space). The space of functions $\Y: [0,T] \to \mathbb{S}$ that are right-continuous on $[0,T]$ and have left-limits at every point in $(0,T]$ is denoted by $\mathrm{D}([0,T];\mathbb{S})$. The space $\mathrm{D}([0,T];\mathbb{S})$ is endowed with the Skorokhod topology, which makes this space separable and metrizable by a complete metric (for more details, see  Chapter 3, \cite{PB} or Chapter 2, \cite{MM}).

\subsection{Assumptions} We start with  some basic definitions related to the operators. 
\begin{definition}\label{def1}
	An operator $\A$ from $\V$ to $\V^*$ is said to be {\bf pseudo-monotone} if the following condition holds:  for any sequence $\{\u_m\}$ with the weak limit $\u$ in $\V$ and 
	\begin{align}\label{3.2}
		\liminf_{m\to\infty}		\langle \A(\u_m),\u_m-\u\rangle\geq 0,
	\end{align}imply that 
	\begin{align}\label{3.3}
		\limsup_{m\to\infty}	\langle \A(\u_m),\u_m-\v\rangle \leq \langle \A(\u),\u-\v \rangle, \ \text{ for all } \ \v\in\V.
	\end{align}
\end{definition}
\begin{remark}\label{rem1}
	We know that if an operator $\A:\V\to \V^*$ is a bounded operator, that is, $\A$ takes every bounded bounded set of $\V$ to a bounded set of $\V^*$, then the pseudo-monotonicity of the operator $\A$ is equivalent to the following condition:
	
	For any sequence $\{\u_m\}$ with the weak limit $\u$ in $\V$ and 
	\begin{align*}
		\liminf_{m\to\infty}		\langle \A(\u_m),\u_m-\u\rangle\geq 0,
	\end{align*}imply that $\{\A(\u_m)\}$ converges to $\A(\u)$ in weak-star topology of $\V^*$ and
	\begin{align}\label{3.4}
		\lim_{m\to\infty} \langle \A(\u_m),\u_m\rangle =\langle \A(\u),\u\rangle.
	\end{align}The interested readers are refereed to Proposition 27.7 in \cite{EZ} and Remark 5.2.12 in \cite{WLMR2}.
\end{remark}

Now, we introduce assumptions on the coefficients $\A$, $\B$ and $\gamma$.
\begin{hypothesis}\label{hypo1} Let  $f\in\L^1(0,T;\R_+)$ and $\beta\in (1,\infty)$.
	\begin{itemize}
	\item[(H.1)] (Hemi continuity). The map $\R\ni\lambda \mapsto \langle \A(t,\u+\lambda \v),\w\rangle \in\R$ is continuous for any $\u,\v,\w\in \V$ and for a.e. $t\in[0,T]$.
	\item[(H.2)] (Local monotonicity).  There exist  non negative constants $\zeta$ and $C$ such that for any $\u,\v\in\V$ and a.e. $t\in[0,T]$, 
	\begin{align}\label{3.5}\nonumber
2\langle \A(t,\u)-\A(t,\v),\u-\v\rangle +&\|\B(t,\u)-\B(t,\v)\|_{\L_2}^2+\int_{\Z}\|\gamma(t,\u,z)-\gamma(t,\v,z)\|_{\H}^2\lambda(\d z) 		\\& \leq \big[f(t)+\rho(\u)+\eta(\v)\big]\|\u-\v\|_{\H}^2, \\ \nonumber
|\rho(\u)|+|\eta(\u)| &\leq C(1+\|\u\|_{\V}^\beta)(1+\|\u\|_{\H}^\zeta), 
	\end{align}where $\rho$ and $\eta$ are two measurable functions from $\V$ to $\R$.
\item[(H.2)$'$] (General local monotonicity). For any constant $r>0$, there exists a function $M_{\cdot}(r)\in\L^1(0,T;\R_+)$  such that for any $\|\u\|_\V\vee \| \v\|_\V\leq r$ and a.e. $t\in[0,T]$,
\begin{align}\label{3.6}
	\langle \A(t,\u)-\A(t,\v),\u-\v\rangle \leq M_t(r)\|\u-\v\|_{\H}^2.
\end{align}	
	\item[(H.3)]	(Coercivity). There exists a positive constant $C$ such that for any $\u\in\V$ and a.e.  $t\in[0,T]$, 
	\begin{align}\label{3.7}
		2\langle \A(t,\u),\u\rangle 
		\leq f(t)(1+\|\u\|_{\H}^2)-C\|\u\|_{\V}^\beta.
	\end{align}
\item[(H.4)] (Growth). There exist non-negative constants  $\alpha$ and $C$ such that for any $\u\in\V$ and a.e. $t\in[0,T]$,
\begin{align}\label{3.8}
	\|\A(t,\u)\|_{\V^*}^{\frac{\beta}{\beta-1}}\leq (f(t)+C\|\u\|_{\V}^\beta)(1+\|\u\|_{\H}^\alpha).
\end{align} 
\item[(H.5)] For any sequence $\{\u_m\}_{m=1}^\infty$ and $\u$ in $\V$ with $\|\u_m-\u\|_{\H}\to 0$ as $m\to \infty$, we have 
\begin{align}\label{3.9}
	\|\B(t,\u_m)-\B(t,\u)\|_{\L_2} \to 0,\ \text{ for a.e. } \ t\in[0,T].
\end{align}Moreover, there exists $g\in \L^1(0,T;\R_+)$ such that for any $\u\in \V$ and a.e. $t\in[0,T]$, 
\begin{align}\label{3.10}
	\|\B(t,\u)\|_{\L_2}^2\leq g(t)(1+\|\u\|_{\H}^2). 
\end{align}
\item[(H.6)]  The jump noise coefficient $\gamma(\cdot,\cdot,\cdot)$ satisfy: \begin{enumerate}
	\item The function $\gamma \in\mathfrak{L}^2_{\lambda,T}(\mathcal{P}\otimes\mathcal{I}, \d\otimes\P\otimes\lambda;\H)$.
	\item  For any sequence $\{\u_m\}_{m=1}^\infty$ and $\u$ in $\V$ with $\|\u_m-\u\|_{\H}\to 0$ as $m\to \infty$, we have 
\begin{align*}
	\int_\Z\|\gamma(t,\u_m,z)-\gamma(t,\u,z)\|_\H^2\lambda(\d z) \to 0, \ \text{for a.e. } \ t\in[0,T].
\end{align*}\item(Growth). There exists functions $h_p\in \L^1(0,T;\R_+)$ such that for any $\u\in \V$, a.e. $t\in[0,T]$,  and all $p\in[2,\infty)$, 
\begin{align}\label{3.11}
	\int_\Z \|\gamma(t,\u,z)\|_{\H}^p\lambda(\d z) \leq h_p(t)(1+\|\u\|_{\H}^p).
\end{align}
\end{enumerate}
	\end{itemize}
\end{hypothesis}
\begin{remark}\label{rem2}
	\begin{enumerate}
	\item One can see that the Hypothesis (H.2)$'$ is weaker than (H.2) and the existence results can be derived with (H.2)$'$  whereas in the pathwise uniqueness we use (H.2).
	\item In applications, the coefficients $\B(\cdot,\u)$ and $\gamma(\cdot,\u,\cdot)$ are usually assumed to be locally Lipschitz and of linear growth. So, the hypothesis (H.5), (H.6) are satisfied. Later, we discuss the case where $\|\B(\cdot,\u)\|_{\L_2}$ and $\|\gamma(\cdot,\u,\cdot)\|_\H$ depend on $\|\u\|_\V$ also. 
\end{enumerate}
\end{remark}

\begin{remark}\label{rem4}
	If we compare the local monotone condition used in \cite{WLMR2}, the main difference is that in Hypothesis (H.2) both measurable functions $\rho$ and $\eta$ can be non-zero. In fact, in \cite{WLMR2}, it is required that the sum $\rho(\u)+\eta(\v)$ either only depends on $\u$ or on $\v,$ when the system \eqref{1.1}   is perturbed by multiplicative noise. This condition was very crucial in \cite{WLMR2} (cf. Remark 2.9, \cite{MRFYW}).
\end{remark}

\subsection{Well-posedness results}
Let us now provide the definition of probabilistically weak and strong solutions for the system \eqref{1.1}:
\begin{definition}[Probabilistically weak solution]\label{def2}
	A {\bf probabilistically weak solution} to the equation \eqref{1.1} is a system $((\wi{\Omega},\wi{\mathscr{F}},\{\wi{\mathscr{F}}_t\}_{t\geq0},\wi{\P}),\wi{\bfX},\wi{\W},\wi{\pi})$, where 
	\begin{enumerate}
		\item $(\wi{\Omega},\wi{\mathscr{F}},\{\wi{\mathscr{F}}_t\}_{t\geq0},\wi{\P})$ is a filtered probability space with the filteration $\{\wi{\mathscr{F}}_t\}_{t\geq0}$,
		\item $\wi{\W}$ is a $\U$-cylindrical Wiener process on $(\wi{\Omega},\wi{\mathscr{F}},\{\wi{\mathscr{F}}_t\}_{t\geq0},\wi{\P})$,
		\item $\wi{\pi}$ is a time homogeneous Poisson random measure on the space $(\Z,\mathcal{I})$ over $(\wi{\Omega},\wi{\mathscr{F}},\{\wi{\mathscr{F}}_t\}_{t\geq0},\wi{\P})$ with intensity measure $\lambda$,
		\item $\wi{\bfX}:[0,T]\times \wi{\Omega} \to \H$ is a predictable process with $\wi{\P}$-a.s., paths 
		\begin{align*}
			\wi{\bfX}(\cdot,\omega)\in \mathrm{D}([0,T];\H)\cap \L^\beta(0,T;\V),
		\end{align*}such that for all $t\in[0,T]$ and  for all $ \v\in\V$, the following holds:
		\begin{align}\label{3.12}\nonumber
		(\wi{\bfX}(t),\v)&=(\x,\v) +\int_0^t\langle \A(s,\wi{\bfX}(s)),\v\rangle \d s+ \int_0^t(\B(s,\wi{\bfX}(s))\d\wi{\W}(s),\v)\\&\quad+\int_0^t\int_\Z(\gamma(s,\wi{\bfX}(s-),z),\v)\vi{\wi{\pi}}(\d s,\d z),\ \wi{\P}\text{-a.s.}
		\end{align}
	\end{enumerate}
\end{definition}
Let us provide the definition of pathwise strong probabilistic (analytically weak) solution for the system \eqref{1.1}.
\begin{definition}[Strong probabilistic solution]
	We are given a stochastic basis \\ $((\Omega,\mathscr{F},\{\mathscr{F}_t\}_{t\geq0},\P),\bfX,\W,\pi)$ and  $\x\in\H$. Then, \eqref{1.1} has a {\bf pathwise strong probabilistic solution} if and only if there exists a progressively measurable process $\bfX:[0,T]\times \Omega\to\H$ with $\P$-a.s., paths \begin{align*}
		\bfX(\cdot,\omega) \in \mathrm{D}([0,T];\H)\cap \L^\beta(0,T;\V),
	\end{align*}and the following equation 
\begin{align*}
		(\bfX(t),\v)&=(\x,\v) +\int_0^t\langle \A(s,\bfX(s)),\v\rangle \d s+ \int_0^t(\B(s,\bfX(s))\d\W(s),\v)\\&\quad+\int_0^t\int_\Z(\gamma(s,\bfX(s-),z),\v)\vi{\pi}(\d s,\d z), \; \text{ for all } \v\in\V,
\end{align*}holds $\P$-a.s., for all $t\in[0,T]$.
\end{definition}
\begin{definition}[Pathwise uniqueness]
For $i=1,2$, let us consider $\bfX_i$ be any solution on the stochastic basis $((\Omega,\mathscr{F},\mathscr{F}_{t\geq0},\P),\bfX_i,\W,\pi)$  to the system \eqref{1.1} with $\bfX_i(0)=\x$. Then, the solutions of the system \eqref{1.1} are pathwise unique if and only if  
	\begin{align*}
		\P\big\{\bfX_1(t)=\bfX_2(t),\text{ for all } t\geq 0\big\}=1.
	\end{align*}
\end{definition}
\begin{definition}[Uniqueness in law]
	We say that solutions of the system \eqref{1.1} are unique in law if and only if the following holds:
	 If $((\Omega_i,\mathscr{F}_t,\{\mathscr{F}_i\}_{t\geq0},\P_i),\bfX_i,\W_i,\pi_i)$ for $i=1,2$ are two  solutions to the system \eqref{1.1} with $\bfX_i(0)=\x$, for $i=1,2$, then $\mathscr{L}_{\P_1}(\bfX_1)=\mathscr{L}_{\P_2}(\bfX_2)$.
\end{definition}

The well-known Yamada-Watanabe theorem (cf. \cite{MRBSXZ,HZ}) states that weak existence and pathwise uniqueness imply strong existence and weak uniqueness.

Let us now state  the main results on the existence of  probabilistically weak  and  strong solution of the system \eqref{1.1}. Proofs are provided in  the subsequent sections. 
\begin{theorem}\label{thrm1}
	Assume that the embedding $\V\subset \H$ is compact and Hypothesis \ref{hypo1} (H.1) and (H.2)$'$-(H.6) hold. Then, for any initial data $\x\in\H$, there exists a {\bf probabilistic weak solution} to the system \eqref{1.1}. Furthermore, for any $p\geq 2,$ the following estimate holds:
	\begin{align}\label{3.13}
		\E\bigg[\sup_{t\in[0,T]}\|\bfX(t)\|_{\H}^p\bigg] +\E\bigg[\bigg(\int_{0}^{T}\|\bfX(t)\|_{\V}^\beta\d t\bigg)^{\frac{p}{2}}\bigg]<\infty.
	\end{align} 

Moreover, if Hypothesis \ref{hypo1} (H.2) holds, then solution of the system \eqref{1.1} is pathwise unique and hence there exists a {\bf unique probabilistic  strong solution} to the system \eqref{1.1}. 
\end{theorem}
The next result is an immediate consequence of   Theorem \ref{thrm1}:
\begin{corollary}\label{cor1}
	Suppose that the embedding $\V\subset \H$ is compact, the operator $\A(t,\cdot)$ is pseudo-monotone for a.e. $t\in[0,T]$, and (H.1), (H.3)-(H.6) hold. Then, for any initial value $\x\in\H$, there exists a {\bf probabilistically  solution} to the system \eqref{1.1}, and the estimate \eqref{3.13} holds.
\end{corollary}

Let us now state a result on the continuous dependence of the solutions on the initial data.

\begin{theorem}\label{thrm2}
	Suppose that the embedding $\V\subset \H$ is compact, Hypothesis \ref{hypo1} (H.1),(H.2) and (H.3)-(H.6) hold. Let $\{\x_m\}$ be a sequence such that $\|\x_m-\x\|_{\H}\to0$ as $m\to\infty$.   Let $\bfX(t,\x)$ be the unique solution to the system \eqref{1.1} with the initial data $\x$. Then, for any $p\geq 2$, 
	\begin{align}\label{3.14}
		\lim_{m\to\infty}\E\bigg[\sup_{0\leq t\leq T}\|\bfX(t,\x_m)-\bfX(t,\x)\|_{\H}^p\bigg]=0.
	\end{align}
\end{theorem}

\section{Proof of Theorem \ref{thrm1}}\label{sec3}\setcounter{equation}{0}
In this section, we establish the proof of Theorem \ref{thrm1}. In this sequel, we assume that the embedding $\V\subset \H$ is compact, the conditions (H.1) and (H.2)$'$-(H.6) from Hypothesis \ref{hypo1} hold. In order to obtain the existence of a probabilistically weak solution, we start with the construction of an approximating solutions using a Faedo-Galerkin approximation and then we prove the tightness  of  laws of the approximating solutions in the approximating space. 

Let $\{e_j\}_{j=1}^\infty\subset \V$ be an orthonormal  basis of $\H$. Let  $\H_m$  be a finite dimensional space  spanned by $\{e_1,e_2,\ldots,e_m\}$.  We define a projection $\PP_m:\V^*\to \H_m$  by \begin{align}\label{3.15}
	\PP_m\u:=\sum_{j=1}^{m}\langle \u,e_j\rangle e_j.
\end{align} Clearly, the restriction of this projection denoted by $\PP_m\big|_\H$ is just the orthogonal projection of $\H$ onto $\H_m$. Since $\{\phi_j\}_{j=1}^{\infty}$ is an orthonormal basis of the Hilbert space $\U$, let us set 
\begin{align}\label{3.16}
	\W_m(t)=\Q_m\W(t)=\sum_{j=1}^{m}\langle \W(t),\phi_j\rangle \phi_j,
\end{align}where $\Q_m$ is the orthogonal projection onto $\mathrm{span}\{\phi_1,\phi_2,\ldots,\phi_m\}$ in $\U$.

Now, we consider  the following stochastic differential equation in the finite-dimensional  space $\H_m$, for any $m\geq 1$
\begin{align}\label{3.17}\nonumber
	\Y_m(t)&=\PP_m\x+\int_{0}^{t}\PP_m\A(s,\Y_m(s))\d s+\int_{0}^{t}\PP_m\B(s,\Y_m(s))\Q_m\d\W(s) \\&\quad+\int_{0}^{t}\int_\Z \PP_m\gamma(s,\Y_m(s-),z)\vi{\pi}(\d s,\d z).
\end{align}The existence and uniqueness of the strong solution of  finite dimensional system has been discussed in Theorem 1, \cite{IG}  (cf. also in Theorem 3.1, \cite{SAZBJLW} or \cite{ZBWLJZ}). We have the following uniform energy estimate for $\{\Y_m\}$.
\begin{lemma}\label{lem1}
	For any $p\geq 2$, there exists a constant $C_p>0$ such that 
	\begin{align}\label{3.18}		
		\sup_{n\in\N}\left\{\E\bigg[\sup_{0\leq t\leq T}\|\Y_m(t)\|_{\H}^p\bigg]+\E\bigg[\bigg(\int_{0}^{T}\|\Y_m(t)\|_{\V}^\beta\d t\bigg)^{\frac{p}{2}}\bigg]\right\}\leq  C_p(1+\|\x\|_{\H}^p).
	\end{align}
\end{lemma}
\begin{proof}
	It is enough to establish this lemma for large values of $p$. 
	\vskip 0.1 cm
	\noindent {\bf Step I:}
	Applying It\^o's formula to the real-valued process $\|\Y_m(\cdot)\|_{\H}^p$, we find 
			\begin{align}\label{3.20}\nonumber
			&	\|\Y_m(t)\|_{\H}^p\\&\nonumber= \|\PP_m\x\|_{\H}^p+\frac{p}{2}\int_{0}^{t}\|\Y_m(s)\|_{\H}^{p-2}\big[2\langle \A(s,\Y_m(s)),\Y_m(s)\rangle +\|\PP_m\B(s,\Y_m(s))\Q_m\|_{\L_2}^2\big]\d s\\&\nonumber\quad +\frac{p(p-2)}{2}\int_{0}^{t}\|\Y_m(s)\|_{\H}^{p-4}\|\Y_m(s)\circ \PP_m\B(s,\Y_m(s))\Q_m\|_{\U}^2\d s\\&\nonumber\quad
			+p\int_{0}^{t}\|\Y_m(s)\|_{\H}^{p-2}(\PP_m\B(s,\Y_m(s))\Q_m\d\W(s),\Y_m(s)) \\&\nonumber\quad
			+\int_{0}^{t}\int_\Z \big[\|\Y_m(s)+\PP_m\gamma(s,\Y_m(s),z)\|_{\H}^p-\|\Y_m(s)\|_{\H}^p\\&\nonumber\qquad-p\|\Y_m(s)\|_{\H}^{p-2}
			(\PP_m\gamma(s,\Y_m(s),z),\Y_m(s))\big]\lambda(\d z)\d s\\&\quad+ \int_{0}^{t}\int_\Z\big[\|\Y_m(s-)+\PP_m\gamma(s,\Y_m(s-),z)\|_{\H}^p-\|\Y_m(s-)\|_{\H}^p\big]\vi{\pi}(\d s,\d z) .
		\end{align}Let us consider the last two terms of inequality \eqref{3.20}, and using Taylor's formula, we obtain 
	\begin{align}\label{3.21}\nonumber
&	\int_{0}^{t}\int_\Z \big[\|\Y_m(s)+\PP_m\gamma(s,\Y_m(s),z)\|_{\H}^p-\|\Y_m(s)\|_{\H}^p\\&\nonumber\qquad-p\|\Y_m(s)\|_{\H}^{p-2}(\PP_m\gamma(s,\Y_m(s),z),\Y_m(s))\big]\lambda(\d z)\d s\\&\nonumber\quad+ \int_{0}^{t}\int_\Z\big[\|\Y_m(s-)+\PP_m\gamma(s,\Y_m(s-),z)\|_{\H}^p-\|\Y_m(s-)\|_{\H}^p\big]\vi{\pi}(\d s,\d z) 
\\&\nonumber= p  \int_{0}^{t}\int_\Z \|\Y_m(s-)\|_{\H}^{p-2}(\PP_m\gamma(s,\Y_m(s-),z),\Y_m(s-))\vi{\pi}(\d s,\d z)\\&\nonumber\quad 
+\frac{p}{2} \int_{0}^{t}\int_\Z  \big[\|\Y_m(s-)+\theta\PP_m\gamma(s,\Y_m(s-),z)\|_\H^{p-2}\|\PP_m\gamma(s,\Y_m(s-),z)\|_\H^2\\&\nonumber\qquad+(p-2)\|\Y_m(s-)+\theta\PP_m\gamma(s,\Y_m(s-),z)\|_\H^{p-4}\\&\qquad\times|(\Y_m(s-)+\theta\PP_m\gamma(s,\Y_m(s-),z),\PP_m\gamma(s,\Y_m(s-),z))|^2\big] \pi(\d s,\d z)\nonumber\\&\leq p  \int_{0}^{t}\int_\Z \|\Y_m(s-)\|_{\H}^{p-2}(\PP_m\gamma(s,\Y_m(s-),z),\Y_m(s-))\vi{\pi}(\d s,\d z)\nonumber\\&\quad+\frac{p(p-1)}{2} \int_{0}^{t}\int_\Z \|\Y_m(s-)+\theta\PP_m\gamma(s,\Y_m(s-),z)\|_\H^{p-2}\|\PP_m\gamma(s,\Y_m(s-),z)\|_\H^2{\pi}(\d s,\d z),
	\end{align}for some $\theta\in(0,1)$. 
Using Hypothesis \ref{hypo1} (H.3) and (H.5), and  the above inequality in \eqref{3.20}, we obtain 
\begin{align}\label{3.020}\nonumber
	&\|\Y_m(t)\|_\H^p+\frac{pC     }{2}\int_0^t\|\Y_m(s)\|_\V^\beta\|\Y_m(s)\|_\H^{p-2}\d s \\& \nonumber\leq \|\PP_m\x\|_\H^p +C\int_{0}^{t}[f(s)+g(s)]\d s+ C\int_0^t [f(s)+g(s)]\|\Y_m(s)\|_\H^p\d s\\&\nonumber \quad 
	+ p \int_0^t \|\Y_m(s)\|_\H^{p-2}\big(\B(s,\Y_m(s))\Q_m\d\W(s),\Y_m(s)\big)\\&\nonumber\quad+p\int_{0}^{t}\int_\Z\|\Y_m(s-)\|_{\H}^{p-2}(\PP_m\gamma(s,\Y_m(s-),z),\Y_m(s-))\vi{\pi}(\d s,\d z)\\&\quad 
	+ C_p\int_{0}^{t}\int_\Z\big(\|\Y_m(s-)\|_\H^{p-2}\|\PP_m\gamma(s,\Y_m(s-),z)\|_\H^2+\|\PP_m\gamma(s,\Y_m(s-),z)\|_\H^p\big)\pi(\d s,\d z).
\end{align}Define sequence of stopping times as follows:
\begin{align*}
	\tau_N^m:=T\wedge \inf\{t\geq 0:\|\Y_m(t)\|_\H>N\}.
\end{align*}
Then $\tau_N^m\to T,\;\P$-a.s., as $N\to \infty$ for every $m$. Next, taking the supremum over time from $0$ to $\t$ and then taking  expectations on both sides of the above inequality \eqref{3.020}, we deduce
\begin{align}\label{3.22}\nonumber
	&\E\bigg[\sup_{t\in[0,\t]}\|\Y_m(t)\|_\H^p\bigg]+\frac{pC}{2}\E\bigg[\int_{0}^{\t}\|\Y_m(s)\|_{\V}^\beta\|\Y_m(s)\|_\H^{p-2}\d s\bigg]\\&\nonumber\leq \|\x\|_\H^p+C\int_{0}^{t}\big[f(s)+g(s)\big]\d s+C\E\bigg[\int_{0}^{\t}\big[f(s)+g(s)\big]\|\Y_m(s)\|_\H^p\d s\bigg]\\&\nonumber\quad +p\E\bigg[\sup_{t\in[0,\t]}\bigg|\int_{0}^{t}\|\Y_m(s)\|_{\H}^{p-2}(\B(s,\Y_m(s))\Q_m\d\W(s),\Y_m(s))\bigg|\bigg]\\&\nonumber\quad 	+  p\E\bigg[\sup_{t\in[0,\t]}\bigg|\int_{0}^{t}\int_\Z\|\Y_m(s-)\|_{\H}^{p-2}(\PP_m\gamma(s,\Y_m(s-),z),\Y_m(s-))\vi{\pi}(\d s,\d z)\bigg|\bigg]\\&\nonumber\quad 	+ C_p\E\bigg[\int_{0}^{\t}\int_\Z\big(\|\Y_m(s-)\|_\H^{p-2}\|\gamma(s,\Y_m(s-),z)\|_\H^2+\|\gamma(s,\Y_m(s-),z)\|_\H^p\big)\pi(\d s,\d z)\bigg] \\& \leq 	\|\x\|_\H^p+\sum_{j=1}^{5}I_j,
\end{align} Now, we consider the term $I_3$ and estimate it using Hypothesis \ref{hypo1} (H5), Burkholder-Davis-Gundy inequality (see Theorem 1.1, \cite{DLB}), H\"older's and Young's inequalities as 
\begin{align}\label{3.23}\nonumber
	I_3& \leq 
	C_p\E\bigg[\bigg(\int_{0}^{\t}\|\Y_m(s)\|_{\H}^{2p-2}\|\B(s,\Y_m(s))\|_{\L_2}^2\d s\bigg)^{\frac{1}{2}}\bigg] \\&\nonumber \leq 
	C_p\E\bigg[\bigg(\sup_{s\in[0,\t]}\|\Y_m(s)\|_\H^p \int_{0}^{\t}\|\Y_m(s)\|_{\H}^{p-2}\|\B(s,\Y_m(s))\|_{\L_2}^2\d s\bigg)^{\frac{1}{2}}\bigg] \\&\nonumber\leq 
	\e\E\bigg[\sup_{s\in[0,\t]}\|\Y_m(s)\|_\H^p\bigg]+C_{\e,p}\E\bigg[\int_{0}^{\t}\|\Y_m(s)\|_{\H}^{p-2}\|\B(s,\Y_m(s))\|_{\L_2}^2\d s\bigg] \\& \leq 
		\e\E\bigg[\sup_{s\in[0,\t]}\|\Y_m(s)\|_\H^p\bigg]+C_{\e,p}\int_{0}^{T}g(s)\d s+C_{\e,p}\E\bigg[\int_{0}^{\t}g(s)\|\Y_m(s)\|_{\H}^p\d s\bigg],
\end{align}where $\e>0$. Again, using Burkholder-Davis-Gundy inequality, Hypothesis   \ref{hypo1} (H.6), Corollary 2.4., \cite{JZZBWL}, Young's and H\"older's inequalities to estimate the term $I_4$ as
\begin{align}\label{3.24}\nonumber
	I_4& \leq C_p\E\bigg[\int_{0}^{\t}\int_\Z\|\Y_m(s)\|_\H^{2p-2}\|\PP_m\gamma(s,\Y_m(s-),z)\|_\H^2\pi(\d s,\d z)\bigg]^{\frac{1}{2}} \\&\nonumber\leq  C_p \E\bigg[\sup_{s\in[0,\t]}\|\Y_m(s)\|_\H^{p-1}\bigg(\int_{0}^{\t}\int_\Z\|\PP_m\gamma(t,\Y_m(s-),z)\|_\H^2\pi(\d s,\d z)\bigg)^{\frac{1}{2}}\bigg] \\&\nonumber \leq 
	\e\E\bigg[\sup_{s\in[0,\t]}\|\Y_m(s)\|_\H^p\bigg] +C_{\e,p}\E\bigg[\bigg(\int_{0}^{\t}\int_\Z\|\PP_m\gamma(t,\Y_m(s-),z)\|_\H^2\pi(\d s,\d z)\bigg)^{\frac{p}{2}}\bigg] \\&\nonumber\leq 
		\e\E\bigg[\sup_{s\in[0,\t]}\|\Y_m(s)\|_\H^p\bigg]+C_{\e,p}\E\bigg[\int_{0}^{\t}\int_\Z\|\PP_m\gamma(t,\Y_m(s),z)\|_\H^p\lambda(\d z)\d s\bigg]\\&\nonumber \quad +C_{\e,p}\E\bigg[\bigg(\int_{0}^{\t}\int_\Z\|\PP_m\gamma(t,\Y_m(s),z)\|_\H^2\lambda(\d z)\d s\bigg)^{\frac{p}{2}}\bigg]  \\&\nonumber\leq 
		\e\E\bigg[\sup_{s\in[0,\t]}\|\Y_m(s)\|_\H^p\bigg]+C_{\e,p}\bigg(\int_0^Th_p(s)\d s+\E\bigg[\int_{0}^{\t} h_p(s)\|\Y_m(s)\|_\H^p\d s\bigg]\bigg)\\&\quad +C_{\e,p}\bigg(\bigg(\int_0^Th_2(s)\d s\bigg)^{\frac{p}{2}}+\bigg(\int_0^Th_2(s)\d s\bigg)^{\frac{p-2}{2}}\E\bigg[\int_{0}^{\t}h_2(s)\|\Y_m(s)\|_\H^p\d s\bigg]\bigg).
\end{align}We consider  the term $I_5$, and estimate it using \eqref{21},  Hypothesis \ref{hypo1} (H.6), H\"older's and Young's inequalities as
\begin{align}\label{3.25}\nonumber
	I_5&=C_p\E\bigg[ \int_{0}^{\t}\int_\Z  \big[\|\Y_m(s)\|_\H^{p-2}+\|\PP_m\gamma(s,\Y_m(s),z)\|_\H^{p-2}\big]\|\PP_m\gamma(s,\Y_m(s),z)\|_\H^2\lambda(\d  z)\d s\bigg]
	\\&\nonumber \leq  C_p \E\bigg[\sup_{s\in[0,\t]}\|\Y_m(s)\|_\H^{p-2}\bigg(\int_{0}^{\t}\int_\Z \|\PP_m\gamma(s,\Y_m(s),z)\|_\H^2\lambda(\d  z)\d s \bigg)\bigg]\\&\nonumber\quad+
	C_p\E\bigg[\int_{0}^{\t}\int_\Z \|\PP_m\gamma(s,\Y_m(s),z)\|_\H^p\lambda(\d  z)\d s\bigg]  \\&\nonumber \leq 
	\e\E\bigg[\sup_{s\in[0,\t]}\|\Y_m(s)\|_\H^p\bigg]+C_{\e,p}\E\bigg[\bigg(\int_{0}^{\t}\int_\Z \|\PP_m\gamma(s,\Y_m(s),z)\|_\H^2\lambda(\d  z)\d s \bigg)^{\frac{p}{2}}   \bigg]\\&\nonumber\quad +C_{\e,p}\bigg(\int_0^th_p(s)\d s+\E\bigg[\int_0^{\t}h_p(s)\|\Y_m(s)\|_\H^p\d s\bigg]\bigg) \\&\nonumber\leq \e\E\bigg[\sup_{s\in[0,\t]}\|\Y_m(s)\|_\H^p\bigg]+C_{\e,p}\bigg(\int_0^Th_p(s)\d s+\E\bigg[\int_0^{\t}h_p(s)\|\Y_m(s)\|_\H^p\d s\bigg]\bigg) \\&\quad +C_{\e,p}\bigg(\bigg(\int_0^Th_2(s)\d s\bigg)^{\frac{p}{2}}+\bigg(\int_0^Th_2(s)\d s\bigg)^{\frac{p-2}{2}}\E\bigg[\int_{0}^{\t}h_2(s)\|\Y_m(s)\|_\H^p\d s\bigg]\bigg).
\end{align}Combining \eqref{3.20}-\eqref{3.25}, and choosing an appropriate parameter $\e$ and then applying Gronwall's inequality,  we get 
\begin{align}\label{3.26}\nonumber
&\E\bigg[\sup_{t\in[0,\t]}\|\Y_m(t)\|_\H^p\bigg]+C\E\bigg[\int_{0}^{\t}\|\Y_m(s)\|_\V^\beta\|\Y_m(s)\|_\H^{p-2}\d s\bigg] \\&\nonumber \leq C\bigg(\|\x\|_\H^p+\int_{0}^{T}\big[f(s)+g(s)+h_p(s)\big]\d s+\bigg(\int_0^Th_2(s)\d s\bigg)^{\frac{p}{2}}\bigg)\\&\quad\times \exp\bigg\{C\int_{0}^{T}\big[f(t)+g(t)+h_p(t)\big]\d t+\bigg(\int_0^Th_p(s)\d s\bigg)^{\frac{p-2}{2}}\bigg\}.
\end{align}Passing $N\to \infty$, and applying Fatou's lemma, we find for all $p\geq 2$,
\begin{align}\label{3.27}
	\sup_{m\in\N}\left\{ \E\bigg[\sup_{0\leq t\leq T}\|\Y_m(t)\|_\H^p\bigg]+\E\bigg[\int_{0}^{T}\|\Y_m(s)\|_\V^\beta\|\Y_m(s)\|_\H^{p-2}\d s\bigg]\right\}<\infty.
\end{align}
	\vskip 0.1 cm
\noindent {\bf Step II:}
Again, we apply It\^o's formula to the process $\|\Y_m(\cdot)\|_{\H}^2$, to find
\begin{align}\label{3.19}\nonumber
	\|\Y_m(t)\|_{\H}^2&= \|\PP_m\x\|_{\H}^2+\int_{0}^{t}\big[2\langle \A(s,\Y_m(s)),\Y_m(s)\rangle +\|\PP_m\B(s,\Y_m(s))\Q_m\|_{\L_2}^2\big]\d s\\&\nonumber\quad
+2\int_{0}^{t}(\B(s,\Y_m(s))\Q_m\d\W(s),\Y_m(s))\\&\nonumber\quad+\int_{0}^{t}\int_\Z \big[\|\Y_m(s-)+\PP_m\gamma(s,\Y_m(s-),z)\|_{\H}^2-\|\Y_m(s-)\|_{\H}^2\big]\vi{\pi}(\d s,\d z)
\\&\nonumber\quad
+\int_{0}^{t}\int_\Z \big[\|\Y_m(s)+\PP_m\gamma(s,\Y_m(s),z)\|_{\H}^2-\|\Y_m(s)\|_{\H}^2\\&\nonumber\qquad
-2(\PP_m\gamma(s,\Y_m(s),z),\Y_m(s))\big]\lambda(\d z)\d s  \\&\nonumber = 
\|\PP_m\x\|_{\H}^2+\int_{0}^{t}\big[2\langle \A(s,\Y_m(s)),\Y_m(s)\rangle +\|\PP_m\B(s,\Y_m(s))\Q_m\|_{\L_2}^2\big]\d s\\&\nonumber\quad
+2\int_{0}^{t}(\B(s,\Y_m(s))\Q_m\d\W(s),\Y_m(s))+ \int_{0}^{t}\int_\Z\|\PP_m\gamma(s,\Y_m(s-),z)\|_{\H}^2\pi(\d s,\d z)\\&\quad+2\int_{0}^{t}\int_\Z(\PP_m\gamma(s,\Y_m(s-),z),\Y_m(s-))\vi{\pi}(\d s,\d z).
\end{align}
Using Hypothesis \ref{hypo1} (H.3) and (H.5) in \eqref{3.19} it follows that 
\begin{align*}
	&\|\Y_m(t)\|_\H^2+C\int_{0}^{t}\|\Y_m(s)\|_\V^\beta\d s\\& \leq \|\PP_m\x\|_\H^2+\int_{0}^{t}\left\{\big(f(s)+g(s)\big)\|\Y_m(s)\|_\H^2+f(s)+g(s)\right\}\d s\\&\quad +2\int_{0}^{t}(\B(s,\Y_m(s))\Q_m\d\W(s),\Y_m(s))+\int_{0}^{t}\int_\Z\|\PP_m\gamma(s,\Y_m(s-),z)\|_{\H}^2\pi(\d s,\d z)\\&\quad+2\int_{0}^{t}\int_\Z(\PP_m\gamma(s,\Y_m(s-),z),\Y_m(s-))\vi{\pi}(\d s,\d z). 
\end{align*}
Therefore, we arrive at 
\begin{align}\label{3.28}\nonumber
	&\E\Bigg[\bigg(\int_{0}^{t}\|\Y_m(s)\|_\V^\beta\d s\bigg)^{\frac{p}{2}}\Bigg]\\& \nonumber\leq C_p\|\PP_m\x\|_\H^p+C_p\E\bigg[\int_{0}^{t}\left\{\big(f(s)+g(s)\big)\|\Y_m(s)\|_\H^2+f(s)+g(s)\right\}\d s\bigg]^{\frac{p}{2}}\\&\nonumber\quad +C_p\E\bigg[\bigg|\int_{0}^{t}(\B(s,\Y_m(s))\Q_m\d\W(s),\Y_m(s))\bigg|^{\frac{p}{2}}\bigg]\\&\quad \nonumber+ C_p\E\bigg[\bigg(\int_{0}^{t}\int_\Z\|\PP_m\gamma(s,\Y_m(s-),z)\|_{\H}^2\pi(\d s,\d z)\bigg)^{\frac{p}{2}}\bigg]\\&\quad+C_p\E\bigg[\bigg|\int_{0}^{t}\int_\Z(\PP_m\gamma(s,\Y_m(s-),z),\Y_m(s-))\vi{\pi}(\d s,\d z)\bigg|^{\frac{p}{2}}\bigg]. 
\end{align}Applying Burkholder-Davis-Gundy inequality, Hypothesis  \ref{hypo1} (H.5), Young's and H\"older's inequalities in the third term of the right hand side of the above inequality \eqref{3.28}, we find
\begin{align}\label{3.29}\nonumber
	&C_p\E\bigg[\bigg|\int_{0}^{t}(\B(s,\Y_m(s))\Q_m\d\W(s),\Y_m(s))\bigg|^{\frac{p}{2}}\bigg] \\& \nonumber\leq C_p\E\bigg[\int_0^T\|\Y_m(s)\|_\H^2\|\B(s,\Y_m(s))\|_{\L_2}^2\d s\bigg]^{\frac{p}{4}}\nonumber\\&\leq   C_p\E\bigg[\sup_{0\leq s\leq T} \|\Y_m(s)\|_\H^p\bigg]+C_p\E\left[\left(\int_0^T\|\B(s,\Y_m(s))\|_{\L_2}^2\d s\right)^{\frac{p}{2}}\right]\\& \nonumber\leq C_p\E\bigg[\sup_{0\leq s\leq T} \|\Y_m(s)\|_\H^p\bigg]+C_p\E\left[\left(\int_0^Tg(s)(1+\|\Y_m(s)\|_{\H}^2)\d s\right)^{\frac{p}{2}}\right]\\&\leq C_p\E\bigg[\sup_{0\leq s\leq T} \|\Y_m(s)\|_\H^p\bigg]+ C_p\bigg(\int_0^Tg(s)\d s\bigg)^{\frac{p}{2}}\bigg(1+\E\bigg[\sup_{0\leq s\leq T}\|\Y_m(s)\|_\H^p\bigg]\bigg).
\end{align}We estimate the penultimate term from the right hand side of the inequality \eqref{3.28}, using Corollary 2.4., \cite{JZZBWL} and Hypothesis  \ref{hypo1} (H.6) as 
\begin{align}\label{3.028}\nonumber
&C_p\E\bigg[\bigg(\int_{0}^{t}\int_\Z\|\PP_m\gamma(s,\Y_m(s-),z)\|_{\H}^2\pi(\d s,\d z)\bigg)^{\frac{p}{2}}\bigg]\\&\nonumber \leq C_p \E\bigg[\int_{0}^{T}\int_\Z\|\PP_m\gamma(s,\Y_m(s),z)\|_{\H}^p\lambda(\d z)\d s\bigg] \\&\quad\nonumber +C_p\E\bigg[\bigg(\int_{0}^{T}\int_\Z\|\PP_m\gamma(s,\Y_m(s),z)\|_{\H}^2\lambda(\d z)\d s\bigg)^{\frac{p}{2}}\bigg] \\&\nonumber\leq C_p\bigg(\int_0^Th_p(t)\d t+\E\bigg[\int_0^Th_p(s)\|\Y_m(s)\|_\H^p\d s\bigg]\bigg)\\&\quad +
C_p\bigg(\bigg(\int_0^Th_2(s)\d s\bigg)^{\frac{p}{2}}+\bigg(\int_0^Th_2(s)\d s\bigg)^{\frac{p-2}{2}}\E\bigg[\int_{0}^{T}h_2(s)\|\Y_m(s)\|_\H^p\d s\bigg]\bigg).
\end{align}
Now, we consider the final term of the right hand side of the inequality \eqref{3.28} and estimate it using Burkholder-Davis-Gundy inequality, Hypothesis  \ref{hypo1} (H.6), Corollary 2.4., \cite{JZZBWL}, Young's and H\"older's inequalities, as
\begin{align}\label{3.30}\nonumber
	&C_p\E\bigg[\bigg|\int_0^t\int_\Z (\PP_m\gamma(s,\Y_m(s-),z),\Y_m(s))\vi{\pi}(\d s,\d z)\bigg|^{\frac{p}{2}}\bigg] \\&\nonumber\leq  C_p\E\bigg[\bigg(\int_0^{T}\int_\Z \|\PP_m\gamma(s,\Y_m(s-),z)\|_\H^2\|\Y_m(s)\|_\H^2\pi(\d s,\d z)\bigg)^{\frac{p}{4}} \bigg]\\&\nonumber \leq C_p\E\bigg[\sup_{0\leq s\leq T} \|\Y_m(s)\|_\H^p\bigg]+C_p\E\bigg[\bigg(\int_0^T\int_\Z\|\PP_m\gamma (s,\Y_m(s),z)\|_\H^2\pi(\d s,\d z)\bigg)^{\frac{p}{2}} \bigg] \\&\nonumber\leq 
C_p\E\bigg[\sup_{0\leq s\leq T} \|\Y_m(s)\|_\H^p\bigg]+C_p\bigg(\int_0^Th_p(s)\d s+\E\bigg[\int_{0}^{T}h_p(s)\|\Y_m(s)\|_\H^p\d s\bigg]\bigg)\\&\quad+	C_p\bigg(\bigg(\int_0^Th_2(s)\d s\bigg)^{\frac{p}{2}}+\bigg(\int_0^Th_2(s)\d s\bigg)^{\frac{p-2}{2}}\E\bigg[\int_{0}^{T}h_2(s)\|\Y_m(s)\|_\H^p\d s\bigg]\bigg).
\end{align}Combining \eqref{3.28}-\eqref{3.30}, we arrive at 
\begin{align}\label{3.31}
	\sup_{m\in\N}\E\bigg[\int_{0}^{T}\|\Y_m(s)\|_\V^\beta\d s\bigg]^{\frac{p}{2}}<\infty,
\end{align}
which completes the proof. 
\end{proof}
Let now prove the tightness property of the laws of $\{\Y_m\}$. Define a sequence of stopping times as follows:
\begin{align}\label{3.32}
	\tau_m^N:=T\wedge\inf\{t\ge0:\|\Y_m(t)\|_\H^2>N\}\wedge\inf\bigg\{t\geq0:\int_0^t\|\Y_m(s)\|_\V^\beta\d s>N\bigg\}\leq T,
\end{align}with the convection that infimum of a void set is infinite. Then, by Markov's inequality and Lemma \ref{lem1}, we have 
\begin{align}\label{3.33}
	\lim_{N\to\infty}\sup_{m\in\N}\P(\tau_m^N<T)=0.
\end{align}

\begin{lemma}\label{lem2}
	The set of measures $\{\mathscr{L}(\Y_m):m\in\N\}$ is tight on $\L^\beta(0,T;\H)$ for $\beta\in(1,\infty)$.
\end{lemma}
\begin{proof}
Since, we have 
	\begin{align}\label{3.34}
		\sup_{m\in\N}\E\bigg[\int_{0}^{T}\|\Y_m(t)\|_\V^\beta\d t\bigg]<\infty,
	\end{align}using Lemma \ref{lemRoc} below (see Lemma 5.2, \cite{MRSSTZ}), it is enough to prove for any $\e>0$ 
\begin{align}\label{3.35}
	\lim_{\delta\to0^+}\sup_{m\in\N}\P\bigg(\int_0^{T-\delta}\|\Y_m(t+\delta)-\Y_m(t)\|_\H^\beta\d t>\e\bigg)=0.
\end{align}
 Let us fix $\Y_m^N(t):=\Y_m(t\wedge\tau_m^N)$. An application of Markov's inequality yields
\begin{align}\label{3.36}\nonumber
	&\P\bigg(\int_0^{T-\delta}\|\Y_m(t+\delta)-\Y_m(t)\|_\H^\beta\d t>\e\bigg)\\&\nonumber\leq \P  \bigg(\int_0^{T-\delta}\|\Y_m(t+\delta)-\Y_m(t)\|_\H^\beta\d t>\e,\tau_m^N=T\bigg)+\P(\tau_m^N<T)
	\\& \leq \frac{1}{\e}\E\bigg[\int_0^{T-\delta}\|\Y_m^N(t+\delta)-\Y_m^N(t)\|_\H^\beta\d t\bigg]+\P(\tau_m^N<T).
\end{align}If we manage to show that for any fixed $N>0$ 
\begin{align}\label{3.37}
	\lim_{\delta\to0+}\sup_{m\in\N}\E\bigg[\int_0^{T-\delta}\|\Y_m^N(t+\delta)-\Y_m^N(t)\|_\H^\beta\d t\bigg]=0,
\end{align}then, in view of \eqref{3.33}, passing $\delta\to0$ and then $N\to\infty$ in \eqref{3.36}, we get \eqref{3.35}, which completes the proof of tightness of laws of $\{\Y_m\}$ in $\L^\beta(0,T;\H)$. Thus, if we establish \eqref{3.37}, then we are done. To complete this, we divide the proof in two parts which depends on the values of $\beta$, that is, $1<\beta\leq 2$ and $\beta>2$.
\vskip 0.1 cm 
\noindent 
\textbf{Part I.} We first consider the case when $\beta\in(1,2]$. Applying It\^o's formula, we find 
\begin{align}\label{3.38}\nonumber
&\E\bigg[\|\Y_m^N(t+\delta)-\Y_m^N(t)\|_\H^2\bigg] \\&\nonumber=\E\bigg[\int_{\tt}^{\td}2\langle \A(s,\Y_m(s)),\Y_m(s)-\Y_m(\tt)\rangle \d s\bigg]	\\&\nonumber\quad
+ \E\bigg[\int_{\tt}^{\td}\|\PP_m\B(s,\Y_m(s))\Q_m\|_{\L_2}^2\d s\bigg]\\&\quad+ \E\bigg[\int_{\tt}^{\td}\int_\Z\|\PP_m\gamma(s,\Y_m(s-),z)\|_\H^2\lambda(\d z)\d s\bigg],
\end{align}
where we have used \eqref{21}. From \eqref{3.38}, we conclude that
\begin{align}\label{3.39}\nonumber
&\E\bigg[\int_0^{T-\delta}\|\Y_m^N(t+\delta)-\Y_m^N(t)\|_\H^2\d t\bigg] \\ &\nonumber= 
\E\bigg[\int_0^{T-\delta}\bigg\{\int_{\tt}^{\td}\bigg(2\langle \A(s,\Y_m(s)),\Y_m(s)\rangle +\|\PP_m\B(s,\Y_m(s))\Q_m\|_{\L_2}^2\\&\nonumber\qquad+\int_\Z\|\PP_m\gamma(s,\Y_m(s),z)\|_\H^2\lambda(\d  z)\bigg)\d s\bigg\}\d t\bigg]	\\&\nonumber\nonumber\quad- 2\E\left[\int_0^{T-\delta}\left\{\int_{\tt}^{\td}\langle \A(s,\Y_m(s)),\Y_m(\tt)\rangle \d s\right\}\d t\right]\\&=: J_1+J_2.
\end{align}
Since we are considering the case $\tau_m^N=T$ only, we have $\mathbb{P}\left(\chi_{\{\tau_m^N\leq t\leq T-\delta\}}\right)=0$, and using Fubini's theorem and Hypothesis  \ref{hypo1} (H.3), (H.5), (H.6), we estimate the term $J_1$ as 
\begin{align}\label{3.40}\nonumber
	J_1& = \E\bigg[\int_0^{\T}\left(\int_{0\vee(s-\delta)}^s\chi_{\{\tau_m^N>t\}}\d t\right)\bigg(2\langle \A(s,\Y_m(s)),\Y_m(s)\rangle +\|\PP_m\B(s,\Y_m(s))\Q_m\|_{\L_2}^2\\&\nonumber\qquad+\int_\Z\|\PP_m\gamma(s,\Y_m(s),z)\|_\H^2\lambda(\d  z)\bigg)\d s\bigg] \\&\nonumber \leq \delta \E\bigg[\int_0^{\T}\big(f(s)+g(s)+h_2(s)\big)(1+\|\Y_m(s)\|_\H^2)\d s\bigg] \\&\nonumber\leq \delta \int_0^{T}\big(f(s)+g(s)+h_2(s)\big)\d s\left(1+\E\bigg[\sup_{s\in[0,T]}\|\Y_m(s)\|_\H^2\bigg]\right) \\&\leq C\delta.
\end{align}Again, applying Fubini's theorem and Hypothesis  \ref{hypo1} (H.4), to estimate the term $J_2$ as 
\begin{align}\label{3.41}\nonumber
	|J_2| &\leq 2\E\bigg[\bigg|\int_0^{\T}\left(\int_{0\vee(s-\delta)}^s\chi_{\{\tau_m^N>t\}}\langle \A(s,\Y_m(s)),\Y_m(\tt)\rangle\right) \d t\d s\bigg|\bigg] \\&\nonumber \leq 2 \E\bigg[\int_0^{\T}\|\A(s,\Y_m(s))\|_{\V^*}\left(\int_{0\vee(s-\delta)}^s\|\Y_m(\tt)\|_\V\d t\right)\d s\bigg] \\&\nonumber \leq 2\delta^{\frac{\beta-1}{\beta}}\E\bigg[\int_0^{\T}\|\A(s,\Y_m(s))\|_{\V^*}\d s\bigg(\int_0^{\T}\|\Y_m(t)\|_\V^\beta\d t\bigg)^{\frac{1}{\beta}}\bigg] \\&\nonumber\leq  2\delta^{\frac{\beta-1}{\beta}}T^{\frac{1}{\beta}}\Bigg\{\E\bigg[\int_0^{\T}\|\A(s,\Y_m(s))\|_{\V^*}^{\frac{\beta}{\beta-1}}\d s\bigg]\Bigg\}^{\frac{\beta-1}{\beta}}\Bigg\{\E\bigg[\int_0^{\T}\|\Y_m(t)\|_\V^\beta\d t\bigg]\Bigg\}^{\frac{1}{\beta}}\\& \leq C\delta^{\frac{\beta-1}{\beta}}.
\end{align}
Combining \eqref{3.39}-\eqref{3.41}, we obtain 
\begin{align}\label{3.43}
	\sup_{m\in\N}\E\bigg[\int_0^{T-\delta}\|\Y_m^N(t+\delta)-\Y_m^N(t)\|_\H^2\d t\bigg] \leq C(\delta+\delta^{\frac{\beta-1}{\beta}}).
\end{align}Now, for $\beta\in(1,2]$, we use H\"older's inequality to get 
\begin{align}\label{3.44}\nonumber
	&\lim_{\delta\to0+} \sup_{m\in\N}\E\bigg[\int_0^{T-\delta}\|\Y_m^N(t+\delta)-\Y_m^N(t)\|_\H^\beta\d t\bigg] \\& \leq C\lim_{\delta\to0+}\sup_{m\in\N}\Bigg\{\E\bigg[\int_0^{T-\delta}\|\Y_m^N(t+\delta)-\Y_m^N(t)\|_\H^2\d t\bigg]\Bigg\}^{\frac{\beta}{2}}=0, 
\end{align} which completes the proof of \eqref{3.37} for $\beta\in(1,2]$.
\vskip 0.1 cm
\noindent\textbf{Part II.}
Now, we move towards the remaining values of $\beta$, that is, $\beta\in(2,\infty)$. Again, applying It\^o's formula to the process $\|\cdot\|_\H^\beta$, and  then taking expectations, we find
\begin{align*}\nonumber
&\E\big[\|\Y_m^N(t+\delta)-\Y_m^N(t)\|_\H^\beta\big]	 \\&\nonumber=\frac{\beta}{2}\E\bigg[\int_{\tt}^{\td}\|\Y_m(s)-\Y_m(\tt)\|_\H^{\beta-2}\big[2\langle \A(s,\Y_m(s)),\Y_m(s)-\Y_m(\tt)\rangle\\&\nonumber\qquad
+\|\PP_m\B(s,\Y_m(s))\Q_m\|_{\L_2}^2\big]\d s\bigg]
+\frac{\beta(\beta-2)}{2}\E\bigg[\int_{\tt}^{\td}\|\Y_m(s)-\Y_m(\tt)\|_{\H}^{\beta-4}\\&\nonumber\qquad\times\|(\Y_m(s)-\Y_m(\tt))\circ \PP_m\B(s,\Y_m(s))\Q_m\|_{\U}^2\d s\bigg] \\&\quad+\E\bigg[\int_{\tt}^{\td}\int_\Z \big(\|\Y_m(s)-\Y_m(\tt)+\PP_m\gamma(s,\Y_m(s),z)\|_\H^{\beta}-\|\Y_m(s)-\Y_m(\tt)\|_\H^\beta\\&\nonumber\qquad-\beta\|\Y_m(s)-\Y_m(\tt)\|_\H^{\beta-2}\big(\PP_m(s,\Y_m(s),z),\Y_m(s)-\Y_m(\tt)\big)\big)\lambda(\d z)\d s\bigg].
\end{align*}
Using Fubini's theorem and Taylor's formula, we find 
\begin{align}\label{3.45}\nonumber
	&\E\bigg[\int_0^{T-\delta}\|\Y_m^N(t+\delta)-\Y_m^N(t)\|_\H^\beta\d t\bigg]	 \\&\nonumber\leq C_{\beta}\E\bigg[\int_0^{T-\delta}\bigg\{\int_{\tt}^{\td}\|\Y_m(s)-\Y_m(\tt)\|_\H^{\beta-2}\big[2\langle \A(s,\Y_m(s)),\Y_m(s)\rangle\\&\nonumber\qquad
	+\|\PP_m\B(s,\Y_m(s))\Q_m\|_{\L_2}^2+\int_{\Z}\|\PP_m\gamma(s,\Y_m(s),z)\|_\H^2\lambda(\d z)\big]\d s\bigg\}\d t\bigg]\\&\nonumber\quad-\beta\E\bigg[\int_0^{T-\delta}\bigg\{\int_{\tt}^{\td}  \|\Y_m(s)-\Y_m^N(\tt)\|_\H^{\beta-2}\langle \A(s,\Y_m(s)),\Y_m(\tt)(s)\rangle  \d s\bigg\}\d t\bigg]
\\&\nonumber\quad	
+C_\beta\E\bigg[\int_0^{T-\delta}\bigg\{\int_{\tt}^{\td}\int_\Z \|\PP_m\gamma(s,\Y_m(s),z)\|_\H^{\beta}\lambda(\d z)\d s\bigg\}\d t\bigg]
	\\&=:\sum_{i=1}^{3}I_i.
\end{align}
Since we are considering the case $\tau_m^N=T$ only, we know that $\mathbb{P}\left(\chi_{\{\tau_m^N\leq t\leq T-\delta\}}\right)=0$, applying Fubini's theorem, Hypothesis \ref{hypo1} (H.3), (H.5), H\"older's and Young's inequalities, we estimate the term $I_1$ as
\begin{align}\label{Tigh1}\nonumber
	I_1 &\leq  C \E\bigg[\int_0^{\T}\bigg(\delta\|\Y_m(s)\|_\H^{\beta-2}+\delta\sup_{0\leq t\leq \T}\|\Y_m(t)\|_\H^{\beta-2}\bigg)\\&\nonumber\qquad\times \bigg\{(f(s)+g(s)+h_2(s))\bigg(1+\|\Y_m(s)\|_{\H}^2\bigg)\bigg\}\d s\bigg] \\& \nonumber\leq  C\E\bigg[\delta\int_0^{\T}\big(f(s)+g(s)+h_2(s)\big)\big(1+\|\Y_m(s)\|_\H^{2}\big)\|\Y_m(s)\|_\H^{\beta-2}\d s\\&\nonumber\qquad 
	+\delta\sup_{0\leq t\leq \T}\|\Y_m(t)\|_\H^{\beta-2}\int_0^{\T}\big(f(s)+g(s)+h_2(s)\big)\big(1+\|\Y_m(s)\|_\H^{2}\big)\d s\bigg]\\&\nonumber \leq 
C	\delta \E\bigg[\int_0^{\T} \left(f(s)+g(s)+h_2(s)\right)\left(1+\|\Y_m(s)\|_\H^{\beta}\right)\d s\bigg]\\&\nonumber \quad +C\delta\E\bigg[\sup_{0\leq t\leq \T}\left\{ \|\Y_m(t)\|_\H^{\beta-2}\big(1+\|\Y_m(t)\|_\H^{2}\big)\right\}\int_0^{T}\big(f(s)+g(s)+h_2(s)\big)\d s\bigg]
	\\&\leq C\delta.
\end{align}Now, consider the term $I_2$ and we estimate it using Fubini's theorem as
\begin{align}\label{Tigh2}\nonumber
	|I_2| &\leq  \beta\E\bigg[\bigg|\int_0^{\T}\int_{0\vee(s-\delta)}^s\chi_{\{\tau_m^N>t\}}\|\Y_m(s)-\Y_m(\tt)\|_\H^{\beta-2}\langle \A(s,\Y_m(s)),\Y_m(\tt)\rangle\d t\d s\bigg|\bigg]
	\\&\nonumber\leq 
	C_\beta\E\bigg[\int_0^{\T}\int_{0\vee(s-\delta)}^s\chi_{\{\tau_m^N>t\}}\|\Y_m(s)\|_\H^{\beta-2}\|\A(s,\Y_m(s))\|_{\V^*}\|\Y_m(\tt)\|_\V\d t\d s\bigg]\\&\nonumber\quad +C_\beta\E\bigg[\int_0^{\T}\int_{0\vee(s-\delta)}^s\chi_{\{\tau_m^N>t\}}\|\Y_m(\tt)\|_\H^{\beta-2}\|\A(s,\Y_m(s))\|_{\V^*}\|\Y_m(\tt)\|_\V\d t\d s\bigg]\\& =:
	I_{21}+I_{22}.
\end{align}Consider the term $I_{21}$ and we estimate it using H\"older's inequality and Hypothesis \ref{hypo1} (H.4), as 
\begin{align}\label{Tigh3}\nonumber
	I_{21} &\leq C_\beta \E\bigg[\sup_{0\leq s\leq \T}\|\Y_m(s)\|_\H^{\beta-2}\int_0^{\T}\|\A(s,\Y_m(s))\|_{\V^*}\left(\int_{0\vee(s-\delta)}^s\|\Y_m(\tt)\|_\V\d t \right)\d s\bigg]
	\\&\nonumber\leq C_\beta \delta^{\frac{\beta-1}{\beta}}\E\bigg[\sup_{0\leq s\leq \T}\|\Y_m(s)\|_\H^{\beta-2}\int_0^{\T}\|\A(s,\Y_m(s))\|_{\V^*}\d s\bigg(\int_0^{\T}\|\Y_m(t)\|_\V^{\beta}\d t\bigg)^{\frac{1}{\beta}}\bigg] \\&\nonumber \leq C_\beta\delta^{\frac{\beta-1}{\beta}}T^{\frac{1}{\beta}}\bigg\{\E\bigg[\sup_{0\leq s\leq \T}\|\Y_m(s)\|_\H^{\frac{\beta(\beta-2)}{2}}\bigg]\bigg\}^{\frac{2}{\beta}}\bigg\{\E\bigg[\int_0^{\T}\|\A(s,\Y_m(s))\|_{\V^*}^{\frac{\beta}{\beta-1}}\d s\bigg]\bigg\}^{\frac{\beta-1}{\beta}}\\&\nonumber\qquad\times 
	\bigg\{\E\bigg[\int_0^{\T}\|\Y_m(t)\|_\V^\beta\d t\bigg]^{\frac{1}{2}}\bigg\}^{\frac{2}{\beta}}\\& \leq C\delta^{\frac{\beta-1}{\beta}}.
\end{align}Similarly, we can estimate $I_{22}$ as 
\begin{align}\label{Tigh4}\nonumber
	I_{22}&\leq C_\beta \delta^{\frac{\beta-1}{\beta}}\E\bigg[\sup_{0\leq t\leq \T}\|\Y_m(t)\|_\H^{\beta-2}\int_0^{\T}\|\A(s,\Y_m(s))\|_{\V^*}\d s\bigg(\int_0^{\T}\|\Y_m(t)\|_\V^{\beta}\d t\bigg)^{\frac{1}{\beta}}\bigg]
	\\& \leq C\delta^{\frac{\beta-1}{\beta}}.
\end{align}Substituting \eqref{Tigh3} and \eqref{Tigh4} in \eqref{Tigh2}, we obtain
\begin{align}\label{Tigh5}
	I_2\leq C\delta^{\frac{\beta-1}{\beta}}.
\end{align}
Now, we consider the term $I_3$ and use  the Fubini's theorem, Hypothesis  \ref{hypo1} (H.6), H\"older's  and Young's inequalities to estimate it as
\begin{align}\label{Tigh6}
	I_3& \leq C_\beta  \delta\E\bigg[\int_0^{\T} h_{\beta}(s)\big(1+\|\Y_m(s)\|_\H^{\beta}\big)\d s\bigg]
	\leq C \delta.
\end{align}Substituting \eqref{Tigh1}, \eqref{Tigh5} and \eqref{Tigh6} in \eqref{3.35}, we arrive at
\begin{align*}
	\sup_{m\in\N}\E\bigg[\int_0^{T-\delta}\|\Y_m(s)-\Y_m(\tt)\|_\H^\beta\bigg] \leq C(\delta+\delta^{\frac{\beta-1}{\beta}}). 
\end{align*}Hence, \eqref{3.37} follows for $\beta\in(2,\infty)$.\end{proof}

Our aim is to find a suitable solution which  is c\`adl\`ag, so  we need to prove the tightness of the family $\{\mathscr{L}(\Y_m):m\in\N\}$ on the space of functions that are c\`adl\`ag in time endowed with the Skorokhod topology (for more details, see \cite{PB} or  \cite{MM}). For the required result, we need to verify the following Aldous condition  (\cite{DA} or Theorem 3.2, \cite{MM}). 
\begin{definition}
		Let $(\mathbb{Y},\|\cdot\|_{\mathbb{Y}})$ be a separable Banach space and let $\{\Y_m\}$ be a sequence of $\mathbb{Y}$-valued random variables. Assume that for every $\e,\eta>0$, there is a $\delta>0$ such that for every sequence  $\{\tau_m\}_{m\in\N}$  of $\mathscr{F}$-stopping times with $\tau_m\leq T$, one has 
		\begin{align}
			\sup_{m\in\mathbb{N}}\sup_{0<\xi\leq \delta}\P\left\{\|\Y_m((\tau_m+\xi)\wedge T)-\Y_m(\tau_m)\|_{\mathbb{Y}}\geq \eta\right\}\leq \e. 
		\end{align}
	In this case, we say that the sequence $\{\Y_m\}$  satisfies the Aldous condition. 
\end{definition}
If a  sequence $\{\Y_m\}$ satisfies the Aldous condition in the space $\mathbb{Y}$, then  the laws of $\{\Y_m\}$ form a tight sequence on $\mathrm{D}([0,T];\mathbb{Y})$ endowed with the Skorokhod topology. By the Chebyshev inequality, the following  result provides a sufficient condition for the verification of Aldous condition. 
\begin{lemma}\label{lemAldous}
Given any $\varepsilon>0$,  there is a $\delta>0$ such that 
	\begin{align}\label{3.035}
		\sup_{m\in\N}\sup_{0<\xi\leq \delta}\E\big[\|\Y_m((\tau_m+\xi)\wedge T)-\Y_m(\tau_m)\|_{\mathbb{Y}}^\zeta\big]\leq \varepsilon,
	\end{align}
for some $\zeta>0$, then the sequence $\{\Y_m\}$ satisfies the Aldous condition in the space $\mathbb{Y}$. 
\end{lemma}
One can see Theorem 13.2, \cite{Te1}, where the author used a similar condition in the above lemma for the deterministic setting. Interested readers can see the works \cite{MTM2,PNKTRT}, etc.,  where the authors established  a similar condition to verify the Aldous condition for the stochastic setting.

Now, our aim is to prove that the laws of sequence of approximated solutions $\{\Y_m\}$ denoted by $\mathscr{L}(\Y_m)$ are tight  as a probability measure on the space $\mathcal{Y}=\mathrm{D}([0,T];\V^*)\cap \L^\beta(0,T;\H)$. We have already proved the set of measures $\{\mathscr{L}(\Y_m):m\in\N\}$ is tight in $\L^\beta(0,T;\H)$ in Lemma \ref{lem2}. 
A similar result is available in  \cite{PNKTRT} (see Proposition 4.8).

\begin{proposition}\label{proptightness}
	Assume that $\x\in\H$. Then the laws $\{\mathscr{L}(\Y_m):m\in \N)\}$ of the Galerkin approximations form a tight sequence of probability measures on $\mathrm{D}([0,T];\V^*)$ endowed with the Skorokhod topology.
\end{proposition}
\begin{proof} 
	Let us consider a sequence of stopping times $\{\tau_m\}$ such that $0\leq \tau_m\leq T$, and from \eqref{3.17}, we find 
	\begin{align*}
			\Y_m(t)&=\PP_m\x+\int_{0}^{t}\PP_m\A(s,\Y_m(s))\d s+\int_{0}^{t}\PP_m\B(s,\Y_m(s))\Q_m\d\W(s) \\&\quad+\int_{0}^{t}\int_\Z \PP_m\gamma(s,\Y_m(s-),z)\vi{\pi}(\d s,\d z)
			\\& =: \PP_m\x+I_1(t)+I_2(t)+I_3(t).
	\end{align*}Let $\xi$ be a positive number and $\Y_m^N(t):=\Y_m(t\wedge\tau_m^N)$. In order to verify the Aldous condition, we need to verify the condition \eqref{3.035} in Lemma \ref{lemAldous} for all terms in the above equality for suitable $\zeta,\varepsilon$ and $\mathbb{Y}=\V^*$. For simplicity of notation, we take $(\tau_m+\xi)\wedge T$ as $\tau_m+\xi$. We consider the term $I_1$ and estimate it using Hypothesis  \ref{hypo1} (H.4) and H\"older's inequality as
\begin{align}\label{i1}
	\E\big[\|\I_1(\tau_m+\xi)-I_1(\tau_m)\|_{\V^*}^\beta\big] &\leq C\xi^{\frac{1}{\beta}}\E\bigg[\int_{\tau_m}^{\tau_m+\xi}\|\A(s,\Y_m^N(s))\|_{\V^*}^{\frac{\beta}{\beta-1}}\d s\bigg]^{\beta-1}
\nonumber	\\& \leq C\xi^{\frac{1}{\beta}}\E\bigg[\int_{0}^{T\wedge \tau_m^N}\big(f(s)+C\|\Y_m(s)\|_\V^\beta\big)\big(1+\|\Y_m(s)
	\|_\H^\alpha\big)\d s\bigg]^{\beta-1} \nonumber\\& \leq C\xi^{\frac{1}{\beta}}.
\end{align}where we have used the fact that $f\in\L^1(0,T;\R_+)$ and \eqref{3.18}. 
Consider the term $I_2$ and we estimate it using It\^o isometry, Hypothesis  \ref{hypo1} (H.5) and  H\"older's inequality as
\begin{align*}
&\E\big[\|I_2(\tau_m+\xi)-I_2(\tau_m)\|_{\V^*}^2\big]\nonumber\\& \leq C\E\bigg[\bigg\|\int_{\tau_m}^{\tau_m+\xi}\PP_m\B(s,\Y_m^N(s))\Q_m\d\W(s)\bigg\|_\H^2\bigg]\\&\leq C\E\bigg[\int_{\tau_m}^{\tau_m+\xi}\|\PP_m\B(s,\Y_m^N(s))\Q_m\|_{\L_2}^2\d s \bigg]\\&\leq C\E\bigg[\int_{\tau_m}^{\tau_m+\xi}g(s)(1+\|\Y_m^N(s)\|_\H^2)\d s\bigg] \\&\leq C\Bigg\{\E\bigg[\sup_{0\leq s\leq T\wedge\tau_m^N}(1+\|\Y_m(s)\|_\H^2)^{\beta}\bigg]\Bigg\}^{\frac{1}{\beta}}\Bigg\{\E\bigg[\bigg(\int_{\tau_m}^{\tau_m+\xi}g(s)\d s\bigg)^{\frac{\beta}{\beta-1}}\bigg]\Bigg\}^{\frac{\beta-1}{\beta}}.
\end{align*}Using the fact that $g\in\L^1(0,T;\R_+)$ and the absolute continuity of the Lebesgue integral, one can obtain the existence of an $\varepsilon_1>0$ such that 
\begin{align}\label{i2}
	\sup_{m\in\N}\sup_{0<\xi\leq \delta}\E\big[\|I_2(\tau_m+\xi)-I_2(\tau_m)\|_{\V^*}^2\big]\leq C\varepsilon_1. 
\end{align}
Next, we consider the term $I_3$ and we estimate it using It\^o isometry, Hypothesis  \ref{hypo1} (H.6) and H\"older's inequality as
	\begin{align*}
	&	\E\big[\|I_3(\tau_m+\xi)-I_3(\tau_m)\|_{\V^*}^2\big]\nonumber\\&\leq C	\E\big[\|I_3(\tau_m+\xi)-I_3(\tau_m)\|_{\H}^2\big] \\&= C\E\bigg[\bigg\|\int_{\tau_m}^{\tau_m+\xi}\int_\Z \PP_m\gamma(s,\Y_m(s-),z)\vi{\pi}(\d s,\d z)\bigg\|_\H^2\bigg]\\&= C\E\bigg[\int_{\tau_m}^{\tau_m+\xi}\int_\Z\|\PP_m\gamma(s,\Y_m(s),z)\|_\H^2\lambda(\d z)\d s\bigg]
		\\& \leq C\Bigg\{\E\bigg[\sup_{0\leq s\leq T\wedge\tau_m^N}(1+\|\Y_m(s)\|_\H^2)^{\beta}\bigg]\Bigg\}^{\frac{1}{\beta}}\Bigg\{\E\bigg[\bigg(\int_{\tau_m}^{\tau_m+\xi}h_2(s)\d s\bigg)^{\frac{\beta}{\beta-1}}\bigg]\Bigg\}^{\frac{\beta-1}{\beta}}.
	\end{align*}
Once again using the fact that $h_2\in\L^1(0,T;\R_+)$ and the absolute continuity of the Lebesgue integral, one can obtain the existence of an $\varepsilon_2>0$ such that 
\begin{align}\label{i3}
	\sup_{m\in\N}\sup_{0<\xi\leq \delta}\E\big[\|I_3(\tau_m+\xi)-I_3(\tau_m)\|_{\V^*}^2\big]\leq C\varepsilon_2. 
\end{align}
Combining the  estimates \eqref{i1}-\eqref{i3}, one can conclude that the family $\{\mathscr{L}(\Y_m)\}$ is tight in the space $\mathrm{D}([0,T];\V^*)$ equipped with the Skorokhod topology.
\end{proof} 
Let us set 
\begin{align*}
	\Gamma=\mathrm{D}([0,T];\V^*)\cap \L^\beta(0,T;\H)\times\mathrm{C}([0,T];\U_1)\times \mathcal{M}_{\bar{\N}}([0,T]\times\Z),
\end{align*}
where $\U_1$ is a Hilbert space such that the embedding $\U\subset\U_1$ is Hilbert-Schmidt. In view of  Lemma \ref{lem2} and Proposition \ref{proptightness}, for $\W_m:=\W,\pi_m:=\pi$, $m\in\mathbb{N}$, we find that the family of the laws $\mathscr{L}(\Y_m,\W_m,\pi_m)$ of the random vectors $(\Y_m,\W_m,\pi_m)$ is tight in $\Gamma$. From the Prokhorov's theorem (see Lemma \ref{lemA.4}) and a version of Skorokhod's representation theorem (see Theorem \ref{thrmA.5}), we can construct a new probability space $(\wi{\Omega},\wi{\mathscr{F}},\wi{\P})$ and a subsequence of random vectors $\{(\wi{\Y}_m,\wi{\W}_m,\wi{\pi}_m)\}$ (still denoted by  the same)  and $(\wi{\Y},\wi{\W},\wi{\pi})$ on the space $\Gamma$  such that 
\begin{enumerate}
		\item $\mathscr{L}(\wi{\Y}_m,\wi{\W}_m,\wi{\pi}_m)=\mathscr{L}(\wi{\Y},\wi{\W},\wi{\pi})$, for all $m\in\N$;
		\item $(\wi{\Y}_m,\wi{\W}_m,\wi{\pi}_m) \to (\wi{\Y},\wi{\W},\wi{\pi})$ in $\Gamma$ with probability 1 on probability space  $(\wi{\Omega},\wi{\mathscr{F}},\wi{\P})$ as $m\to\infty$;
	\item $(\wi{\W}(\wi{\omega}),\wi{\pi}(\wi{\omega}))=(\W(\wi{\omega}),\pi(\wi{\omega}))$ for all $\wi{\omega}\in\wi{\Omega}$; 
\end{enumerate}
Using the definition of $\Gamma$,  we have
\begin{align}\label{3.54}
	\|\wi{\Y}_m-\wi{\Y}\|_{\L^\beta(0,T;\H)}+\|\wi{\Y}_m-\wi{\Y}\|_{\mathrm{D}([0,T];\V^*)}\to 0, \ \wi{\P} \text{-a.s.}
\end{align}
Our next goal is to prove that $(\wi{\Y},\wi{\W},\wi{\pi})$ is a solution to the system \eqref{1.1}.
 Let us denote the filteration  by $\{\wi{\mathscr{F}}_t\}_{t\geq 0}$ satisfying the usual conditions and generated by $\{\wi{\Y}_m(s),\wi{\Y}(s),\wi{\W}(s),\wi{\pi}:s\leq t\}$.
 
 Then, $\wi{\W}$ is an $\{\wi{\mathscr{F}}_t\}$-cylindrical Wiener process on $\U$ and $\wi{\pi}$ is a  time homogeneous Poisson random measure.  The equation \eqref{3.17} satisfied by the random vector $(\Y_m,\W_m,\pi_m)=(\Y_m,\W,\pi)$,  and hence it follows that 
 \begin{align}\label{3.55}\nonumber
 	\wi{\Y}_m(t)&=\PP_m\x+\int_{0}^{t}\PP_m\A(s,\wi{\Y}_m(s))\d s+\int_0^t\PP_m\B(s,\wi{\Y}_m(s))\Q_m\d\wi{\W}(s)  	\\&\quad + \int_0^t\int_\Z \PP_m\gamma(s,\wi{\Y}_m(s-),z)\vi{\wi{\pi}}(\d s,\d z).
 \end{align} It also satisfy the energy estimate obtained in Lemma \ref{lem1}, that is, for any $p\geq2,$ we have 
\begin{align}\label{3.56}
	\sup_{m\in\N}\left\{\wi{\E}\bigg[\sup_{0\leq t\leq T}\|\wi{\Y}_m(t)\|_\H^p\bigg]+\wi{\E}\bigg[\int_0^T\|\wi{\Y}_m(t)\|_\V^\beta\d t\bigg]^{\frac{p}{2}}\right\}<\infty.
\end{align}
Using the fact that $\|\cdot\|_{\H}$ and $\|\cdot\|_{\V}$  are lower semicontinuous in $\V^*$, the convergence \eqref{3.54} and Fatou's lemma yield 
\begin{align}\label{3.57} \nonumber
	\wi{\E}\bigg[\sup_{0\leq t\leq T}\|\wi{\Y}(t)\|_\H^p\bigg] &\leq \wi{\E}\bigg[\sup_{0\leq t\leq T}\liminf_{m\to\infty}\|\wi{\Y}_m(t)\|_\H^p\bigg]\leq \wi{\E}\bigg[\liminf_{m\to\infty}\sup_{0\leq t\leq T}\|\wi{\Y}_m(t)\|_\H^p\bigg] \\&\leq 
\liminf_{m\to\infty}	\wi{\E}\bigg[\sup_{0\leq t\leq T}\|\wi{\Y}_m(t)\|_\H^p\bigg]
<\infty.
\end{align}Similarly, from \eqref{3.56}, one can also deduce that  
\begin{align}\label{3.58}
	\wi{\E}\bigg[\int_0^T\|\wi{\Y}(t)\|_\V^\beta\d t\bigg]^{\frac{p}{2}}<\infty.
\end{align}
With the help of \eqref{3.56}, Hypothesis \ref{hypo1} (H.4)-(H.6), one can show that 
\begin{lemma}\label{lem3}
	The following estimates hold:
	\begin{align}\label{3.59}
		\sup_{m\in\N}\wi{\E}\bigg[\int_0^T\|\PP_m\A(t,\wi{\Y}_m(t))\|_{\V^*}^{\frac{\beta}{\beta-1}}\d t\bigg]&<\infty,\\\label{3.60}
		\sup_{m\in\N}\wi{\E}\bigg[\int_0^T\|\PP_m\B(t,\wi{\Y}_m(t))\Q_m\|_{\L_2}^2\d t\bigg]&<\infty, \\ \label{3.61}
		\sup_{m\in\N}\wi{\E}\bigg[\int_0^T\int_\Z \|\PP_m\gamma(t,\wi{\Y}_m(t),z)\|_\H^2\lambda(\d z)\d t\bigg] &<\infty.
	\end{align}
\end{lemma}The above Lemma \ref{lem3} implies    the existence of $\overline{\Y}\in \L^\beta(\wi{\Omega}\times[0,T];\V),\;\wi{\A}\in \L^{\frac{\beta}{\beta-1}}(\wi{\Omega}\times[0,T];\V^*), \;\wi{\B}\in \L^2(\wi{\Omega}\times[0,T];\L_2(\U,\H))$ and $\wi{\gamma}\in \mathfrak{L}_{\lambda,T}^2(\mathcal{P}\otimes\mathcal{I},\d\otimes\wi{\P}\otimes\lambda;\H)$ (an application of Banach-Alaoglu) such that along a subsequence we have the following convergences:
\begin{align}\label{3.62}
	\wi{\Y}_m &\xrightarrow{w} \overline{\Y}, \text{ in } \L^\beta(\wi{\Omega}\times[0,T];\V),\\ \label{3.63}
\PP_m	\A(\cdot,\wi{\Y}_m(\cdot))&\xrightarrow{w} \wi{\A}(\cdot),  \text{ in } \L^{\frac{\beta}{\beta-1}}(\wi{\Omega}\times[0,T];\V^*),\\ \label{3.64}
	\PP_m\B(\cdot,\wi{\Y}_m(\cdot))\Q_m &\xrightarrow{w} \wi{\B}(\cdot), \text{ in } \L^2(\wi{\Omega}\times[0,T];\L_2(\U,\H)),\\\label{3.65}
	\int_0^\cdot\PP_m\B(s,\wi{\Y}_m(s))\Q_m\d\wi{\W}(s) &\xrightarrow{w} 	\int_0^\cdot\wi{\B}(s)\d\wi{\W}(s),\text{ in } \L^\infty(0,T;\L^2(\wi{\Omega};\H)),\\\label{3.66}
	\PP_m\gamma(\cdot,\wi{\Y}_m(\cdot),\cdot )&\xrightarrow{w} \wi{\gamma}(\cdot), \text{ in } \mathfrak{L}_{\lambda,T}^2(\mathcal{P}\otimes\mathcal{I},\d\otimes\wi{\P}\otimes\lambda;\H),\\
	\int_0^{\cdot}\int_\Z \PP_m\gamma(s,\wi{\Y}_m(s),z)\vi{\wi{\pi}}(\d s,\d z)&\xrightarrow{w} 	\int_0^{\cdot}\int_\Z \wi{\gamma}(s,z)\vi{\wi{\pi}}(\d s,\d z),\text{ in } \L^\infty(0,T;\L^2(\wi{\Omega};\H))., 
\end{align} as $m\to\infty$. Let us fix 
\begin{align}\label{3.67}
	\vi{\Y}(t):= \x+\int_0^t\wi{\A}(s)\d s+\int_0^t\wi{\B}(s)\d\wi{\W}(s)+\int_0^t\int_\Z\wi{\gamma}(s,z)\vi{\wi{\pi}}(\d s,\d z).
\end{align}Then, one can verify that 
\begin{align}\label{3.68}
	\wi{\Y}=\overline{\Y}=\vi{\Y},\; \wi{\P}\otimes\d t\text{-a.e.},
\end{align}where the first equality in \eqref{3.68} holds from the uniqueness of the limits. The second equality, we prove in the following way:
Let us consider $\v\in\cup_{n=1}^{\infty}\H_n$ and $\varphi\in\L^\infty(\wi{\Omega}\times[0,T])$ and apply Fubini's theorem to find 
\begin{align*}
	&\wi{\E}\bigg[\int_0^T\langle \overline{\Y}(t),\varphi(t)\v\rangle \d t\bigg]\\&= \lim_{m\to\infty} \wi{\E}\bigg[\int_0^T\langle \Y_m(t),\varphi(t)\v\rangle \d t\bigg]\\& =\lim_{k\to\infty} \wi{\E}\bigg[\int_0^T(\PP_m\x,\varphi(t)\v)\d t+\int_0^T \int_0^t \langle \PP_m\A(s,\Y_m(s)),\vphi(t)\v\rangle \d s\d t\\&\qquad+
	\int_0^T\bigg(\int_0^t \PP_m\B(s,\Y_m(s))\Q_m\d\wi{\W}(s), \vphi(t)\v\bigg)\d t\\&\qquad +\int_0^T\bigg( \int_0^t\int_\Z\PP_m\gamma(s,\Y_m(s-),z)\vi{\wi{\pi}}(\d s,\d z), \vphi(t)\v\bigg) \d t\bigg]
	\\& = 
	\lim_{k\to\infty}\bigg\{ \wi{\E}\bigg[(\PP_m\x,\v)\int_0^T\varphi(t)\d t\bigg] +\wi{\E}\bigg[\int_0^T  \bigg\langle \PP_m\A(s,\Y_m(s)),\int_s^T\vphi(t)\d t\v\bigg\rangle \d s\bigg]\\&\qquad+
	\int_0^T\wi{\E}\bigg[\varphi(t)\bigg( \int_0^t \PP_m\B(s,\Y_m(s))\Q_m\d\wi{\W}(s), \v\bigg)\bigg] \d t\\&\qquad +\int_0^T\wi{\E}\bigg[\vphi(t)\bigg( \int_0^t\int_\Z\PP_m\gamma(s,\Y_m(s-),z)\vi{\wi{\pi}}(\d s,\d z), \v\bigg) \bigg]\d t\bigg\}\\& =
	\wi{\E}\bigg[\int_0^T\bigg\langle \x+\int_0^t\wi{\A}(s)\d s+\int_0^t\wi{\B}(s)\d\W(s)+\int_0^t\int_\Z\wi{\gamma}(s,z)\tilde{\pi}(\d s,\d z) ,\vphi(t)\v\bigg\rangle \d t\bigg]\\& =\wi{\E}\bigg[\int_0^T\big\langle \vi{\Y}(t),\vphi(t)\v\big\rangle\d t\bigg].
\end{align*}

In the sequel, we prove our results in the the newly constructed filtered probability space $(\wi{\Omega},\wi{\mathscr{F}},\{\wi{\mathscr{F}}_t\}_{t\geq 0},\wi{\P})$. Now,  we drop the superscript notation, for example, we  write $\{\wi{\Y}_m\}$ and  $\wi{\Y}$ as $\{\Y_m\}$ and $\Y,$ respectively. Therefore, \eqref{3.54} can be rewritten as	\begin{align*}
	\|\Y_m-\Y\|_{\L^\beta(0,T;\H)}+\|\Y_m-\Y\|_{\mathrm{D}([0,T];\V^*)}\to 0.
\end{align*}
Now, we recall some convergence  results from \cite{MRSSTZ}.
\begin{lemma}[Lemma 2.14, \cite{MRSSTZ}]\label{lem4}
	$\wi{\B}(\cdot)=\B(\cdot,\Y(\cdot)),\;\P\otimes \d t$-a.e.
\end{lemma}
\begin{lemma}[Lemma 2.15, \cite{MRSSTZ}]\label{lem5}
	Assume that the Hypothesis \ref{hypo1} (H.1) and (H.2)$'$ hold, the embedding $\V\subset\H$ is compact. Then $\A(t,\cdot)$ is pseudo-monotone from $\V\to\V^*$ for any $t\in[0,T]$.
\end{lemma}
\begin{lemma}[Lemma 2.16, \cite{MRSSTZ}]\label{lem6}
If
 \begin{align}\nonumber
	\Y_m&\xrightarrow{w} \Y, \ \text{ in } \ \L^\beta(\Omega\times[0,T];\V),\\ \label{3.69}
	\A(\cdot,\Y_m(\cdot)) &\xrightarrow{w} \wi{\A}(\cdot),\  \text{ in }\ \L^{\frac{\beta}{\beta-1}}(\Omega\times[0,T];\V^*),\\ \label{3.70}
	\liminf_{m\to\infty}\E\bigg[\int_0^T\langle \A(t,\Y_m(t)),\Y_m(t)\rangle\d t\bigg] &\geq \E\bigg[\int_0^T\langle\wi{ \A}(t),\Y(t)\rangle\d t\bigg], 
\end{align}then $\wi{\A}(\cdot)=\A(\cdot,\Y(\cdot)),\; \P\otimes \d t$-a.e.
\end{lemma}
\begin{lemma}\label{lemma}
	$\wi{\gamma}(\cdot) =\gamma(\cdot,\Y(\cdot),\cdot),\;\P\otimes \d t\otimes \lambda$-a.e.
\end{lemma}
\begin{proof}
	We know that $\|\Y_m-\Y\|_{\L^\beta(0,T;\H)}\to0,\ \P$-a.s. Using  \eqref{3.56}-\eqref{3.57},  and Vitali's convergence theorem, we find 
\begin{align}\label{2.062}
	\lim_{m\to\infty}\E\bigg[\int_0^T\|\Y_m(t)-\Y(t)\|_\H^\rho\d t\bigg]=0,  \ \text{ for all } \ \rho \in[1,\beta], 
\end{align}
since for all $1<p<\infty$
\begin{align*}
	\E\left[\left(\int_0^T\|\Y_m(t)\|_{\H}^{\beta}\d t\right)^p\right]\leq T^p\E\left[\sup_{t\in[0,T]}\|\Y_m(t)\|_{\H}^{p\beta}\right]<\infty. 
\end{align*}
Therefore, along a subsequence $\{\Y_m\}$ (still denoting by the same index), we have the following convergence:
\begin{align}\label{2.0062}
	\lim_{m\to\infty} \|\Y_m(t,\omega)-\Y(t,\omega)\|_\H=0,\text{ a.e. } (t,\omega).
\end{align}Now, using Hypothesis \ref{hypo1} (H.6),  \eqref{3.56} and \eqref{3.57}, we obtain 
\begin{align}\label{GC}
	\lim_{m\to\infty} \E\bigg[\int_0^T\int_\Z\|\PP_m\gamma(t,\Y_m(t),z)-\gamma(t,\Y(t),z)\|_\H^2\lambda(\d z)\d t\bigg]=0.
\end{align}By \eqref{3.67}, and the uniqueness of limit, one can conclude the result. The convergence in \eqref{GC} can be justified as follows:
\begin{align*}
	&\E\bigg[\int_0^T\int_\Z \|\PP_m \gamma(s,\Y_m(s),z)-\gamma(s,\Y(s),z)\|_\H^2\lambda(\d z)\d t\bigg] \\&\leq \E\bigg[\int_0^T\int_\Z \|\PP_m \gamma(s,\Y_m(s),z)-\PP_m\gamma(s,\Y(s),z)\|_\H^2\lambda(\d z)\d t\bigg] \\&\quad +\E\bigg[\int_0^T\int_\Z \|(\I-\PP_m)\gamma(s,\Y(s),z)\|_\H^2\lambda(\d z)\d t \bigg] \\&\leq \E\bigg[\int_0^T\int_\Z \|\PP_m (\gamma(s,\Y_m(s),z)-\gamma(s,\Y(s),z))\|_\H^2\lambda(\d z)\d t\bigg]\\&\quad +\E\bigg[\int_0^T\int_\Z \|(\I-\PP_m)\gamma(s,\Y(s),z)\|_\H^2\lambda(\d z)\d t\bigg] \\& \to 0, \text{ as } m\to \infty,
\end{align*}where we have used the Hypothesis \ref{hypo1} (H.6), \eqref{2.0062} for the first term in the above inequality and for the final term Lebesgue dominated convergence theorem, since as $m\to\infty$ the sequence $\|\I-\PP_m\|_{\mathcal{L}(\H)}\to 0$ as $m\to\infty$.
\end{proof}

A similar result to Theorem \ref{thrm3} has been established in \cite{MRSSTZ} (see Theorem 2.17), where  authors considered the equation perturbed by the multiplicative Gaussian  noise. 
\begin{theorem}\label{thrm3}
	There exists a {\em probabilistically weak solution} to the system \eqref{1.1} which satisfies the energy estimate \eqref{3.13}.
\end{theorem}
\begin{proof}
In this theorem, we establish that the limit $\Y$ of the approximating sequence $\{\Y_m\}$  obtained above is a probabilistically weak solution to the system \eqref{1.1}. In order to prove this, we need  to verify  \eqref{3.70}, with the help of \eqref{3.67}, Lemmas \ref{lem4} and  \ref{lem6}.   We know that the equations \eqref{3.55} and \eqref{3.67} are satisfied by $\Y_m$ and $\Y,$ respectively. Since we have the Gelfand triplet $\V\subset\H\equiv\H^*\subset \V^*$ and the Hypothesis \ref{hypo1} (H.4) ensures that 
\begin{align*}
	&\left(\int_0^T\|\A(t,\Y(t))\|_{\V^*}^{\frac{\beta}{\beta-1}}\d t\right)^{\frac{\beta-1}{\beta}}\\&\leq \sup_{t\in[0,T]}(1+\|\Y(t)\|_{\H}^\alpha)^{\frac{\beta-1}{\beta}}\left(\int_0^T(f(t)+C\|\Y(t)\|_{\V}^\beta)\d t\right)^{\frac{\beta-1}{\beta}}<\infty,\ \mathbb{P}\text{-a.s.},
\end{align*}
for all $\Y$ such that $\sup\limits_{t\in[0,T]}\|\Y(t)\|_{\H}+\left(\int_0^T\|\Y(t)\|_{\V}^{\beta}\d t\right)^{\frac{1}{\beta}}<\infty, \ \mathbb{P}\text{-a.s.}$ Therefore one can apply Theorem 1.2, \cite{IGDS} to obtain that the process $\Y$ has an $\H$-valued modification (still denoted by $\Y$) such that the It\^o formula (stochastic energy equality) is satisfied.

Applying the finite and infinite dimensional It\^o formulae (see Theorem 1.2, \cite{IGDS} and Theorem 1, \cite{IGNV}) to the processes $\Y_m$ and $\Y$, respectively, and then taking expectations on both sides, we get
\begin{align}\label{3.71}\nonumber
	\E\big[\|\Y_m(t)\|_\H^2\big] &= 	 \|\PP_m\x\|_{\H}^2+\E\bigg[\int_{0}^{T}\bigg\{2\langle \A(t,\Y_m(t)),\Y_m(t)\rangle +\|\PP_m\B(t,\Y_m(t))\Q_m\|_{\L_2}^2\\&\qquad
	+ \int_\Z\|\PP_m\gamma(t,\Y_m(t),z)\|_{\H}^2\lambda(\d  z)\bigg\}\d t\bigg],\\ 
		\E\big[\|\Y(t)\|_\H^2\big] &= 	 \|\x\|_{\H}^2+\E\bigg[\int_{0}^{T}\bigg\{2\langle \wi{\A}(t),\Y(t)\rangle +\|\wi{\B}(t)\|_{\L_2}^2
	+\int_\Z\|\wi{\gamma}(t,z)\|_{\H}^2\lambda(\d  z)\bigg\}\d t\bigg]\label{3.72}.
\end{align}
Using the convergence  \eqref{3.54}, the lower semicontinuity of $\|\cdot\|_\H$ in $\V^*$ and Fatou's lemma, we obtain
\begin{align}\label{3.73}
	\E\big[\|\Y(t)\|_\H^2\big] \leq \E\big[\liminf_{m\to\infty}\|\Y_m(t)\|_\H^2\big]\leq \liminf_{m\to\infty}\E\big[\|\Y_m(t)\|_\H^2\big].
\end{align}Using Lemmas \ref{lem4} and \ref{lemma} (see \eqref{3.64} and \eqref{3.66} also), and comparing \eqref{3.71} and \eqref{3.72}, we ensure that \eqref{3.70} holds. Furthermore, the energy estimate  \eqref{3.13}   for $\Y$  follows from \eqref{3.57} and \eqref{3.58}.
\end{proof}

\begin{theorem}\label{thrm4}
	Under the assumption  (H.2) in Hypothesis \ref{hypo1}, the pathwise uniqueness holds for the solutions to the system \eqref{1.1}.
	\end{theorem} 
\begin{proof}
	Let us assume $\Y_1(\cdot)$ and $\Y_2(\cdot)$ be the two solutions to the system \eqref{1.1} defined on the same probability space $(\Omega,\mathscr{F},\{\mathscr{F}_t\}_{t\geq0},\P)$, with the initial data $\Y_1(0)=\x_1$ and $\Y_2(0)=\x_2,$ respectively. Let us define 
	\begin{align}\label{3.74}
		\vphi(t):= \exp\bigg(-\int_0^t\big[f(s)+\rho(\Y_1(s))+\eta(\Y_2(s))\big]\d s\bigg).
	\end{align}Applying It\^o's formula to the process $\vphi(\cdot)\|\Y_1(\cdot)-\Y_2(\cdot)\|_\H^2$, we find
\begin{align}\label{3.75}\nonumber
&	\vphi(t)\|\Y_1(t)-\Y_2(t)\|_\H^2 \\&\nonumber= \|\x_1-\x_2\|_\H^2+\int_0^t\vphi(s)\bigg\{2\langle \A(s,\Y_1(s))-\A(s,\Y_2(s)),\Y_1(s)-\Y_2(s)\rangle \\&\nonumber\qquad +\|\B(s,\Y_1(s))-\B(s,\Y_2(s))\|_{\L_2}^2+ \int_\Z\|\gamma(s,\Y_1(s),z)-\gamma(s,\Y_2(s),z)\|_\H^2\lambda(\d z)\\&\nonumber\qquad-\big[f(s)+\rho(\Y_1(s))+\eta(\Y_2(s))\big]\|\Y_1(s)-\Y_2(s)\|_\H^2\bigg\}\d s\\&\nonumber\quad+2\int_0^t\vphi(s)\big((\B(s,\Y_1(s))-\B(s,\Y_2(s)))\d\W(s),\Y_1(s)-\Y_2(s)\big)\\&\nonumber\quad+ \int_0^t\int_\Z \vphi(s)\big[\|\gamma(s,\Y_1(s-),z)-\gamma(s,\Y_2(s-),z)\|_\H^2\\&\nonumber\qquad+2\big( \gamma(s,\Y_1(s-),z)-\gamma(s,\Y_2(s-),z),\Y_1(s-)-\Y_2(s-)\big)\big]\vi{\pi}(\d s,\d z)
\\&\nonumber \leq   \|\x_1-\x_2\|_\H^2+2\int_0^t\vphi(s)\big((\B(s,\Y_1(s))-\B(s,\Y_2(s)))\d\W(s),\Y_1(s)-\Y_2(s)\big)\\&\nonumber\quad+ \int_0^t\int_\Z\vphi(s)\big[\|\gamma(s,\Y_1(s-),z)-\gamma(s,\Y_2(s-),z)\|_\H^2\\&\qquad+2\big( \gamma(s,\Y_1(s-),z)-\gamma(s,\Y_2(s-),z),\Y_1(s-)-\Y_2(s-)\big)\big]\vi{\pi}(\d s,\d z),
\end{align}where we have used Hypothesis \ref{hypo1} (H.2). Let   $\{\sigma_k\}\uparrow\infty$ be a sequence of stopping times in such a way that the local martingale  appearing in the above inequality \eqref{3.75} is a martingale. Taking expectations on both side of inequality \eqref{3.75}, we get
\begin{align}\label{3.76} 
	\E\big[\vphi(t\wedge \sigma_k)\|\Y_1(t\wedge \sigma_k)-\Y_2(t\wedge \sigma_k)\|_\H^2\big] \leq \|\x_1-\x_2\|_\H^2.
\end{align}Passing $k\to\infty$ and using Fatou's lemma, we find 
\begin{align}\label{3.77}
	\E\big[\vphi(t)\|\Y_1(t)-\Y_2(t)\|_\H^2\big] \leq \|\x_1-\x_2\|_\H^2,
\end{align}where we have used the fact that 
\begin{align}\label{3.78}
	\int_0^T\big[f(s)+\rho(\Y_1(s))+\eta(\Y_2(s))\big]\d s<\infty,\;\P\text{-a.s.}
\end{align}The inequality \eqref{3.77} gives the pathwise uniqueness of solutions to the system \eqref{1.1}. 
\end{proof}

\begin{proof}[Proof of Theorem \ref{thrm1}] We have already proved the existence of a \emph{probabilistically  weak solution} to the system \eqref{1.1} which satisfies the required uniform estimate in Theorem \ref{thrm3} as  well as  the pathwise uniqueness in Theorem \ref{thrm4} under the Hypothesis \ref{hypo1} (H.2). Therefore, combining  Theorems \ref{thrm3}, \ref{thrm4} and an application of the classical Yamada-Watanabe theorem (see Theorem 8, \cite{HZ}) leads to the proof of  Theorem \ref{thrm1}.
	\end{proof}

Let us move to the proof of Theorem \ref{thrm1} which deals with the continuous dependency of the solutions on the initial data.
\begin{proof}[Proof of Theorem \ref{thrm2}]
	Define a sequence of stopping times:
	\begin{align*}
		\sigma_m^N&:=T\wedge\inf\{t\geq0:\|\Y(t,\x_m)\|_{\H}>N\}\wedge\inf\bigg\{t\geq0:\int_0^t\|\Y(s,\x_m)\|_\V^\beta\d s>N\bigg\}\\&\quad\wedge \inf\{t\geq0:\|\Y(t,\x)\|_{\H}>N\}\wedge\inf\bigg\{t\geq0:\int_0^t\|\Y(s,\x)\|_\V^\beta\d s>N\bigg\}\leq T.
	\end{align*} From the estimate \eqref{3.13}, we have
\begin{align}\label{3.79}
	\lim_{N\to\infty}\sup_{m\in\N}\P\big(\sigma_m^N<T\big)=0.
\end{align}From \eqref{3.76} and \eqref{3.77}, we infer that
\begin{align}\label{3.80}
		\E\big[\vphi_m(t\wedge \sigma_m^N)\|\Y(t\wedge \sigma_m^N,\x_m)-\Y(t\wedge \sigma_m^N,\x)\|_\H^2\big] \leq \|\x_m-\x\|_\H^2, 
\end{align}where 
\begin{align}\label{3.81}
	\vphi_m(t) :=\exp\bigg(-\int_0^t\big[ f(s)+\rho(\Y(s,\x_m))+\eta(\Y(s,\x))\big]\d s\bigg). 
\end{align}By Markov's inequality, for any $\e>0$, there is a  positive constant $C_N$ such that 
\begin{align}\label{3.82} \nonumber
	&\P\big(\|\Y(t,\x_m)-\Y(t,\x)\|_\H>\e \big) \\&\nonumber  \leq  \P\big(\|\Y(t,\x_m)-\Y(t,\x)\|_\H>\e,\sigma_m^N=T\big)+\P\big(\sigma_m^N<T\big) \\& \nonumber
	\leq\frac{1}{\e^2C_N}\E\big[\vphi_m(t\wedge \sigma_m^N)\|\Y(t\wedge \sigma_m^N,\x_m)-\Y(t\wedge \sigma_m^N,\x)\|_\H^2\big]+\P\big(\sigma_m^N<T\big) \\& \leq
	\frac{1}{\e^2C_N}\|\x_m-\x\|_\H^2+\sup_{m\in\N}\P\big(\sigma_m^N<T\big).
\end{align}Using \eqref{3.79}, we pass the limit $m\to\infty$ and then $N\to\infty$, to  get that for any $t\in[0,T]$, 
\begin{align}\label{3.83}
	\lim_{m\to\infty}\|\Y(t,\x_m)-\Y(t,\x)\|_\H=0, \ \text{ in  probability } \ \P.
\end{align}Using \eqref{3.13}, for any $p>2,$ we have
\begin{align}\label{3.84}
	\sup_{m\in\N}\E\bigg[\sup_{0\leq t\leq T}\|\Y(t,\x_m)\|_\H^p\bigg]<\infty.
\end{align}Applying Vitali's convergence theorem, we get
\begin{align}\label{3.85}
	\lim_{m\to\infty}\E\bigg[\int_0^T\|\Y(t,\x_m)-\Y(t,\x)\|_\H^2\d t\bigg]=0.
\end{align}In particular, we have 
\begin{align}\label{3.86}
	\|\Y(t,\x_m)-\Y(t,\x)\|_\H \to 0, \text{ as } m\to\infty, \; \P\otimes\d t.
\end{align}Using Hypothesis \ref{hypo1} (H.5), \eqref{3.86} and Vitali's convergence theorem, we find 
\begin{align}\label{3.87}
	\lim_{m\to\infty} \E\bigg[\int_0^T \|\B(t,\Y(t,\x_m))-\B(t,\Y(t,\x))\|_{\L_2}^2\d t\bigg]=0.
\end{align}Once again, using Hypothesis  \ref{hypo1}  (H.6), \eqref{3.86} and Vitali's convergence theorem, we obtain 
\begin{align}\label{3.88}
	\lim_{m\to\infty} \E\bigg[\int_0^T\int_\Z\|\gamma(s,\Y(s,\x_m),z)-\gamma(s,\Y(s,\x),z)\|_\H^2\lambda(\d z)\d s\bigg]=0.
\end{align}Considering  \eqref{3.75} and applying Proposition 2.2., \cite{JZZBWL}, Burkholder-Davis-Gundy inequality, H\"older's and Young's inequalities, we find 
\begin{align}\label{3.89}\nonumber
	&\E\bigg[\sup_{0\leq t\leq T\wedge \sigma_m^N}\vphi_m(t)\|\Y(t,\x_m)-\Y(t,\x)\|_\H^2\bigg] \\&\nonumber\leq \|\x_m-\x\|_\H^2 \\&\nonumber\quad
	+2\E\bigg[\sup_{0\leq t\leq T\wedge \sigma_m^N}\bigg|\int_0^t\vphi_m(s)\big((\B(s,\Y(s,\x_m))-\B(s,\Y(s,\x)))\d\W(s),\Y(s,\x_m)-\Y(s,\x)\big)\bigg|\bigg]\\&\nonumber\quad 
	+\E\bigg[\sup_{0\leq t\leq T\wedge \sigma_m^N}\bigg|\int_0^t\int_\Z \vphi_m(s)\big[\|\gamma(s,\Y(s-,\x_m),z)-\gamma(s,\Y(s-,\x),z)\|_\H^2\\&\nonumber\qquad+2\big(\gamma(s,\Y(s-,\x_m),z)-\gamma(s,\Y(s-,\x),z),\Y(s-,\x_m)-\Y(s-,\x)\big)\big]\vi{\pi}(\d s,\d z)\bigg|\bigg]
	\\& \nonumber
	\leq \|\x_m-\x\|_\H^2 \\&\nonumber\quad
	+C\E\bigg[\bigg(\int_0^{T\wedge \sigma_m^N}\vphi_m^2(t)\|\Y(t,\x_m)-\Y(t,\x)\|_\H^2\|\B(t,\Y(t,\x_m))-\B(t,\Y(t,\x))\|_{\L_2}^2\d t\bigg)^{\frac{1}{2}}\bigg]\\&\nonumber\quad 
	+C\E\bigg[\int_0^{T\wedge \sigma_m^N}\int_\Z\vphi_m(t)\|\gamma(s,\Y(s,\x_m),z)-\gamma(s,\Y(s,\x),z)\|_\H^2\lambda(\d  z)\d s\bigg]
	\\&\nonumber\quad+C\E\bigg[\bigg(\int_0^{T\wedge\sigma_m^N}\int_\Z\vphi_m^2(t)\|\Y(t-,\x_m)-\Y(t-,\x)\|_\H^2\\&\nonumber\qquad\times\|\gamma(t,\Y(t-,\x_m),z)-\gamma(t,\Y(t-,\x),z)\|_\H^2\pi(\d s,\d z)\bigg)^{\frac{1}{2}}\bigg]
		\\& \nonumber
	\leq \|\x_m-\x\|_\H^2 	+\frac{1}{2}\E\bigg[\sup_{0\leq t\leq T\wedge \sigma_m^N}\vphi_m(t)\|\Y(t,\x_m)-\Y(t,\x)\|_\H^2\bigg]\\&\nonumber\quad+
	C\E\bigg[\int_0^{T\wedge \sigma_m^N}\vphi_m(t)\|\B(t,\Y(t,\x_m))-\B(t,\Y(t,\x))\|_{\L_2}^2\d t\bigg]
	\\&\quad+C\E\bigg[\int_0^{T\wedge\sigma_m^N}\int_\Z\vphi_m(t)\|\gamma(t,\Y(t,\x_m),z)-\gamma(t,\Y(t,\x),z)\|_\H^2\lambda(\d  z)\d s\bigg].
\end{align}Using \eqref{3.87} and  \eqref{3.88} in \eqref{3.89}, we conclude that 
\begin{align}\label{3.90}
\lim_{m\to\infty}\E\bigg[\sup_{0\leq t\leq T\wedge \sigma_m^N}\vphi_m(t)\|\Y(t,\x_m)-\Y(t,\x)\|_\H^2\bigg]=0.
\end{align}Again, using calculations similar  to \eqref{3.82}, we find 
\begin{align}\label{3.91}
	\lim_{m\to\infty}\sup_{0\leq t\leq T}\|\Y(t,\x_m)-\Y(t,\x)\|_\H=0, \ \ \text{ in  probability } \ \P.
\end{align}Therefore, from \eqref{3.84} and Vitali's convergence theorem, we find 
\begin{align}\label{3.92}
		\lim_{m\to\infty}\E\bigg[\sup_{0\leq t\leq T}\|\Y(t,\x_m)-\Y(t,\x)\|_\H^p\bigg]=0,
\end{align} for any $p\geq 2$, which completes the proof.
\end{proof}

\section{Part II}\label{sec4}\setcounter{equation}{0}
In this section, we assume the dependency of $\|\B(\cdot,\v)\|_{\L_2}$ and $\|\gamma(\cdot,\v,\cdot)\|_\H$ on $\|\v\|_\V$.  In the case of classical SPDE, this typically means that $\B(\cdot,\v)$ and $\gamma(\cdot,\v,\cdot)$ are allowed to depend  on the gradient $\nabla\v$. For the sequel, we will modify our procedure used in section \ref{sec2}  to prove the well-posedness of system \eqref{1.1} under a new  set of conditions of local monotonicity, which are  modification to our Hypothesis \ref{hypo1} of section \ref{sec2}. Let us  now state the new set of assumptions:

\begin{hypothesis}\label{hypo2}Let $f\in\L^1(0,T;\R_+), \beta\in(1,\infty)$ and $\alpha \in[0,\infty)$:
\begin{enumerate}
	\item[(H.2)$^*$] There exists non-negative constants $\vartheta\in[0,\beta), \zeta,\lambda$ and $C$ such that for  $\u,\v\in\V$ and a.e. $t\in[0,T]$, 
	\begin{align}\label{4.1}\nonumber
	2\langle \A(t,\u)-\A(t,\v),\u-\v\rangle +&\|\B(t,\u)-\B(t,\v)\|_{\L_2}^2+\int_{\Z}\|\gamma(t,\u)-\gamma(t,\v)\|_{\H}^2\lambda(\d z) 		\\& \leq \big[f(t)+\rho(\u)+\eta(\v)\big]\|\u-\v\|_{\H}^2,
		\end{align} where $\rho$ and $\eta$ are two measurable functions from $\V$ to $\R$ such that
	  \begin{align}\label{4.2}
	|\rho(\u)| \leq C(1+\|\u\|_{\H}^\lambda)+C\|\u\|_\V^\vartheta(1+\|\u\|_{\H}^\zeta), \end{align}and \begin{align} \label{4.3}
	|\eta(\u)|\leq C(1+\|\u\|_\H^{2+\alpha})+C\|\u\|_\V^\beta(1+\|\u\|_{\H}^\alpha).
	\end{align}  
\item[(H.3)$^*$] There exists a positive constant $L_\A$ such that for any $\u\in\V$ and a.e. $t\in[0,T]$, 
\begin{align}\label{4.4}
	\langle \A(t,\u),\u\rangle \leq f(t)(1+\|\u\|_\H^2)-L_\A\|\u\|_\V^\beta.
\end{align} 
	\item[(H.4)$^*$] There exists a  non negative constant $C$ such that for any $\u\in\V$ and a.e. $t\in[0,T]$, 
	\begin{align}\label{4.5}
	\|\A(t,\u)\|_{\V^*}^\frac{\beta}{\beta-1} \leq f(t)(1+\|\u\|_\H^{2+\alpha})+C\|\u\|_\V^\beta(1+\|\u\|_\H^\alpha).
	\end{align} 
\item[(H.5)$^*$] There exists $g\in\L^1(0,T;\R_+)$ and a non negative constant $L_\B$ such that for any $\u\in\V$ and a.e. $t\in[0,T]$, 
\begin{align}\label{4.6}
	\|\B(t,\u)\|_{\L_2}^2 \leq g(t)(1+\|\u\|_\H^2)+L_\B\|\u\|_\V^\beta.
\end{align}
\item[(H.6)$^*$]  The jump noise coefficient $\gamma(\cdot,\cdot,\cdot)$ satisfy: \begin{enumerate}
	\item[(1)] The function $\gamma \in\mathfrak{L}^2_{\lambda,T}(\mathcal{P}\otimes\mathcal{I}, \d\otimes\P\otimes\lambda;\H)$.
	\item[(2)](Growth). There exist  a non negative constant $L_\gamma$ and functions $h_p\in \L^1(0,T;\R_+)$ such that for any $\u\in \V$, a.e. $t\in[0,T]$,  and all  $p\in[2,\infty)$, 
	\begin{align}\label{3.11}
		\int_\Z \|\gamma(t,\u,z)\|_{\H}^p\lambda(\d z) \leq h_p(t)(1+\|\u\|_{\H}^p)+L_\gamma\|\u\|_\H^{p-2}\|\u\|_\V^\beta.
	\end{align}
\end{enumerate}
\end{enumerate}
\end{hypothesis}
A similar remark to the following  can be found in Remark 3.1, \cite{MRSSTZ}, where the authors considered the case of multiplicative Gaussian noise. 
\begin{remark}\label{rem5}
	The stronger assumption $\vartheta<\beta$ (than that in (H.2) in (H.2)$^*$) is crucial in the proof of Theorem \ref{thrm5} below. Since, $\rho$ and $\eta$ in \eqref{4.1} are symmetric, therefore one can interchange the role of  $\rho$ and $\eta$ in \eqref{4.2} and \eqref{4.3}. In view of Hypotheses \ref{hypo1} (H.5), (H.6)  and \ref{hypo2} (H.5)$^*$, (H.6)$^*,$ there are no assumption of continuity of $\B$ and $\gamma$ with respect to $\H$-norm and here  $\B$ and $\gamma$ depend on $\V$-norm, which is the main aim  of this section.
\end{remark}
Now, we state the main result  of this section:
\begin{theorem}\label{thrm5}
	Assume that Hypotheses \ref{hypo1} (H.1) and \ref{hypo2} (H.2)$^*$-(H.6)$^*$ hold and the embedding $\V\subset \H$ is compact, with 
	\begin{align}\label{4.7}
L_\B+2C_1L_\gamma< \frac{2L_\A+L_\B}{\chi}
	\end{align}where
\begin{equation}\label{4.8}
	\chi=
\left\{	\begin{aligned}
		&\max\Big\{1+\alpha,1+\lambda,1+\zeta+\frac{2\vartheta }{\beta}\Big\}, &&\text{ if } \beta\leq 2,\\ 
		& \max\Big\{1+\alpha,3+\lambda-\beta,1+\zeta+\frac{2\vartheta }{\beta}\Big\}, &&\text{ if } \beta>2,
	\end{aligned}
\right.
\end{equation}and  the constant $C_1$ is given by 
\begin{equation}\label{410}
	C_1=
	\left\{
	\begin{aligned}
		&1, &&\text{ if } 2\leq p\leq 3,\\
		&2^{p-3}, &&\text{ if } p\geq 3.
	\end{aligned}
	\right.
\end{equation}Then for any initial data $\x\in\H$, there exists a {\em unique  probabilistically strong solution} to the system \eqref{1.1}. Moreover, for any $p$ such that 
\begin{align}\label{4.9}
	2\leq p<1+\frac{2L_\A+L_\B }{L_\B+2C_1L_\gamma}
\end{align}we have the following  estimate:
\begin{align}\label{4.10}
	\E\bigg[\sup_{t\in[0,T]}\|\Y(t)\|_\H^p\bigg]+\E\bigg[\int_0^T\|\Y(t)\|_\V^\beta\d t\bigg]^{\frac{p}{2}} <\infty.
\end{align} 

Furthermore, let $\{\x_m\}$ be a sequence and $\x$ in $\H$  such that $\|\x_m-\x\|_\H\to0$, and $\Y(t,\x)$ be the unique solution of the system \eqref{1.1} with the initial data $\x\in\H$. Then, 
\begin{align}\label{4.11}
	\lim_{m\to\infty} \E\bigg[\sup_{0\leq t\leq T}\|\Y(t,\x_m)-\Y(t,\x)\|_\H^p\bigg]=0, 
\end{align}for $p$ satisfying \eqref{4.9}.
\end{theorem}
The remaining part of this work is devoted to  the proof of Theorem \ref{thrm5}. We will  assume the new set of assumptions of Theorem \ref{thrm5} throughout the remaining section. In section \ref{sec2}, we started our proof from a sequence of Galerkin approximations. Since, in this part we do not have the continuity assumptions on $\B$  and $\gamma$ on $\H$, therefore some parts of the proof do not work, for example the proof of Theorem \ref{thrm3} does not hold. Also, the pseudo-monotonicity argument is not applicable in this part. We will combine the tightness of Galerkin approximations with the monotonicity argument.

Now, we start with the proof of the uniform energy estimates of Galerkin approximations $\{\Y_m\}$ under the modified assumptions:
\begin{lemma}\label{lem7}
	For any $p$ satisfying \eqref{4.9}, there exist  constants $\vi{C}_p$ and $C_2$ satisfying \begin{align}\label{412}
L_\gamma^{p/2}<\frac{L_\A^{p/2}}{(1+\sqrt{3}C_2)C_2^2\vi{C}_p} \text{ and } \
	C_2=\begin{cases*}
		1, \text{ if } 2\leq p\leq 4,\\
		2, \text{ if } p\geq 4,
	\end{cases*}
\end{align}such that 
	\begin{align}\label{4.12}\nonumber
	&	\sup_{m\in\N} \bigg\{\E\bigg[\sup_{0\leq t\leq T}\|\Y_m(t)\|_\H^p\bigg]+\E\bigg[\int_0^T\|\Y_m(t)\|_\V^\beta\|\Y_m(t)\|_\H^{p-2}\d t\bigg]\\&\qquad +\big(1-	C_2^2(1+\sqrt{3})C_pL_\gamma^{\frac{p}{2}}\big)\E\bigg[\int_0^T\|\Y_m(t)\|_\V^\beta\d t\bigg]^{\frac{p}{2}}\bigg\}\nonumber\\& \leq C_p(1+\|\x\|_\H^p).
	\end{align}
\end{lemma}
\begin{proof}Applying It\^o's formula to the process $\|\Y_m(\cdot)\|_\H^p$, we find (cf. \eqref{3.20} and  \eqref{3.21})
\begin{align}\label{4.13}\nonumber
&	\|\Y_m(t)\|_\H^p\\&\nonumber \leq  \|\PP_m\x\|_\H^p+\frac{p}{2}\int_0^t\|\Y_m(s)\|_\H^{p-2}\big[\langle \A(s,\Y_m(s)),\Y_m(s)\rangle +\|\PP_m\B(s,\Y_m(s))\|_{\L_2}^2\big]\d s \\&\nonumber\quad+ \frac{p(p-2)}{2}\int_0^t \|\Y_m(s)\|_\H^{p-4}\|\Y_m(s)\circ\PP_m\B(s,\Y_m(s))\Q_m \|_{\U}^2\d s \\&\nonumber\quad+p\int_0^t\|\Y_m(s)\|_\H^{p-2}\big(\PP_m\B(s,\Y_m(s))\Q_m\d \W(s),\Y_m(s)\big) \\&\nonumber\quad + p\int_0^t\int_\Z\|\Y_m(s-)\|_\H^{p-2}\big(\PP_m\gamma(s,\Y_m(s-),z),\Y_m(s-)\big)\vi{\pi}(\d s,\d z)\\&\quad + \frac{p(p-1)}{2}\int_0^t\int_\Z\|\Y_m(s-)+\theta\PP_m\gamma(s,\Y_m(s-),z)\|_\H^{p-2}\|\PP_m\gamma(s,\Y_m(s-),z)\|_\H^2\pi(\d s,\d z),
\end{align}for some $\theta\in(0,1)$. Taking expectation on both sides, and using Hypothesis  \ref{hypo2} (H.3)$^*$,(H.5)$^*$, and Young's inequality followed  by stopping times arguments, we get
\begin{align}\label{4.14}\nonumber
	&\E\big[	\|\Y_m(t)\|_\H^p\big]+p\bigg(L_\A-\frac{L_\B(p-2)}{2}\bigg)\E\bigg[\int_0^t\|\Y_m(s)\|_\V^{\beta}\|\Y_m(s)\|_\H^{p-2}\d s\bigg] \\& \nonumber\leq 
	\|\x\|_\H^p+C\E\bigg[\int_0^t\big(f(s)+g(s)\big)\d s\bigg]+C\E\bigg[\int_0^t\big(f(s)+g(s)\big)\|\Y_m(s)\|_\H^p\d s\bigg]\\& \quad +
	\frac{p(p-1)}{2}\E\bigg[\int_0^t\int_\Z\|\Y_m(s)+\theta\PP_m\gamma(s,\Y_m(s),z)\|_\H^{p-2}\|\PP_m\gamma(s,\Y_m(s))\|_\H^2\lambda(\d z)\d s\bigg] .
	\end{align}We consider the final term of the right hand side of the above inequality \eqref{4.14} and estimate it using Hypothesis \ref{hypo2} (H.6)$^*$, H\"older's and Young's inequalities as 
\begin{align}\label{4.015}\nonumber
	&\frac{p(p-1)}{2}\E\bigg[\int_0^t\int_\Z\|\Y_m(s)+\theta\PP_m\gamma(s,\Y_m(s),z)\|_\H^{p-2}\|\PP_m\gamma(s,\Y_m(s))\|_\H^2\lambda(\d z)\d s\bigg] \\&\nonumber \leq 
	\frac{p(p-1)C_1}{2}\E\bigg[\int_0^t\int_\Z\big\{\|\Y_m(s)\|_\H^{p-2}\|\gamma(s,\Y_m(s),z)\|_\H^2+\|\gamma(s,\Y_m(s),z)\|_\H^p\big\}\lambda(\d z)\d s\bigg]\\& \nonumber\leq 
	C\E\bigg[\int_0^t\big\{h_2(s)\|\Y_m(s)\|_\H^{p-2}+\big(h_2(s)+h_p(s)\big)\|\Y_m(s)\|_\H^p\big\}\d s\bigg]\\&\nonumber\quad +p(p-1)C_1L_\gamma\E\bigg[\int_0^t\|\Y_m(s)\|_\V^\beta\|\Y_m(s)\|_\H^{p-2}\d s\bigg]+C\int_0^Th_p(t)\d t\\&\nonumber \leq 
	C\E\bigg[\int_0^t\big(h_2(s)+h_p(s)\big)\|\Y_m(s)\|_\H^p\d s\bigg]+p(p-1)C_1L_\gamma\E\bigg[\int_0^t\|\Y_m(s)\|_\V^\beta\|\Y_m(s)\|_\H^{p-2}\d s\bigg]\\&\quad +C\int_0^T\big(h_2(t)+h_p(t)\big)\d t,
\end{align}where the constant $C_1$ is given by
\begin{equation*}
	C_1=
\left\{
\begin{aligned}
&1, &&\text{ if } 2\leq p\leq 3,\\
&2^{p-3}, &&\text{ if } p\geq 3.
\end{aligned}
\right.
\end{equation*}
  Substituting the  estimate \eqref{4.015} in \eqref{4.14}, we obtain
  \begin{align}\label{4.014}\nonumber
  	&\E\big[	\|\Y_m(t)\|_\H^p\big]+p\bigg(L_\A-\frac{L_\B(p-2)}{2}-(p-1)C_1L_\gamma\bigg)\E\bigg[\int_0^t\|\Y_m(s)\|_\V^{\beta}\|\Y_m(s)\|_\H^{p-2}\d s\bigg] \\& \nonumber\leq 
  	\|\x\|_\H^p+C\E\bigg[\int_0^T\big(f(t)+g(t)+h_2(t)+h_p(t)\big)\d s\bigg]\\& \quad +C\E\bigg[\int_0^t\big(f(s)+g(s)+h_2(s)+h_p(s)\big)\|\Y_m(s)\|_\H^p\d s\bigg].
  \end{align}The permissible values of $p$ is given by 
\begin{align}\label{4.15}
	\bigg(L_\A-\frac{L_\B(p-2)}{2}-   (p-1)C_1L_\gamma\bigg)>0.
\end{align}An application of Gronwall's inequality  in  \eqref{4.014} yields 
\begin{align}\label{4.16}\nonumber
	&\sup_{m\in\N}\bigg\{\sup_{0\leq t\leq T}\E\big[\|\Y_m(t)\|_\H^p\big]+\E\bigg[\int_0^T\|\Y_m(s)\|_\V^\beta\|\Y_m(s)\|_\H^{p-2}\d s\bigg]\bigg\} \\&\nonumber\leq \bigg(\|\x\|_\H^p+C\int_0^T\big(f(t)+g(t)+h_2(t)+h_p(t)\big)\d t\bigg)\\&\quad\times\exp\bigg\{C\int_0^T\big(f(t)+g(t)+h_2(t)+h_p(t)\big)\d t\bigg\}.
\end{align}Again, we consider \eqref{4.13}, and use Hypothesis \ref{hypo2} (H.3)$^*$ to  find 
\begin{align}\label{4.131}\nonumber
	&\E\bigg[\sup_{0\leq t\leq T}\|\Y_m(t)\|_\H^p\bigg]+pL_\A\E\bigg[\int_0^T\|\Y_m(s)\|_\V^\beta\|\Y_m(s)\|_\H^{p-2}\d s\bigg]\\& \leq \|\PP_m\x\|_\H^p+ \frac{p(p-2)}{2}\int_0^T \|\Y_m(s)\|_\H^{p-4}\|\Y_m(s)\circ\PP_m\B(s,\Y_m(s))\Q_m\|_{\U}^2\d s \nonumber\\&\nonumber\quad+p\E\bigg[\sup_{0\leq t\leq T}\bigg|\int_0^t\|\Y_m(s)\|_\H^{p-2}\big(\PP_m\B(s,\Y_m(s))\Q_m\d \W(s),\Y_m(s)\big)\bigg| \bigg]\\&\nonumber\quad + p\E\bigg[\sup_{0\leq t\leq T}\bigg|\int_0^t\int_\Z\|\Y_m(s-)\|_\H^{p-2}\big(\PP_m\gamma(s,\Y_m(s-),z),\Y_m(s-)\big)\vi{\pi}(\d s,\d z)\bigg|\bigg]\\&\nonumber\quad + \frac{p(p-1)}{2}
	\E\bigg[\int_0^t\int_\Z\|\Y_m(s-)+\theta\PP_m\gamma(s,\Y_m(s-),z)\|_\H^{p-2}\|\PP_m\gamma(s,\Y_m(s-),z)\|_\H^2\pi(\d s,\d z)\bigg]\\&=: 
	\|\PP_m\x\|_\H^p+\sum_{i=1}^4I_i,
\end{align}for some $\theta\in(0,1)$. We consider the term $I_1$ and estimate it as 
\begin{align*}
	I_1\leq C\E\bigg[\int_0^T \|\Y_m(s)\|_\H^{p-2}\|\B(s,\Y_m(s))\|_{\L_2}^2\d s\bigg].
\end{align*} Now, we consider the  term $I_2$ of the inequality \eqref{4.131}, and estimate it using a similar calculation as in  \eqref{3.23} to find 
\begin{align*}
	I_2 \leq 	\frac{1}{4}\E\bigg[\sup_{0\leq t\leq T}\|\Y_m(t)\|_\H^p\bigg]+C\int_{0}^{T}\|\Y_m(s)\|_{\H}^{p-2}\|\B(s,\Y_m(s))\|_{\L_2}^2\d s\bigg].
\end{align*}Next, we consider the  term $I_3$  and estimate it using the Burholder-Davis-Gundy inequality, Hypothesis \ref{hypo2}, H\"older's and Young's inequalities as
\begin{align*}
	I_3& \leq C \E\bigg[\int_0^T\int_\Z \|\Y_m(s)\|_\H^{2p-2}\|\PP_m\gamma(s,\Y_m(s),z)\|_\H^2\pi(\d s,\d z)\bigg]^{\frac{1}{2}} \\& \leq C \E\bigg[\sup_{0\leq t\leq T}\|\Y_m(t)\|_\H^{\frac{p}{2}}\bigg(\int_0^T\int_\Z \|\Y_m(s)\|_\H^{p-2}\|\PP_m\gamma(s,\Y_m(s),z)\|_\H^2\pi(\d s,\d z)\bigg)^{\frac{1}{2}}\bigg]  \\ &\leq 
	\frac{1}{4}\E\bigg[\sup_{0\leq t\leq T}\|\Y_m(t)\|_\H^p\bigg]+C\E\bigg[\int_0^T\int_\Z \|\Y_m(s)\|_\H^{p-2}\|\PP_m\gamma(s,\Y_m(s),z)\|_\H^2\lambda(\d z)\d s\bigg] \\& \leq \frac{1}{4} \E\bigg[\sup_{0\leq t\leq T}\|\Y_m(t)\|_\H^p\bigg]+C\E\bigg[\int_0^T \big\{h_2(s)\|\Y_m(s)\|_\H^p+L_\gamma\|\Y_m(s)\|^{p-2}\|\Y_m(s)\|_\V^\beta\big\}\d s\bigg] .
\end{align*} A similar calculation to \eqref{4.015} gives 
\begin{align*}
	I_4 &\leq C\E\bigg[\int_0^T\big(h_2(s)+h_p(s)\big)\|\Y_m(s)\|_\H^p\d s\bigg]+CL_\gamma\E\bigg[\int_0^T\|\Y_m(s)\|_\H^{p-2}\|\Y_m(s)\|_\V^\beta\d s\bigg]\\&\quad +C\int_0^T\big(h_2(t)+h_p(t)\big)\d t,
\end{align*}where $C$ is a generic constant that may change from line to line. 
Substituting the above estimates in \eqref{4.131}, we find
\begin{align}\label{4.17}\nonumber
	\E\bigg[\sup_{0\leq t\leq T}\|\Y_m(t)\|_\H^p\bigg] &\leq C \|\x\|_\H^p+C\int_0^T\big(f(t)+g(t)+h_2(t)+h_p(t)\big)\left(1+\|\Y_m(s)\|_\H^p\right)\d t\\& \quad +C(L_\B+L_\gamma)\E\bigg[\int_0^T\|\Y_m(s)\|_\V^\beta\|\Y_m(s)\|_\H^{p-2}\d s\bigg]. 
\end{align}Applying Gronwall's inequality and then using  \eqref{4.16} in\eqref{4.17}, one can   conclude that
\begin{align}\label{4.18}
	\sup_{m\in\N}\E\bigg[\sup_{0\leq t\leq T}\|\Y_m(t)\|_\H^p\bigg] \leq  C_p(1+\|\x\|_\H^p).
\end{align}
Again, from \eqref{3.28}, we have 
\begin{align}\label{4.19}\nonumber
	\E\bigg[\bigg(\int_0^t\|\Y_m(s)\|_\V^\beta\d s\bigg)^{\frac{p}{2}}\bigg]&\nonumber \leq C\|\PP_m\x\|_\H^p+C\E\bigg[\bigg(\int_0^t\big(f(s)+g(s)\big)(1+\|\Y_m(s)\|_\H^2)\d s\bigg)^{\frac{p}{2}}\bigg]\\&\nonumber\quad +C\E\bigg[\bigg|\int_0^t\big(\PP_m\B(s,\Y_m(s))\Q_m\d\W(s),\Y_m(s)\big)\d s\bigg|^{\frac{p}{2}}\bigg]\\&\nonumber\quad+\frac{C_2}{L_\A^{p/2}}\E\bigg[\bigg(\int_{0}^{t}\int_\Z\|\PP_m\gamma(s,\Y_m(s-),z)\|_{\H}^2\pi(\d s,\d z)\bigg)^{\frac{p}{2}}\bigg]\\&\quad +\frac{C_2^2}{L_\A^{p/2}}\E\bigg[\bigg|\int_0^t\int_{\Z}\big(\PP_m\gamma(s,\Y_m(s-),z),\Y_m(s-)\big)\vi{\pi}(\d s,\d z)\bigg|^{\frac{p}{2}}\bigg],
\end{align}where $C_2$ is defined in \eqref{412}. We use  Burkholder-Davis-Gundy inequality, Hypothesis \ref{hypo2} (H.5)$^*$, H\"older's and Young's inequalities to estimate the third  term of the right hand side of the above inequality \eqref{4.19} as 
\begin{align}\label{4.20}\nonumber
&	C\E\bigg[\bigg|\int_0^t\big(\B(s,\Y_m(s))\Q_m\d\W(s),\Y_m(s)\big)\bigg|^{\frac{p}{2}}\bigg]\\&\nonumber\leq 
C\E\bigg[\bigg(\int_0^T\|\Y_m(s)\|_\H^2\big\{g(s)\big(1+\|\Y_m(s)\|_\H^2\big)+L_\B\|\Y_m(s)\|_\V^\beta\big\}\d s\bigg)^{\frac{p}{4}}\bigg]\\&\nonumber\leq 
C\E\bigg[\bigg(\big(1+\sup_{0\leq t\leq T}\|\Y_m(s)\|_\H^4\big)\int_0^Tg(s)\d s\bigg)^{\frac{p}{4}}\bigg]+C\E\bigg[\bigg(\sup_{0\leq t\leq T} \|\Y_m(s)\|_\H^2\int_0^T\|\Y_m(s)\|_\V^{\beta}\d s\bigg)^{\frac{p}{4}}\bigg]
\\& \leq \frac{1}{4}\E\bigg[\int_0^T\|\Y_m(s)\|_\V^\beta\d s\bigg]^{\frac{p}{2}}+C\bigg\{\bigg(\bigg(\int_0^Tg(s)\d s\bigg)^{\frac{p}{4}}+1\bigg)\E\bigg[\sup_{0\leq s\leq T}\|\Y_m(s)\|_\H^p\bigg]\bigg\}.
\end{align}Consider the penultimate term of the right hand side of inequality \eqref{4.19} and estimate it using a similar calculation to \eqref{3.028} and Hypothesis \ref{hypo2} (H.6)$^*$ as
 \begin{align}\label{4.020}\nonumber
 	&\frac{C_2}{L_\A^{p/2}}\E\bigg[\bigg(\int_{0}^{t}\int_\Z\|\PP_m\gamma(s,\Y_m(s-),z)\|_{\H}^2\pi(\d s,\d z)\bigg)^{\frac{p}{2}}\bigg]
 	\\&\nonumber\leq C \E\bigg[\int_{0}^T\int_\Z\|\gamma(s,\Y_m(s),z)\|_{\H}^p\lambda(\d z)\d s\bigg] \\&\nonumber\quad +\frac{C_2\vi{C}_p}{L_\A^{p/2}}\E\bigg[\bigg(\int_{0}^T\int_\Z\|\gamma(s,\Y_m(s),z)\|_{\H}^2\lambda(\d z)\d s\bigg)^{\frac{p}{2}}\bigg] \\&\nonumber 
 	\leq C\bigg(\int_0^Th_p(s)\d s+\E\bigg[\int_0^Th_p(s)\|\Y_m(s)\|_\H^p\d s\bigg]+L_\gamma \E\bigg[\int_0^T\|\Y_m(s)\|_\H^{p-2}\|\Y_m(s)\|_{\V}^\beta\d s\bigg]\bigg)\\&\nonumber\quad
 +\frac{C_2\vi{C}_p}{L_\A^{p/2}}\E\bigg[\bigg(\int_0^T\big\{h_2(s)\big(1+\|\Y_m(s)\|_\H^2\big)+L_\gamma\|\Y_m(s)\|_\V^\beta\big\}\d s\bigg)^{\frac{p}{2}}\bigg] \\&\nonumber
 	\leq C\bigg(\int_0^Th_p(s)\d s+\E\bigg[\int_0^Th_p(s)\|\Y_m(s)\|_\H^p\d s\bigg]+L_\gamma \E\bigg[\int_0^T\|\Y_m(s)\|_\H^{p-2}\|\Y_m(s)\|_{\V}^\beta\d s\bigg]\bigg)\\&\nonumber\quad
 +C\bigg(\bigg(\int_0^Th_2(s)\d s\bigg)^{\frac{p}{2}}+\bigg(\int_0^Th_2(s)\d s\bigg)^{\frac{p-2}{2}}\E\bigg[\int_{0}^{T}h_2(s)\|\Y_m(s)\|_\H^p\d s\bigg]\bigg)\\&\quad +C_2^2\vi{C}_p\bigg(\frac{L_\gamma}{L_\A}\bigg)^{\frac{p}{2}}\E\bigg[\bigg(\int_0^T\|\Y_m(s)\|_\V^\beta\d s\bigg)^{\frac{p}{2}}\bigg].
 \end{align}Using Burkholder-Davis-Gundy inequality, Corollary 2.4., \cite{JZZBWL}, and Hypothesis \ref{hypo2} (H.6)$^*$, we estimate the final term of right hand side of \eqref{4.19} as 
\begin{align}\label{4.21}\nonumber
	&\frac{C_2^2}{L_\A^{p/2}}\E\bigg[\bigg|\int_0^t\int_Z(\PP_m\gamma(s,\Y_m(s-),z),\Y_m(s-))\vi{\pi}(\d s,\d z)\bigg|^{\frac{p}{2}}\bigg]
	\\ &\nonumber\leq 
	 \frac{\sqrt{3}C_2^2}{L_\A^{p/2}}\E\bigg[\bigg(\int_0^{T}\int_\Z \|\gamma(s,\Y_m(s-),z)\|_\H^2\|\Y_m(s-)\|_\H^2\pi(\d s,\d z)\bigg)^{\frac{p}{4}}\bigg] \\&\nonumber \leq C\E\bigg[\sup_{0\leq t\leq T} \|\Y_m(s)\|_\H^p\bigg]+\frac{\sqrt{3}C_2^2}{L_\A^{p/2}}\E\bigg[\bigg(\int_0^T\int_\Z\|\gamma (s,\Y_m(s-),z)\|_\H^2\pi(\d s,\d z)\bigg)^{\frac{p}{2}} \bigg] \\&\nonumber\leq 
	C\E\bigg[\sup_{0\leq t\leq T} \|\Y_m(s)\|_\H^p\bigg]+C\bigg(\int_0^Th_p(s)\d s+\E\bigg[\int_{0}^{T}h_p(s)\|\Y_m(s)\|_\H^p\d s\bigg]\\&\nonumber\qquad +L_\gamma \E\bigg[\int_0^T\|\Y_m(s)\|_\H^{p-2}\|\Y_m(s)\|_{\V}^\beta\d s\bigg]\bigg)+	C\bigg(\bigg(\int_0^Th_2(s)\d s\bigg)^{\frac{p}{2}}+\bigg(\int_0^Th_2(s)\d s\bigg)^{\frac{p-2}{2}}\\&\qquad\times\E\bigg[\int_{0}^{T}h_2(s)\|\Y_m(s)\|_\H^p\d s\bigg]\bigg)+  \sqrt{3}C_2^3\vi{C}_p\bigg(\frac{L_\gamma}{L_\A}\bigg)^{\frac{p}{2}}\E\bigg[\bigg(\int_0^T\|\Y_m(s)\|_\V^\beta\d s\bigg)^{\frac{p}{2}}\bigg].
\end{align}Substituting \eqref{4.20}-\eqref{4.21} in \eqref{4.19}, to obtain 
\begin{align}\label{4.22}\nonumber
\bigg\{1-	\big(1&+\sqrt{3}C_2\big)C_2^2\vi{C}_p\bigg(\frac{L_\gamma}{L_\A}\bigg)^{\frac{p}{2}}\bigg\}\E\bigg[\bigg(\int_0^T\|\Y_m(s)\|_\V^\beta\d s\bigg)^{\frac{p}{2}} \bigg]\\& \leq C\|\x\|_\H^p+C_p\bigg(1+\E\bigg[\sup_{0\leq s\leq T}\|\Y_m(s)\|_\H^p\bigg]\bigg),
\end{align}provided $\frac{L_\A^{p/2}}{(1+\sqrt{3}C_2)C_2^2\vi{C}_p}>L_\gamma^{p/2}$, where the constant $\vi{C}_p$ can be found in Corollary 2.4., \cite{JZZBWL} and the other constant $C_2$ is given  in \eqref{412}. Substituting \eqref{4.18} in \eqref{4.22}, we obtain the required estimate.
\end{proof}
\begin{remark}\label{rem4.5}
We can relax the condition on the constant in \eqref{4.22} with the assumption that $\lambda$ is a finite measure, that is, $\lambda(\Z)<\infty$. Under this assumption, one can estimate the last two terms of \eqref{4.19} in the following way:
\begin{align}\label{FM1}\nonumber
&\frac{C_2}{L_\A^{p/2}}\E\bigg[\bigg(\int_{0}^{t}\int_\Z\|\PP_m\gamma(s,\Y_m(s-),z)\|_{\H}^2\pi(\d s,\d z)\bigg)^{\frac{p}{2}}\bigg]
\\&\nonumber\leq C \E\bigg[\int_{0}^T\int_\Z\|\gamma(s,\Y_m(s),z)\|_{\H}^p\lambda(\d z)\d s\bigg] \\&\nonumber\quad +\frac{C_2\vi{C}_p}{L_\A^{p/2}}\E\bigg[\bigg(\int_{0}^T\int_\Z\|\gamma(s,\Y_m(s),z)\|_{\H}^2\lambda(\d z)\d s\bigg)^{\frac{p}{2}}\bigg]\\&\nonumber\leq \left(C+\frac{\sqrt{3}C_2^2}{L_\A^{p/2}}[T\lambda(\Z)]^{\frac{p-2}{2}}\right) \E\bigg[\int_{0}^T\int_\Z\|\gamma(s,\Y_m(s),z)\|_{\H}^p\lambda(\d z)\d s\bigg] \\&\leq C\E\bigg[\int_0^Th_p\big(1+\|\Y_m(s)\|_\H^p\big)\d s\bigg]+C\E\bigg[\int_0^T\|\Y_m(s)\|_\H^{p-2}\|\Y_m(s)\|_\V^\beta\|\d s\bigg].
\end{align}Next, we consider the final term of the right hand side of \eqref{4.19}, and estimate it using Burkholder-Davis-Gundy inequality, Corollary 2.4., \cite{JZZBWL}, and Hypothesis \ref{hypo2} (H.6)$^*$ as
\begin{align}\label{FM2}\nonumber
	&\frac{C_2^2}{L_\A^{p/2}}\E\bigg[\bigg|\int_0^t\int_Z\big(\PP_m\gamma(s,\Y_m(s-),z),\Y_m(s-)\big)\vi{\pi}(\d s,\d z)\bigg|^{\frac{p}{2}}\bigg]
	\\ &\nonumber\leq 
	\frac{\sqrt{3}C_2^2}{L_\A^{p/2}}\E\bigg[\bigg(\int_0^{T}\int_\Z \|\gamma(s,\Y_m(s-),z)\|_\H^2\|\Y_m(s-)\|_\H^2\pi(\d s,\d z)\bigg)^{\frac{p}{4}}\bigg] \\&\nonumber \leq C\E\bigg[\sup_{0\leq t\leq T} \|\Y_m(s)\|_\H^p\bigg]+\frac{\sqrt{3}C_2^2}{L_\A^{p/2}}\E\bigg[\bigg(\int_0^T\int_\Z\|\gamma (s,\Y_m(s-),z)\|_\H^2\pi(\d s,\d z)\bigg)^{\frac{p}{2}} \bigg]\\&\nonumber  \leq C\E\bigg[\sup_{0\leq t\leq T} \|\Y_m(s)\|_\H^p\bigg]+C\E\bigg[\int_0^Th_p\big(1+\|\Y_m(s)\|_\H^p\big)\d s\bigg]\\&\quad +L_\gamma\E\bigg[\int_0^T\|\Y_m(s)\|_\H^{p-2}\|\Y_m(s)\|_\V^\beta\|\d s\bigg].
\end{align}Substituting \eqref{4.20}, \eqref{FM1} and \eqref{FM2} in \eqref{4.19}, we find 
\begin{align}\label{FM3}
\E\bigg[\bigg(\int_0^T\|\Y_m(s)\|_\V^\beta\d s\bigg)^{\frac{p}{2}} \bigg]\leq C\|\x\|_\H^p+C_p\bigg(1+\E\bigg[\sup_{0\leq s\leq T}\|\Y_m(s)\|_\H^p\bigg]\bigg).
\end{align}
\end{remark}

Using similar calculations as in Lemma \ref{lem2}, one can prove that the family $\{\mathscr{L}(\Y_m):m\in\N\}$ is tight in the space $\L^\beta(0,T;\H)$. Note that the laws of $\{\Y_m\}$ may not be tight in $\D([0,T];\V^*)$. Therefore, by Prokhorov's theorem (see Lemma \ref{lemA.4}) and a version of Skorokhod's representation theorem (see Theorem \ref{thrmA.5}), we construct another probability space  $(\Omega,\mathscr{F},\{\mathscr{F}_t\}_{t\geq 0},\P)$ (still denoting by the same) as in section \ref{sec3}. Also, we may assume that there exists an $\{\mathscr{F}_t\}$-adapted process $\Y$ such that 
\begin{align}\label{4.23}
	\|\Y_m(t)-\Y(t)\|_{\L^\beta(0,T;\H)}\to 0,\; \P\text{-a.s.}
\end{align}Using Hypothesis  \ref{hypo2} (H.4)$^*$-(H.6)$^*$, and Lemma \ref{lem7}, we obtain the following uniform estimates:
\begin{align}\label{4.24}
	\sup_{m\in\N}\E\bigg[\int_{0}^{T}\|\PP_m\A(t,\Y_m(t))\|_{\V^*}^{\frac{\beta}{\beta-1}}\d t\bigg]&<\infty, \\\label{4.25}
		\sup_{m\in\N}\E\bigg[\int_{0}^{T}\|\PP_m\B(t,\Y_m(t))\Q_m\|_{\L_2}^2\d t\bigg]&<\infty,\\\label{4.26}
		\sup_{m\in\N}\E\bigg[\int_0^T\int_\Z \|\PP_m\gamma(t,\Y_m(t),z)\|_\H^2\lambda(\d z)\d t\bigg] &<\infty.
\end{align}Similar to \eqref{2.062} and \eqref{2.0062} in section \ref{sec3}, we have (along a subsequence, still denoted by the same index)
\begin{align}\label{4.27}
	\lim_{m\to\infty}\|\Y_m(t,\omega)-\Y(t,\omega)\|_\H=0,\; \text{a.e. } (t,\omega)\in [0,T]\times\Omega.
\end{align}Again, as in section \ref{sec3}, $\{\Y_m\}$ converges weakly to $\Y$ in $\L^\beta(\Omega\times[0,T];\V)$ and 
\begin{align}\label{4.28}
	\Y(t)=\x+\int_0^t\wi{\A}(s)\d s+\int_0^t\wi{\B}(s)\d\W(s)+\int_0^t\int_\Z\wi{\gamma}(s,z)\vi{\pi}(\d s,\d z),\; \P\otimes\d t\text{-a.s.}
\end{align}where $\wi{\A},\wi{\B}$ and $\wi{\gamma}$ are the limits of the following weak convergence along a subsequence (still denoted by the same index)
\begin{align}\label{4.29}
	\PP_m\A(\cdot,\Y_m(\cdot))&\xrightarrow{w} \wi{\A}(\cdot), \  \text{ in }\  \L^{\frac{\beta}{\beta-1}}(\Omega\times[0,T];\V^*),\\ \label{4.30}
	\PP_m\B(\cdot,\Y_m(\cdot))\Q_m &\xrightarrow{w} \wi{\B}(\cdot),\  \text{ in }\  \L^2(\Omega\times[0,T];\L_2(\U,\H)),\\\label{4.31}
	\PP_m\gamma(\cdot,\Y_m(\cdot),\cdot) )&\xrightarrow{w} \wi{\gamma}(\cdot),\  \text{ in }\  \mathfrak{L}_{\lambda,T}^2(\mathcal{P}\otimes\mathcal{I},\d\otimes\P\otimes\lambda;\H),
\end{align}where we have used Banach-Alaoglu theorem. The next Lemma proves that $\Y$ is a solution of the system \eqref{1.1}.
\begin{lemma}\label{lem8}
	$\wi{\B}(\cdot)=\B(\cdot,\Y(\cdot)),\; \wi{\A}(\cdot)=\A(\cdot,\Y(\cdot))$, $\P\otimes\d t$-a.e., and $\wi{\gamma}(\cdot)=\gamma(\cdot,\Y(\cdot),\cdot),\;\P\otimes \d t\otimes \lambda$-a.e. Consequently, $\Y$ is a {\em probabilistic weak solution} to the system \eqref{1.1}.
\end{lemma}
\begin{proof}
	Fix any time $T>0$ and  let us denote the Lebesgue measure on the interval $[0,T]$ by $\d t$. Assume that $\w$ is  any given $\H$-valued c\'adl\'ag adapted process such that 
	\begin{align}\label{4.32}
		\E\bigg[\sup_{0\leq t\leq T}\|\w(t)\|_\H^{2+\alpha}\bigg]+\E\bigg[\int_0^T\|\w(t)\|_\V^\beta(1+\|\w(t)\|_\H^\alpha)\d t\bigg]\leq C<\infty.
	\end{align}Define a sequence of stopping times: 
\begin{align}\label{4.33}
	\tau_{\w}^N:=T\wedge \inf\{t\geq 0:\|\w(t)\|_\H^2>N\}\wedge\inf\bigg\{t\geq0:\int_0^t\|\w(s)\|_\V^\beta\d s>N\bigg\}\leq T.
\end{align}Then, we have 
\begin{align}\label{4.34}
	\lim_{N\to\infty}\P\big(\tau_{\w}^N<T\big)=0.
\end{align}For any given $\e>0$, 
\begin{align}\label{4.35}\nonumber
	&\P\otimes\d t\big(\big\{(t,\omega):\|\Y_m(t\wedge \tau_{\w}^N(\omega),\omega)-\Y(t\wedge \tau_{\w}^N(\omega),\omega)\|_\H>\e\big\}\big) \\&\nonumber \leq	\P\otimes\d t\big(\big\{(t,\omega):\|\Y_m(t\wedge \tau_{\w}^N(\omega),\omega)-\Y(t\wedge \tau_{\w}^N(\omega),\omega)\|_\H>\e\big\}\cap \big\{(t,\omega):\tau_{\w}^N= T\big\}\big) \\& \nonumber\quad + \P\otimes\d t \big(\big\{(t,\omega):\tau_{\w}^N< T\big\}\big) \\& \leq	\P\otimes\d t\big(\big\{(t,\omega):\|\Y_m(t,\omega)-\Y(t,\omega)\|_\H>\e\big\}\big) + T\P \big(\big\{\omega:\tau_{\w}^N< T\big\}\big).
\end{align}Passing $m\to\infty$ and $N\to\infty$, with help of \eqref{4.27} and \eqref{4.34}, we find
\begin{align}\label{4.36}
	\lim_{m\to\infty,N\to\infty}\|\Y_m(t\wedge\tau_{\w}^N)-\Y(t\wedge\tau_{\w}^N)\|_\H=0, \text{ in } \P\otimes \d t. 
\end{align}Hence, for any $\varphi\in\L^\infty(0,T;\R_+)$,  by Lemma \ref{lem7} and \eqref{4.36}, we get
\begin{align}\label{4.37}\nonumber
&	\lim_{N\to\infty} \E\bigg[\int_0^T\varphi(t)\big[\|\Y(t\wedge\tau_{\w}^N)\|_\H^2-\|\x\|_\H^2\big]\d t\bigg]\\& = \lim_{N\to\infty}\liminf_{m\to\infty}\E\bigg[\int_0^T\varphi(t)\big[\|\Y_m(t\wedge\tau_{\w}^N)\|_\H^2-\|\PP_m\x\|_\H^2\big]\d t\bigg].
\end{align}Applying It\^o's formula to the process $\|\Y_m(\cdot)\|_\H^2$, and then taking expectations on both sides, we get
\begin{align}\label{4.38}\nonumber
&	\E\big[\|\Y_m(t\wedge\tau_{\w}^N)\|_\H^2-\|\PP_m\x\|_\H^2\big] \\&\nonumber\nonumber=\E\bigg[\int_0^{t\wedge\tau_{\w}^N}\bigg\{2\langle \A(s,\Y_m(s)),\Y_m(s)\rangle +\|\PP_m\B(s,\Y_m(s))\Q_m\|_{\L_2}^2\\&\nonumber\qquad+ \int_\Z\|\PP_m\gamma(s,\Y_m(s),z)\|_\H^2\lambda(\d z)\bigg\}\d s\bigg]\\& \nonumber\leq
 \E \bigg[\int_0^{t\wedge\tau_{\w}^N}\bigg\{2\langle \A(s,\Y_m(s))-\A(s,\w(s)),\Y_m(s)-\w(s)\rangle +\|\B(s,\Y_m(s))-\B(s,\w(s))\|_{\L_2}^2\\&\qquad\nonumber+\int_\Z\|\gamma(s,\Y_m(s),z)-\gamma(s,\w(s),z)\|_\H^2\lambda(\d z)\bigg\}\d s\\&\nonumber \quad+\E\bigg[\int_0^{t\wedge\tau_{\w}^N}\bigg\{2\langle \A(s,\Y_m(s)),\w(s)\rangle+2\langle \A(s,\w(s)),\Y_m (s)\rangle -2\langle \A(s,\w(s)),\w(s)\rangle \\&\nonumber\qquad  +2\big(\B(s,\Y_m(s)),\B(s,\w(s))\big)_{\L_2}-\|\B(s,\w(s))\|_{\L_2}^2+\int_\Z\big\{2\big(\gamma(s,\Y_m(s),z),\gamma(s,\w(s),z)\big)\\&\qquad -\|\gamma(s,\w(s),z)\|_\H^2\big\}\lambda(\d z)\bigg\}\d s \bigg].
\end{align}Using Hypothesis \ref{hypo2} (H.2)$^*$, Lemma \ref{lem7}, \eqref{4.29}-\eqref{4.31} and Fubini's theorem, we find 
\begin{align}\label{4.39}\nonumber
&	\liminf_{m\to\infty} \E\bigg[\int_0^T\varphi(t)\big[\|\Y_m(t\wedge\tau_{\w}^N)\|_\H^2-\|\mathrm{P}_m\x\|_\H^2\big]\d t\bigg]  \\&\nonumber\leq 
\liminf_{m\to\infty}\E\bigg[\int_0^T\varphi(t)\int_0^{t\wedge\tau_{\w}^N}\big\{f(s)+\rho(\Y_m(s))+\eta(\w(s))\big\}\|\Y_m(s)-\w(s)\|_\H^2\d s \d t \bigg]\\&\nonumber\quad +\E\bigg[\int_0^T\varphi(t)\bigg\{\int_0^{t\wedge\tau_{\w}^N}\bigg(2\langle \wi{\A}(s),\w(s)\rangle+2\langle \A(s,\w(s)),\Y(s) \rangle -2\langle \A(s,\w(s)),\w(s)\rangle\\&\nonumber\quad+2\big(\wi{\B}(s),\B(s,\w(s))\big)_{\L_2} -\|\B(s,\w(s))\|_{\L_2}^2  +\int_\Z \big\{2\big(\wi{\gamma}(s),\gamma(s,\w(s),z)\big)\\&\quad-\|\gamma(s,\w(s),z)\|_\H^2\big\}\lambda(\d z)\bigg)\d s\bigg\}\d t\bigg].
\end{align}By the condition \eqref{4.32}, as $N\to\infty$,  the limit  of  the second term  on the right hand side of \eqref{4.39} is finite. Applying infinite dimensional It\^o's formula (see Theorem 1.2, \cite{IGDS} and Theorem 1, \cite{IGNV}) to the process $\|\Y(\cdot)\|_\H^2$,  and taking expectations on both sides,  we get
\begin{align}\label{4.40}\nonumber
	&\E\bigg[\int_0^T\varphi(t)\big\{\|\Y(t\wedge\tau_{\w}^N)-\|\x\|_\H^2\big\}\d t\bigg]\\&=\E\bigg[\int_0^T\varphi(t)\bigg\{\int_0^{t\wedge\tau_{\w}^N}\bigg(2\langle \wi{\A}(s),\Y(s)\rangle +\|\wi{\B}(s)\|_{\L_2}^2+\int_\Z\|\wi{\gamma}(s,z)\|_\H^2\lambda(\d z)\bigg)\d s\bigg\}\d t\bigg],
\end{align}where we have  used Fubini's theorem. From \eqref{4.37}, \eqref{4.39} and \eqref{4.40}, we deduce 
\begin{align}\label{4.41}\nonumber
&\lim_{N\to\infty} \Bigg\{\E\bigg[\int_0^T\varphi(t)\bigg\{\int_0^{t\wedge\tau_{\w}^N}\bigg(2\langle \wi{\A}(s)-\A(s,\w(s)),\Y(s)-\w(s)\rangle +\|\wi{\B}(s)-\B(s,\w(s))\|_{\L_2}^2 \\&\nonumber\quad 
+ \int_\Z\|\wi{\gamma}(s,z)-\gamma(s,\w(s),z)\|_\H^2\lambda(\d z)\bigg)\d s\bigg\}\d t  \bigg] \Bigg\}	 \\& \leq 
 \lim_{N\to\infty}\liminf_{m\to\infty} \E\bigg[\int_0^T\varphi(t)\int_0^{t\wedge\tau_{\w}^N}\big\{f(s)+\rho(\Y_m(s))+\eta(\w(s))\big\}\|\Y_m(s)-\w(s)\|_\H^2 \d s\d t\bigg].
\end{align}Using the dominated convergence theorem, we can pass the limit on the left hand side of the inequality \eqref{4.41} to  obtain 
\begin{align}\label{4.42}\nonumber
	&\E\bigg[\int_0^T\varphi(t)\bigg\{\int_0^{t}\bigg(2\langle \wi{\A}(s)-\A(s,\w(s)),\Y(s)-\w(s)\rangle +\|\wi{\B}(s)-\B(s,\w(s))\|_{\L_2}^2 \\&\nonumber\quad + 
\int_\Z\|\wi{\gamma}(s)-\gamma(s,\w(s),z)\|_\H^2\lambda(\d  z)\bigg)\d s\bigg\}\d t  \bigg]
	\\&\leq  C \lim_{N\to\infty}\liminf_{m\to\infty} \E\bigg[\int_0^{T\wedge\tau_{\w}^N}\big\{f(s)+\rho(\Y_m(s))+\eta(\w(s))\big\}\|\Y_m(s)-\w(s)\|_\H^2 \d s\bigg].
\end{align}Now, let us  substitute $\w(\cdot)=\Y(\cdot)$ (since $\Y$ also has the same regularity given in \eqref{4.32}) in the above inequality \eqref{4.42} to get
\begin{align}\label{4.43}\nonumber
	&\E\bigg[\int_0^T\varphi(t)\bigg\{\int_0^{t} \bigg(\|\wi{\B}(s)-\B(s,\Y(s))\|_{\L_2}^2  + 	\int_\Z\|\wi{\gamma}(s,z)-\gamma(s,\Y(s),z)\|_\H^2\lambda(\d  z)\bigg)\d s\bigg\}\d t  \bigg]
	\\&\leq  C \lim_{N\to\infty}\liminf_{m\to\infty} \E\bigg[\int_0^{T\wedge\tau_{\Y}^N}\big\{f(s)+\rho(\Y_m(s))+\eta(\Y(s))\big\}\|\Y_m(s)-\Y(s)\|_\H^2 \d s\bigg].
\end{align}Setting 
\begin{align*}
	I_1&:= \E\bigg[\int_0^{T\wedge\tau_{\Y}^N}f(s)\|\Y_m(s)-\Y(s)\|_\H^2 \d s\bigg],\\ 
	I_2&:=\E\bigg[\int_0^{T\wedge\tau_{\Y}^N}\rho(\Y_m(s))\|\Y_m(s)-\Y(s)\|_\H^2 \d s\bigg],\\
	I_3&:=\E\bigg[\int_0^{T\wedge\tau_{\Y}^N}\eta(\Y(s))\|\Y_m(s)-\Y(s)\|_\H^2 \d s\bigg].
\end{align*}In order to  prove  $\wi{\B}(\cdot)=\B(\cdot,\Y(\cdot)), \; \d t\otimes\P$-a.e., and $\wi{\gamma}(\cdot)=\gamma(\cdot,\Y(\cdot),\cdot)\; \P\otimes \d t\otimes \lambda$-a.e., it is enough to show 
\begin{align}\label{4.44}
	\lim_{N\to\infty}\liminf_{m\to\infty}\big(I_1+I_2+I_3\big)=0. 
\end{align}
Using \eqref{4.27}, Lemma \ref{lem7} and Vitali's convergence theorem, we obtain
\begin{align}\label{4.45}
		\lim_{N\to\infty}\lim_{m\to\infty}I_1\leq \lim_{m\to\infty}\E\bigg[\int_0^{T}f(s)\|\Y_m(s)-\Y(s)\|_\H^2 \d s\bigg]=0.
\end{align}
Using \eqref{4.3}, \eqref{4.27} and the definition of stopping times $\tau_\Y^N$, we obtain 
\begin{align}\label{4.46}
	\lim_{N\to\infty}\lim_{m\to\infty}I_3\leq \lim_{N\to\infty}\E\bigg[C_N\lim_{m\to\infty}\int_0^{T}\|\Y_m(s)-\Y(s)\|_\H^2 \d s\bigg]=0.
\end{align}
Note that the stopping time arguments cannot be directly utilized for $I_2$, as $\rho(\cdot)$ depends on $\{\Y_m\}$. Now, consider the term $I_2$ and using \eqref{4.2}, we find 
\begin{align}\label{4.47}\nonumber
	I_2 &\leq C\E\bigg[\int_0^{T\wedge\tau_{\Y}^N}(1+\|\Y_m(t)\|_\H^\lambda)\|\Y_m(t)-\Y(t)\|_\H^2 \d t\bigg]\\&\nonumber \quad 
	+ C\E\bigg[\int_0^{T\wedge\tau_{\Y}^N}\|\Y_m(t)\|_\V^\vartheta\|\Y_m(t)\|_\H^\zeta\|\Y_m(t)-\Y(t)\|_\H^2 \d t\bigg]\\&\nonumber\quad +
	C\E\bigg[\int_0^{T\wedge\tau_{\Y}^N}\|\Y_m(t)\|_\V^\vartheta\|\Y_m(t)-\Y(t)\|_\H^2 \d t\bigg]\\& 
	=: I_{21}+I_{22}+I_{23}.
\end{align}Choose $p$ such that
\begin{align}\label{4.48}
	1+\chi<p<1+\frac{2L_\A+L_\B }{L_\B+2C_1L_\gamma},
\end{align}where $C_1$ is defined in \eqref{410}. From \eqref{4.8}, we infer that 
\begin{equation}\label{4.49}
	\left\{
	\begin{aligned}
		\lambda+2&<p, \;&& \text{when } \beta \leq 2,\\
		\lambda+2&<\beta+p-2,\; &&\text{when } \beta>2.
	\end{aligned}
\right.
\end{equation}If $\beta\leq 2$, and for $q=\frac{p}{\lambda+2}>1$, then applying Lemma \ref{lem7}, Young's and H\"older's  inequalities, we deduce
\begin{align}\label{4.50}\nonumber
&C\E\bigg[\int_0^T\left\{(1+\|\Y_m(t)\|_\H^\lambda)\|\Y_m(t)-\Y(t)\|_\H^2\right\}^q\d s\bigg] \\&\nonumber\leq C\E\bigg[\int_0^T\left\{\|\Y_m(t)\|_\H^{2q}+\|\Y_m(t)\|_\H^{(\lambda+2)q}+\|\Y(t)\|_\H^{2q}+\|\Y_m(t)\|_\H^{\lambda q}\|\Y(t)\|_\H^{2q}\right\}\d t\bigg]
\\& \nonumber
 \leq C T\bigg\{1+\sup_{m\in\N}\E\bigg[\sup_{0\leq t\leq T}\|\Y_m(t)\|_\H^p\bigg]+\E\bigg[\sup_{0\leq t\leq T}\|\Y(t)\|_\H^p\bigg]\bigg\}\\&<\infty.
\end{align} Using \eqref{4.27}, \eqref{4.50} and Vitali's convergence theorem, we obtain 
\begin{align}\label{4.51}
	\lim_{N\to\infty}\lim_{m\to\infty}I_{21}\leq C\lim_{m\to\infty}\E\bigg[\int_0^T(1+\|\Y_m(t)\|_\H^\lambda)\|\Y_m(t)-\Y(t)\|_\H^2\d t\bigg]=0.
\end{align}Now, if $\beta>2$ and for $q=\frac{\beta+p-2}{\lambda+2}>1$, applying Young's inequality and Lemma \ref{lem7}, we find
\begin{align}\label{4.52}\nonumber
&C	\E\bigg[\int_0^T \big\{(1+\|\Y_m(t)\|_\H^\lambda)\|\Y_m(t)-\Y(t)\|_\H^2\big\}^q\d t \bigg]\\& \nonumber \leq 
C\E\bigg[\int_0^T\left\{\|\Y_m(t)\|_\H^{2q}+\|\Y_m(t)\|_\H^{(\lambda+2)q}+\|\Y(t)\|_\H^{2q}+\|\Y_m(t)\|_\H^{\lambda q}\|\Y(t)\|_\H^{2q}\right\}\d t\bigg]\\& \nonumber\leq 
C\E\bigg[T+\int_0^T\left\{\|\Y_m(t)\|_\H^{\beta+p-2}+\|\Y(t)\|_\H^{\beta+p-2}\right\}\d t\bigg]\\&\nonumber \leq 
C\E\bigg[T+\int_0^T\left\{\|\Y_m(t)\|_\V^\beta\|\Y_m(t)\|_\H^{p-2}+\|\Y(t)\|_\V^\beta\|\Y(t)\|_\H^{p-2}\right\}\d t\bigg]\\&<\infty.
\end{align}Therefore,  \eqref{4.27} and Vitali's convergence theorem imply \eqref{4.51}. For the term $I_{22}$, we divide the proof into three cases depending on the range of the parameter $\zeta$. If 
\begin{align}\label{4.53}
	0<\zeta\leq \frac{\vartheta(p-2)}{\beta}-2,
\end{align}then, using H\"older's inequality and the definition of stopping times $\tau_\Y^N$, we have 
\begin{align}\label{4.54}\nonumber
	&C\E\bigg[\int_0^{T\wedge\tau_{\Y}^N}\|\Y_m(s)\|_\V^\vartheta\|\Y_m(s)\|_\H^\zeta\|\Y_m(s)-\Y(s)\|_\H^2 \d s\bigg] \\& \nonumber\leq C\left\{\E\bigg[\int_0^{T\wedge\tau_{\Y}^N}\left\{\|\Y_m(s)\|_\V^\vartheta\|\Y_m(s)\|_\H^\zeta\|\Y_m(s)-\Y(s)\|_\H^{2-\frac{\beta-\vartheta}{\beta}}\right\}^{\frac{\beta}{\vartheta}}\d s\bigg] \right\}^{\frac{\vartheta}{\beta}}\\&\nonumber\quad\times\left\{\E\bigg[\int_0^{T\wedge\tau_{\Y}^N}\|\Y_m(s)-\Y(s)\|_\H\d s\bigg]\right\}^{\frac{\beta-\vartheta}{\beta}} \\& \nonumber\leq C\left\{\E\bigg[\int_0^{T\wedge\tau_{\Y}^N}\left\{\|\Y_m(s)\|_\V^\beta\|\Y_m(s)\|_\H^{\frac{\zeta\beta+\beta+\vartheta}{\vartheta}}+\|\Y_m(s)\|_\V^\beta\|\Y_m(s)\|_\H^{\frac{\zeta\beta}{\vartheta}}\|\Y(s)\|_\H^{\frac{\beta+\vartheta}{\vartheta}}\right\}\d s\bigg] \right\}^{\frac{\vartheta}{\beta}}\\&\nonumber\quad\times\left\{\E\bigg[\int_0^{T\wedge\tau_{\Y}^N}\|\Y_m(s)-\Y(s)\|_\H\d s\bigg]\right\}^{\frac{\beta-\vartheta}{\beta}}\\&\leq C_N  \left\{\E\bigg[\int_0^{T\wedge\tau_{\Y}^N}\|\Y_m(s)-\Y(s)\|_\H\d s\bigg]\right\}^{\frac{\beta-\vartheta}{\beta}},
\end{align}since $$0<\zeta\leq\frac{\vartheta}{\beta}(p-2)-2\leq  \frac{\vartheta}{\beta}(p-2)-\left(1+\frac{\theta}{\beta}\right).$$ Using Lemma \ref{lem7} and Vitali's convergence theorem, we conclude that 
\begin{align}\label{4.55}
	\lim_{N\to\infty}\lim_{m\to\infty}I_{22} \leq \lim_{N\to\infty}\bigg\{C_N\lim_{m\to\infty}\E\bigg[\int_0^T\|\Y_m(s)-\Y(s)\|_\H\d s	\bigg]\bigg\}^{\frac{\beta-\vartheta}{\beta}}=0.
\end{align}If 
\begin{align}\label{4.56}
\zeta>\frac{\vartheta(p-2)}{\beta},
\end{align}again, by Lemma \ref{lem7}, and the definition of stopping times $\tau_\Y^N$ and H\"older's and Young's inequalities, we obtain
\begin{align}\label{4.57}\nonumber
&C\E\bigg[\int_0^{T\wedge\tau_{\Y}^N}\|\Y_m(s)\|_\V^\vartheta\|\Y_m(s)\|_\H^{\zeta-\frac{\vartheta(p-2)}{\beta}+\frac{\vartheta(p-2)}{\beta}}\|\Y_m(s)-\Y(s)\|_\H^2 \d s\bigg]\\&\nonumber\leq  C\left\{\E\bigg[\int_0^{T\wedge\tau_{\Y}^N}\|\Y_m(s)\|_\V^\beta\|\Y_m(s)\|_\H^{p-2}\d s\bigg]\right\}^{\frac{\vartheta}{\beta}}\\& \nonumber\quad\times
	\bigg\{\E\bigg[\int_0^{T\wedge\tau_{\Y}^N}\left\{\|\Y_m(s)\|_\H^{\zeta-\frac{\vartheta(p-2)}{\beta}}\|\Y_m(s)-\Y(s)\|_\H^2\right\}^{\frac{\beta}{\beta-\vartheta}}\d s\bigg\}^{\frac{\beta-\vartheta}{\beta}} \\&\leq C\bigg\{\E\bigg[\int_0^{T\wedge\tau_{\Y}^N}\left\{\|\Y_m(s)\|_\H^{\zeta-\frac{\vartheta(p-2)}{\beta}}\|\Y_m(s)-\Y(s)\|_\H^2\right\}^{\frac{\beta}{\beta-\vartheta}}\d s\bigg]
	\bigg\}^{\frac{\beta-\vartheta}{\beta}} .
\end{align}
In order to apply Vitali's convergence theorem for the convergence of the final term in the right hand side of \eqref{4.57}, we need the condition $$\left[\zeta-\frac{\vartheta(p-2)}{\beta}\right]\left(\frac{\beta}{\beta-\vartheta}\right)<p\Leftrightarrow \zeta+2+\frac{2\theta}{\beta}<p.$$ 
Using \eqref{4.8} and Vitali's convergence theorem, we find 
\begin{align}\label{4.59}
	\lim_{N\to\infty}\lim_{m\to\infty}I_{22}=0.
\end{align}The case 
\begin{align}\label{4.60}
	\frac{\vartheta(p-2)}{\beta}-2 <\zeta\leq \frac{\vartheta(p-2)}{\beta},
\end{align}can be computed in following way:
\begin{align}\label{4.059}\nonumber
	&C\E\bigg[\int_0^{T\wedge\tau_{\Y}^N}\|\Y_m(s)\|_\V^\vartheta\|\Y_m(s)\|_\H^{\zeta-\frac{\vartheta(p-2)}{\beta}+2+\frac{\vartheta(p-2)}{\beta}-2}\|\Y_m(s)-\Y(s)\|_\H^2 \d s\bigg]  \\&\nonumber \leq C\left\{\E\bigg[\int_0^{T\wedge\tau_{\Y}^N} \|\Y_m(s)\|_\H^{\left\{\left(\zeta-\frac{\vartheta(p-2)}{\beta}+2\right)\frac{\beta}{\beta-\vartheta}\right\}} \d s\bigg]\right\}^{\frac{\beta-\vartheta}{\beta}} \\&
	\nonumber\qquad\times \left\{\E\bigg[\int_0^{T\wedge\tau_{\Y}^N} \left\{\|\Y_m(s)\|_\V^\vartheta\|\Y_m(s)\|_\H^{\frac{\vartheta(p-2)}{\beta}-2}\|\Y_m(s)-\Y(s)\|_\H^2\right\}^{\frac{\beta}{\vartheta}} \d s\bigg]\right\}^{\frac{\vartheta}{\beta}}
\\&\nonumber \leq C\left\{\E\bigg[\int_0^{T\wedge\tau_{\Y}^N} \|\Y_m(s)\|_\H^{\left\{\frac{\zeta\beta-\vartheta p+2\vartheta+2\beta}{\beta-\vartheta}\right\}} \d s\bigg]\right\}^{\frac{\beta-\vartheta}{\beta}} \\&
\qquad\times \left\{\E\bigg[\int_0^{T\wedge\tau_{\Y}^N} \left\{\|\Y_m(s)\|_\V^\vartheta\|\Y_m(s)\|_\H^{\frac{\vartheta(p-2)}{\beta}-2}\|\Y_m(s)-\Y(s)\|_\H^2\right\}^{\frac{\beta}{\vartheta}} \d s\bigg]\right\}^{\frac{\vartheta}{\beta}}.
\end{align}where we have used  H\"older's inequality. In order to get a finite value for the first term in the right hand side of \eqref{4.059}, we need the condition \begin{align}\label{4.0591}\left[\frac{\zeta\beta-\vartheta p+2\vartheta+2\beta}{\beta-\vartheta}\right]\leq p\Leftrightarrow \zeta+2+\frac{2\theta}{\beta}\leq p.
\end{align}
Using \eqref{4.8}, Lemma \ref{lem7}, the definition of stopping times $\tau_\Y^N$ and Vitali's convergence theorem, we find 
\begin{align}\label{4.0059}
		\lim_{N\to\infty}\lim_{m\to\infty}I_{22}=0.
\end{align}Now, we consider the term $I_{23}$, and estimate it using H\"older's inequality, the definition of stopping times $\tau_\Y^N$ and Lemma \ref{lem7}, as 
\begin{align}\label{4.060}\nonumber
& C\E\bigg[\int_0^{T\wedge\tau_{\Y}^N} \|\Y_m(s)\|_\V^\vartheta \|\Y_m(s)-\Y(s)\|_\H^{\frac{2\vartheta^2}{\beta^2}+2-\frac{2\vartheta^2}{\beta^2}}\d s\bigg]\\&\nonumber \leq C\left\{\E\bigg[\int_0^{T\wedge\tau_{\Y}^N}\left\{\|\Y_m(s)\|_\V^\beta\|\Y_m(s)\|_\H^{\frac{2\vartheta}{\beta}}+\|\Y_m(s)\|_\V^\beta\|\Y(s)\|_\H^{\frac{2\vartheta}{\beta}}\right\}\d s\bigg]\right\}^{\frac{\vartheta}{\beta}}\\&\nonumber\qquad\times \left\{\E\bigg[\int_0^{T\wedge\tau_{\Y}^N}\|\Y_m(s)-\Y(s)\|_\H^{\frac{2(\beta+\vartheta)}{\beta}}\d s\bigg]\right\}^{\frac{\beta-\vartheta}{\beta}} \\&\leq C_N \left\{\E\bigg[\int_0^{T\wedge\tau_{\Y}^N}\|\Y_m(s)-\Y(s)\|_\H^{\frac{2(\beta+\vartheta)}{\beta}}\d s\bigg]\right\}^{\frac{\beta-\vartheta}{\beta}}.
\end{align}
In order to apply Vitali's convergence theorem for the convergence of the right hand side term of \eqref{4.060}, we need the condition \begin{align}\label{4.0601}2\left[1+\frac{\vartheta}{\beta}\right]<p.\end{align}
Using \eqref{4.8} and \eqref{4.060}, we obtain
\begin{align}\label{4.61}
		\lim_{N\to\infty}\lim_{m\to\infty}I_{23}=0.
\end{align}Substituting \eqref{4.51}, \eqref{4.55}, \eqref{4.59}, \eqref{4.0059} and \eqref{4.61} in \eqref{4.47} to yield 
\begin{align}\label{4.62}
\lim_{N\to\infty}\lim_{m\to\infty}I_2=0,
\end{align} and  \eqref{4.44} follows. Hence $\wi{\B}(\cdot)=\B(\cdot,\Y(\cdot)), \;\P\otimes \d t$-a.e., and $\wi{\gamma}(\cdot)=\gamma(\cdot,\Y(\cdot),\cdot),\; \P\otimes \d t\otimes \lambda$-a.e.

Next, we choose $\w= \Y-\e\psi \v$,  for $\e>0$, where  $\psi\in\L^\infty(\Omega\times[0,T];\R)$, $\v\in\V$ and substitute it in  \eqref{4.42} to get 
\begin{align*}
&	\E\bigg[\int_0^T \varphi(t)\left\{\int_0^{t}\e\psi(s)\langle \wi{\A}(s)-\A(s,\Y(s)-\e\psi(s)\v(s)), \v(s)\rangle  \d s\right\}\d t\bigg] \\& \leq C\lim_{N\to\infty}\liminf_{m\to\infty}\E\bigg[\int_0^{T\wedge \tau_{\Y}^N}\left\{f(s)+\rho(\Y_m(s))+\eta(\Y(s)-\e \psi(s)\v(s))\right\}\\&\qquad\times\left\{\|\Y_m(s)-\Y(s)\|_\H^2+\e^2\psi^2(s)\|\v(s)\|_\H^2\right\}\d s\bigg].
\end{align*}
 Dividing the above inequality by $\e$ and passing $\e\to0^+$ with the help of Hypothesis \ref{hypo1} (H.1), \eqref{4.2}, \eqref{4.3},  the convergences \eqref{4.45}, \eqref{4.46} and \eqref{4.62}, we obtain
\begin{align}\label{4.63}
	\E\bigg[\int_0^T\varphi(t)\int_0^t\psi(s)\langle \wi{\A}(s)-\A(s,\Y(s)),\v(s)\rangle\d s\d t\bigg]\leq 0.
\end{align}Due to the arbitrary choices of $\v,\psi$ and $\varphi$, we infer that $\wi{\A}(\cdot)=\A(\cdot,\Y(\cdot)),\; \P\otimes\d t$-a.e.
\end{proof}
\begin{proof}[Proof of Theorem \ref{thrm5}] Using \eqref{4.28} and Lemma \ref{lem8}, we obtain the existence of \emph{a probabilistically weak solution} $\Y$  to the system \eqref{1.1}.  The pathwise uniqueness of the solution to the system \eqref{1.1} follows from Theorem \ref{thrm4}. Therefore, combining the above results and an application of the classical Yamada-Watanabe theorem (see Theorem 8, \cite{HZ}) ensure the existence of \emph{a unique probabilistically strong  solution}  to the system \eqref{1.1}. The final part of this theorem, that is, the continuous dependency of the solution on the initial data follows from Theorem \ref{thrm2}.
\end{proof}

\subsection{Applications} The results obtained in this work are applicable to the stochastic versions of hydrodynamic models like (cf. \cite{HBEH,ICAM1,WL4,WLMR1,WLMR2,EM2,CPMR,MRSSTZ}, etc.) Burgers equations, 2D Navier-Stokes equations, 2D magneto-hydrodynamic equations, 2D Boussinesq model for the B\'enard convection, 2D Boussinesq system, 2D magnetic B\'enard equations, 3D Leray-$\alpha$-model, the Ladyzhenskaya model, some shell models of turbulence like GOY, Sabra, dyadic, etc.,  porous media equations,  $p$-Laplacian equations, fast-diffusion equations, power law fluids,  Allen-Cahn equations,    Kuramoto-Sivashinsky equations and 3D tamed Navier-Stokes equations. The paper \cite{CPMR} provided a detailed framework of the models, which comes under the mathematical setting of this paper,  like quasilinear SPDEs, convection-diffusion equations, Cahn-Hilliard equations, 2D Liquid crystal model, 2D Allen-Cahn-Navier-Stokes model, etc., and references therein.

	\begin{appendix}
	\renewcommand{\thesection}{\Alph{section}}
	\numberwithin{equation}{section}
	\section{Some Useful Results}\label{sec5}\setcounter{equation}{0} 
In this section, we state some useful results for the tightness of the laws as well as the well-known Skorokhod's representation  theorem. Let $(\mathbb{Y},\mathcal{B}(\mathbb{Y}))$  be a separable and complete metric space

\begin{definition}
	A family of probability measures $M$ on $(\mathbb{Y},\mathcal{B}(\mathbb{Y}))$ is tight if for any $\e>0$, there exists a compact set $K_{\e}\subset E$ such that $\mu(K_{\e}) \geq 1 -\e$ for all $\mu\in M$. A sequence of measures $\{\mu_n\}$ on $(\mathbb{Y},\mathcal{B}(\mathbb{Y}))$  is weakly convergent to a measure $\mu$  if for all continuous and bounded functions $h$ on $\mathbb{Y}$ $$\lim_{n\to\infty}\int_\mathbb{Y} h(x)\mu_n(x)\d x=\int_\mathbb{Y} h(x)\mu(x)\d x.$$ 
\end{definition}

\begin{lemma}[Prokhorov Theorem, section 5, \cite{PB}]\label{lemA.4} A sequence of probability measures $\{\mu_n\}$  on $(\mathbb{Y},\mathcal{B}(\mathbb{Y}))$  is tight if and only if it is relatively compact, that is,  there exists a subsequence $\{\mu_{n_k}\}$ which converges weakly to a probability measure $\mu$.
\end{lemma}


\begin{lemma}[Theorem 5, \cite{Simon}]\label{lemA.3}
	Let $1\leq p<\infty$. Let $\V,\H$ and $\mathbb{Y}$ be  Banach spaces satisfying the embedding $\V\subset \H\subset \mathbb{Y}$. Suppose that the embedding $\V\subset\H$ is compact. If    $\Gamma$ is a bounded subset of $\L^p(0,T;\V)$ satisfying 
	\begin{align*}
		\lim_{\delta\to0^+} \sup_{f\in\Gamma}\int_0^{T-\delta}\|f(t+\delta)-f(t)\|_{\mathbb{Y}}^p\d t=0,
		\end{align*}then $\Gamma$ is a relatively compact subset of $\L^p(0,T;\H)$.
\end{lemma}Based on the previous lemma, the following result has been established in the work \cite{MRSSTZ} which prove the tightness of the laws in $\L^p(0,T;\H)$.
\begin{lemma}[Lemma 5.2,  \cite{MRSSTZ}]\label{lemRoc}
	Let $1\leq p<\infty$. Let $\V,\H$ and $\mathbb{Y}$ be  Banach spaces satisfying the embedding $\V\subset \H\subset \mathbb{Y}$. Suppose that the embedding $\V\subset \H$ is compact. Let $\{\Y_m\}$ be a sequence of the stochastic processes. If 
	\begin{align}\label{AP1}
		\lim_{N\to\infty}\sup_{m\in\N}\P\bigg[\int_0^T\|\Y_m(t)\|_\V^p\d t>N\bigg]=0,
	\end{align}and for any $\e>0$,
\begin{align}\label{Ap2}
	\lim_{\delta\to0^+}\sup_{m\in\N}\P\bigg[\int_0^{T-\delta}\|\Y_m(t+\delta)-\Y_m(t)\|_{\mathbb{Y}}^p\d t>\e\bigg]=0.
\end{align}Then, $\{\Y_m\}$ is tight in $\L^p(0,T;\H)$.
\end{lemma}


Let us now recall the following  version of the Skorokhod's representation theorem.
\begin{theorem}[Theorem A.1, \cite{PNKTRT},  Theorem C.1, \cite{ZBWLJZ}]\label{thrmA.5}
	Let $(\Omega,\mathscr{F},\P)$ be a probability space and $\mathbb{Y}^1,\mathbb{Y}^2$ be two complete separable metric spaces. Let  $\chi_m:\Omega\to \mathbb{Y}^1\times \mathbb{Y}^2,\; m\in\N$,  be a sequence of weakly convergent random variables with laws $\{\rho_m\}$. For $j=1,2$ let $\Pi_j:\mathbb{Y}^1\times\mathbb{Y}^2\to  \mathbb{Y}^j$ denote the natural projection, that is,
	\begin{align*}
		\Pi_j(\chi^1,\chi^2)=\chi^j,\ \text{  for all  }\  (\chi^1,\chi^2)\in \mathbb{Y}^1\times \mathbb{Y}^2.
	\end{align*}Finally, we assume that the random variables $\Pi_1(\chi_m)=\chi_m^1$ on $\mathbb{Y}^1$ have the same law, independent of $m$. 

Then, there exists a family of $\mathbb{Y}^1\times\mathbb{Y}^2$-valued random variables $\{\wi{\chi}_m\}$ on the probability space $(\wi{\Omega},\wi{\mathscr{F}},\wi{\P}):=\big([0,1)\times [0,1), \mathcal{B}([0,1)\times [0,1)), \text{ Lebesgue measure}\big)$ and a random variable $\wi{\chi}_\infty$ on $(\wi{\Omega},\wi{\mathscr{F}},\wi{\P})$ such that the following statements hold:
\begin{enumerate}
	\item The law of $\{\wi{\chi}_m\}$ is $\{\rho_m\}$ for every $m\in\N$,
	\item $\wi{\chi}_m\to \wi{\chi}_\infty$ in $\mathbb{Y}^1\times\mathbb{Y}^2,\;\wi{\P}$-a.s.,
	\item $\wi{\chi}_m^1=\wi{\chi}_\infty^1$ everywhere on $\wi{\Omega}$ for every $m\in\N$.
\end{enumerate}
\end{theorem}

	\medskip\noindent
{\bf Acknowledgments:} The first author would like to thank Ministry of Education, Government of India - MHRD for financial assistance. M. T. Mohan would  like to thank the Department of Science and Technology (DST), India for Innovation in Science Pursuit for Inspired Research (INSPIRE) Faculty Award (IFA17-MA110).  

\medskip\noindent
{\bf Data availability:} 
Data sharing not applicable to this article as no datasets were generated or analysed during the current study.

\medskip\noindent	{\bf Deceleration:} 	The author has no competing interests to declare that are relevant to the content of this article.

\end{appendix}

	\end{document}